\documentclass[12pt]{amsart}
\usepackage{geometry}   
\usepackage[colorlinks,citecolor = red, linkcolor=blue,hyperindex]{hyperref}
\usepackage{euscript,verbatim, mathrsfs}
\usepackage[psamsfonts]{amssymb}
\usepackage[usenames]{color}
\usepackage{bbm}
\usepackage{graphicx}
\usepackage{color}
\usepackage{float}
\usepackage{dsfont}

\usepackage{euscript}
\usepackage{helvet}         
\usepackage{courier}        
\usepackage{type1cm}        
\usepackage{multicol}        
\usepackage[bottom]{footmisc}

\newtheorem{theorem}{Theorem}[section]
\newtheorem*{theorem*}{Theorem B} 
\newtheorem*{conjecture}{Conjecture}
\newtheorem{lemma}[theorem]{Lemma}
\newtheorem{lem-def}[theorem]{Lemma and Definition}

\newtheorem{proposition}[theorem]{Proposition}
\newtheorem{corollary}[theorem]{Corollary}
\newtheorem*{question-A}{Question A}
\newtheorem*{question-A-1}{Question A'}
\newtheorem*{question-B}{Question B}
\newtheorem*{question-C}{Question C}

\newtheorem*{definition*}{Definition}

\newtheorem*{observation*}{Observation}

\newtheorem*{assumption*}{Assumption}
\newtheorem*{ass-1}{Assumption (A1)}
\newtheorem*{ass-2}{Assumption (A2)}
\newtheorem*{question*}{Question}

\theoremstyle{definition}
\newtheorem{definition}{Definition}[section]

\theoremstyle{remark}
\newtheorem{remark}{Remark}[section]
\newtheorem*{remark*}{Remark}

\newcommand{\R}{\mathbb{R}}
\newcommand{\N}{\mathbb{N}}

\newcommand{\A}{\mathbb{A}}

\newcommand{\D}{\mathbb{D}}
\newcommand{\C}{\mathbb{C}}
\newcommand{\E}{\mathbb{E}}
\newcommand{\T}{\mathbb{T}}

\newcommand{\PP}{\mathbb{P}}

\newcommand{\Sph}{\mathbb{S}}

\newcommand{\RR}{\mathcal{R}}

\newcommand{\MM}{\mathcal{M}}

\newcommand{\HH}{\mathcal{H}}
\newcommand{\TT}{\mathcal{T}}

\newcommand{\Conf}{\mathrm{Conf}}

\newcommand{\Var}{\mathrm{Var}}
\newcommand{\supp}{\mathrm{supp}}

\newcommand{\Aut}{\mathrm{Aut}}

\newcommand{\an}{\text{\, and \,}}

\newcommand{\sexp}{\mathcal{S}_{\mathrm{exp}}}

\makeatletter
\newcommand*{\bigs}[1]{{\hbox{$\left#1\vbox to28\p@{}\right.\n@space$}}}
\makeatother

\makeatletter
\newcommand*{\biggs}[1]{{\hbox{$\left#1\vbox to38\p@{}\right.\n@space$}}}
\makeatother

\begin{document}

\title[Patterson-Sullivan Interpolation of Pluriharmonic Functions]{The Patterson-Sullivan Interpolation of Pluriharmonic Functions for Determinantal Point Processes on Complex Hyperbolic Spaces}

\author
{Alexander I. Bufetov}
\address
{Alexander I. BUFETOV: 
Aix-Marseille Universit\'e, Centrale Marseille, CNRS, Institut de Math\'ematiques de Marseille, UMR7373, 39 Rue F. Joliot Curie 13453, Marseille, France;
Steklov Mathematical Institute of RAS, Moscow, Russia}
\email{bufetov@mi.ras.ru, alexander.bufetov@univ-amu.fr}

\author
{Yanqi Qiu}
\address
{Yanqi QIU: Institute of Mathematics and Hua Loo-Keng Key Laboratory of Mathematics, AMSS, Chinese Academy of Sciences, Beijing 100190, China; CNRS, Institut de Math\'ematiques de Toulouse, Universit\'e Paul Sabatier
}
\email{yanqi.qiu@amss.ac.cn, yanqi.qiu@hotmail.com}

\begin{abstract}
The Patterson-Sullivan construction is proved almost surely to recover a Bergman function from its values on a random discrete subset sampled with  the determinantal point process induced by the Bergman kernel on the unit ball $\D_d$ in $\C^d$.  For super-critical weighted Bergman spaces, the interpolation is uniform when the functions range over the unit ball of the weighted Bergman space. As main results,  we  obtain a necessary and sufficient condition for interpolation of  a fixed pluriharmonic function in the complex hyperbolic space of arbitrary dimension (cf. Theorem~\ref{prop-intro-single-bergman} and Theorem~\ref{thm-pl-2side});  optimal simultaneous uniform interpolation for weighted Bergman spaces  (cf. Theorem~\ref{prop-w-disc}, Proposition~\ref{prop-critical-weight} and Theorem~\ref{thm-wb}); strong simultaneous uniform interpolation for weighted harmonic Hardy spaces (cf. Theorem~\ref{thm-intro-poi} and Theorem~\ref{thm-poi});  and establish the impossibility of the uniform simultaneous interpolation for the Bergman space $A^2(\D_d)$ on $\D_d$ (cf. Theorem~\ref{prop-failure-intro} and Theorem~\ref{thm-fail-d}).  
\end{abstract}
\dedicatory{To the memory of Alexander Ivanovich Balabanov (1952 -- 2018)}

\maketitle


\setcounter{equation}{0}

\section{Introduction}\label{sec-one-dim-disk}
Consider  the unit ball $\D_d$ in the $d$-dimensional complex Euclidean space $\C^d$ and  the Bergman space $A^2(\D_d)$ of square Lebesgue integrable holomorphic functions on $\D_d$.  It is proved in \cite{BQS-LP} that almost any realization of the determinantal point process  on $\D_d$ induced by the Bergman kernel is a {\it uniqueness set} for $A^2(\D_d)$, which, however,  is far from being a {\it sampling set} for $A^2(\D_d)$.  This paper is devoted to the {\it explicit interpolation} of holomorphic, pluriharmonic and $\MM$-harmonic functions  on $\D_d$ from their restrictions onto a typical subset of $\D_d$ sampled with respect to the determinantal point process induced by the Bergman  kernel (see \S \ref{sec-M-har} for a more detailed background).

We apply the {\it Patterson-Sullivan construction} (cf.  Patterson \cite{Patterson-acta} and Sullivan \cite{Sullivan-IHES})  in our setting of determinantal point processes and establish 
\begin{itemize}
\item a necessary and sufficient condition for interpolation of  a fixed pluriharmonic function in the complex hyperbolic space of arbitrary dimension (cf. Theorem~\ref{prop-intro-single-bergman} and Theorem~\ref{thm-pl-2side}); 
\item optimal simultaneous uniform interpolation for weighted Bergman spaces  (cf. Theorem~\ref{prop-w-disc}, Proposition~\ref{prop-critical-weight} and Theorem~\ref{thm-wb}): more precisely, we obtain an explicit critical weight such that the  simultaneous uniform interpolation holds for any Bergman space with a super-critical weight but fail  for the Bergman space with the critical weight;

\item the impossibility of the uniform simultaneous linear interpolation for the Bergman space $A^2(\D_d)$ (cf. Theorem \ref{prop-failure-intro} and Theorem \ref{thm-fail-d}). The proof of the impossibility relies on a new identity \eqref{var-R-id} or inequality \eqref{I-d2},   and on  universal lower bounds \eqref{inf-all-R} and \eqref{low-univers-d} on the variance of Bergman kernel-valued linear statistics of our  determinantal point processes.  
\end{itemize}

Our interpolation formulae can be viewed as  {\it discrete version mean-value properties} for square-integrable $\MM$-harmonic or pluriharmonic functions on $\D_d$  using their ``boundary values" which do not exist in the usual non-tangential limit sense (although the ``boundary values" in the distributional sense,  see \eqref{cv-H1} below,  does exist, they however do not seem to be relevant here).

The reconstruction of Bergman functions from their restrictions onto a discrete sampling set has been extensively studied in the well-developed theory of  interpolation and sampling in Bergman spaces  (cf. \cite{Seip-invent, Seip-IS, Seip-small, Bern-O, Borichev-Kellay}). Nonetheless, we are not aware of previous work on the reconstruction of Bergman functions from their restrictions onto a uniqueness set that is not sampling.

Generalization of our formalism in  {\it metric measure spaces} (including in particular real and quaternion hyperbolic spaces and more general hyperbolic spaces)  and more general  random point processes is given in the sequel to this paper.

In the remaining part of  the introduction, we illustrate  our main results in the case of dimension $d=1$.

\subsection{The Bergman kernel and Patterson-Sullivan interpolation}
Let $dA (z)$ be the  normalized Lebesgue measure on the open unit disk $\D \subset \C$. Let $K_\D(z, w)$ denote the Bergman kernel on the disk $\D$ and $\PP_{K_\D}$ the determinantal point process induced by $K_\D$. In the spirit of the Patterson-Sullivan's construction, for a locally finite subset $X\subset \D$ and any $s>1$, any $z\in \D$, we define the {\it Poincar\'e series} as (following Sullivan  \cite{Sullivan-IHES}, we denote the Poincar\'e series by the letter $g$):
\begin{align}\label{def-poincare-s}
g_X(s,z): = \sum_{x\in X} e^{-s d_\D(x,z)} \in [0, \infty],
\end{align}
where  $d_\D(\cdot, \cdot)$ is the Poincar\'e metric on  the disk $\D$ viewed as the Lobachevsky plane, see the precise formula \eqref{def-d-D} for $d_\D(\cdot, \cdot)$. Given a harmonic function $f$ on $\D$, we shall also need the following definition:  
\begin{align}\label{def-g-X-f}
g_X(s,z;f) : = \sum_{k=0}^\infty \sum_{x\in X\atop k\le d_\D(x,z) < k +1} e^{-s d_\D(x,z)}f(x).
\end{align}
The condition on $f$ under which the series \eqref{def-g-X-f} converges for $\PP_{K_\D}$-almost every $X\subset \D$ is obtained in Lemma \ref{lem-sub-expon}. Note that for the constant function $f \equiv 1$, we have the equality 
\[
g_X(s,z) = g_X(s, z; 1).
\]

  Informally, we shall obtain equalities of the following type: 
\begin{align}\label{f-rec}
f(z)= \lim_{s\to 1^{+}} \frac{g_X(s, z; f)}{g_X(s,z)},
\end{align}
for $\PP_{K_\D}$-almost every $X\subset \D$ and  Bergman  functions $f$ in $A^2(\D)$ and indeed beyond, cf. Theorem \ref{prop-intro-single-bergman}.

The right hand side of \eqref{f-rec} does not depend on the values of $f$ on  any finite subset of $X$, it depends only  on the boundary behaviour of the  function $f$ restricted on $X$.    In this sense,  the formula \eqref{f-rec} should be viewed as a discrete mean-value property of $f$ using the ``boundary values" of $f$ which may not exist in the usual non-tangential limit sense.  If $f$ is a harmonic function on $\D$ which is continuous upto the closed unit disk, then it admits a  radial limit on the boundary $\T = \partial \D$ and  the limit equality \eqref{f-rec} for $\PP_{K_\D}$-almost every $X\subset \D$ can be obtained from a statement of the weak convergence of  probability measures.   For a general Bergman function $f\in A^2(\D)$,  the main obstacle for us is,  $f$ need not have a radial limit on the boundary of the unit disk and may have rather complicated behaviour near the boundary. 
Note that,  for a general $f\in A^2(\D)$, we have \begin{align}\label{cv-H1}
\lim_{r\to 1^{-}} \| f_r - f\|_{H^{1/2}(\T)} = 0
\end{align}
where $H^{1/2}(\T)$ is the Sobolev space on $\T$.   However,  the convergence \eqref{cv-H1} is not strong enough to be applied in our situation.

The informal descriptions of our main results are as follows. 
\begin{itemize}
\item[1.] For a harmonic function $f:\D\rightarrow \C$,  Theorem \ref{prop-intro-single-bergman} says that 
\begin{align}\label{L2-iff}
\lim_{s\to 1^{+}} \E_{\PP_{K_\D}} \Big(\Big| \frac{g_X(s,z;f)}{\E_{\PP_{K_\D}} (g_X(s,z) ) }  - f(z)\Big|^2\Big) = 0
\end{align}
holds  for a fixed point $z\in \D$ if and only if $f$ satisfies a {\it tempered growth condition}:
\begin{align}\label{temp-intro}
\lim_{\alpha \to 0^{+}}  \alpha^2 \int_{\D} | f(x)|^2 (1- |x|^2)^{\alpha}  dA(x) = 0.
\end{align}
Moreover,  under the condition \eqref{temp-intro}, for any relatively compact subset $D\subset \D$, we have
\[
\lim_{s\to 1^{+}} \sup_{z\in D}\E_{\PP_{K_\D}} \Big(\Big| \frac{g_X(s,z;f)}{\E_{\PP_{K_\D}} (g_X(s,z) ) }  - f(z)\Big|^2\Big) = 0.
\]
 Note that while all Bergman functions in $A^2(\D)$ satisfy the tempered growth condition, the converse is not true.  The proof of Theorem~\ref{prop-intro-single-bergman} relies on upper and lower  estimates of the variance of $g_X(s,z;f)$ (cf. Proposition~\ref{prop-2-side}). The lower estimate of the variance of  $g_X(s,z;f)$ (cf. Proposition~\ref{prop-pl-2side}) will also play a role in the following results. 
\item[2.]  Theorem~\ref{prop-w-disc} gives {\it optimal simultaneous uniform interpolation} for weighted Bergman spaces: for any relatively compact subset $D\subset \D$, 
\[
\lim_{s\to 1^{+}} \sup_{z\in D}\E_{\PP_{K_\D}}  \Big ( \sup_{f\in \mathcal{B}(  W)} \Big|  \frac{g_X(s, z; f)}{\E_{\PP_{K_\D}} (g_X(s,z) ) } - f(z)  \Big|^2  \Big) = 0,
\]
provided that $W \in L^1(\D, dA)$ is any non-negative function with 
\begin{align}\label{super-W-def}
\lim_{|z|\to 1^{-}} \frac{W(z)}{(1 -|z|^2)^{-1}[\log(\frac{4}{1-|z|^2})]^{-2}} = \infty
\end{align}
and $\mathcal{B}(W)$ is the unit ball of the weighted Bergman space with weight $W$: 
\[
\mathcal{B}(W): = \Big\{f:\D\rightarrow \C\Big|\text{$f$ is holomorphic and $\int_\D |f(x)|^2 W(x) dA(x)<1$}\Big\}. 
\]

A weight $W$ satisfying \eqref{super-W-def} will be called {\it supercritical}. 
Our result in the weighted Bergman spaces is optimal:  while  the simultaneous uniform interpolation  holds for all super-critical weights, it fails (see Proposition~\ref{prop-critical-weight}) for the  {\it critical weight} itself: 
\[
W_{\mathrm{cr}}(z)  =  \frac{1}{(1 - |z|^2) \log^{2} \big(\frac{4}{1 - |z|^2}\big)} \quad \text{for all $z \in \D$}.
\]

The proof of  the  optimality  relies on a lower estimate of the variance of vector-valued linear statistics: there exists a numerical constant $c>0$, such that  for any Hilbert space vector-valued harmonic function $F:\D\rightarrow \HH$ with {\it sub-exponential mean-growth} (see  \eqref{def-sexp-g} for the precise meaning), for  all $s\in (1, 2]$ and all $z\in \D$, we have 
\[
\Var_{\PP_{K_\D}}(g_X(s, z; F)) \ge  c \int_{\D}   e^{-2s d_\D(x,z)}\| F(x)\|^2 \frac{dA(x)}{(1- |x|^2)^2}. 
\]
The following estimate (cf. Lemma~\ref{lem-var-M}) is also used: there exists a constant $c>0$ such that  the reproducing kernel  $K_{W_{\mathrm{cr}}}(x,y)$  of the weighted Bergman space $A^2(\D, W_{\mathrm{cr}})$ satisfies
\begin{align}\label{K-cr-low-es}
\qquad \int_0^{2\pi} K_{W_{\mathrm{cr}}}(x e^{i\theta}, y) |K_\D(xe^{i\theta},y)|^2 \frac{d\theta}{2\pi} \ge \frac{c}{(1 - |xy|^2)^4} \log \Big(\frac{2}{1 - |xy|^2}\Big) 
\end{align}
for all $x,y \in\D$. We note that the precise formula for $K_{W_{\mathrm{cr}}}(x,y)$ is not known and the proof of the lower estimate \eqref{K-cr-low-es} does not follow from the diagonal-asymptotics (cf. Lemma~\ref{lem-critical-w}) of the weighted Bergman kernel $K_{W_{\mathrm{cr}}}$.

\item[3.]  Theorem \ref{thm-intro-poi} gives {\it strong simultaneous uniform interpolation} for weighted harmonic Hardy spaces: 
Let $\mu$ be any Borel probability measure on $\T$ and set
\[
\quad \quad \quad \mathcal{H}^2(\D; \mu) : = \Big\{f: \D\rightarrow \C\Big|  f (z)= P[h\mu]:= \int_\T   \frac{1- |z|^2}{ |1 - \bar{\zeta}z|^2} h(\zeta) d\mu(\zeta), \, h \in L^2(\T,\mu)\Big\}.
\]
Then for $\PP_{K_\D}$-almost every $X$,   simultaneously for  all $ f  \in \mathcal{H}^2(\D; \mu)$, all $s>1$ and all $z\in \D$, the series \eqref{def-g-X-f} converges  and  we have 
\begin{align}\label{as-cv-wH}
f(z) =    \lim_{s\to 1^{+}}  \frac{g_X(s,z;f)}{g_X(s,z)},
\end{align}
where the convergence is uniform  for  $z$ ranges over any compact subset of $\D$ and for $f$ ranges over the unit ball of $\mathcal{H}^2(\D; \mu)$:
\[
\mathcal{H}_{1}^2(\D; \mu):= \{ f = P[h\mu]: \| h\|_{L^2(\mu)} \le 1\}.
\] 

Note that if $\mu$ is not absolutely continuous with respect to the Lebesgue measure on $\T$, then a function in $\HH^2(\D, \mu)$ may have no radial limit  on the boundary $\T = \partial \D$.   The proof of  Theorem \ref{thm-intro-poi} relies  on (i) a concept of {\it sharply tempered growth condition}  (cf. Definition \ref{def-cmvp-sharp} and Lemma \ref{lem-cmvp-cb}) involving the following estimate: there exists a constant $C>0$ such that  for any $\varepsilon> 0$,
\[
\int_{\D}   \int_\T  (1 - |w|^2)^{2\varepsilon}  P(w, \zeta)^2 d\mu(\zeta)   dA(w)  \le \frac{C}{\varepsilon}
\]
 and on  (ii) the proof of the following inequality (cf.  Lemma \ref{lem-s-d} and the inequalities \eqref{lap-es-cb} and \eqref{Prob-epsilon-less}): for any relatively compact subset $D\subset \D$, there exists a constant $C_D>0$ such that for any $\varepsilon \in (0,1)$ and any $s\in (1,2)$, we have 
\begin{align}\label{sup-sup-prob}
\PP_{K_\D}\Big( \sup_{f \in \mathcal{H}_{1}^2(\D; \mu)}\sup_{z\in D}\Big| \frac{g_X(s,z;f)}{\E_{\PP_{K_\D}}(g_X(s,z))} - f(z) \Big| > \varepsilon  \Big)  \le C_D\cdot  \frac{s-1}{\varepsilon^4}.
\end{align}
\item[4.]  In Theorem \ref{prop-failure-intro}, we consider interpolation with general radial weights. Namely, instead of considering the weights $e^{-s d_\D(0, z)}$ used in the classical Patterson-Sullivan theory, we consider all compactly  supported radial weights $\RR: \D\rightarrow \R_{+}$  and set 
\begin{align}\label{def-gXR}
\qquad g^\RR_X(z; f): =  \sum_{x\in X} \RR(\varphi_z(x)) f(x), \quad g^\RR_X(z): =  \sum_{x\in X} \RR(\varphi_z(x))\,\, \text{where}\,\, \varphi_z(x) = \frac{x-z}{1 - \bar{z}x}.
\end{align}
Here we require the compact support assumption on the weights for avoiding the convergences issue, see Remark \ref{rem-cv-issue} below.  We study the uniform simultaneous interpolation  of all functions in $A^2(\D)$  and obtain a universal positive lower bound for the variance of the associated Bergman kernel-valued linear statistics. Therefore, we show that,  in this very general setting, for the determinantal point process $\PP_{K_\D}$,  any uniform linear interpolation  of $A^2(\D)$ is not possible in $L^2$-sense. The proof of the impossibility of the uniform linear interpolation of $A^2(\D)$ relies on a precise formula (cf. Proposition \ref{prop-new-var-f}) for the variance of the following linear statistics: for any $z\in \D$, any bounded compactly supported radial weight $\RR: \D \rightarrow \R_{+}$,
\begin{align}\label{var-R-id}
\begin{split}
 & \E_{\PP_{K_\D}} \left[    \sup_{f\in A^2(\D): \|f \|\le 1}\Big|   g_X^\RR(z,f) -  f(z) \E_{\PP_{K_\D}} (g^\RR_X(z)) \Big|^2\right] 
\\
& = \frac{1}{2 ( 1 - |z|^2)^2} \int_\D \int_\D  \frac{| \RR(x) - \RR(y)|^2}{(1 - |xy|^2)^5}  \cdot I_z(x, y) dA(x) dA(y),
\end{split}
\end{align}
where $I_z(x,y)$ is given by the formula:
\[
I_z(x, y) = 1 + (3+8|z|^2)|xy|^2 + (3|z|^4 + 8 |z|^2) |xy|^4 + |z|^4|xy|^6.
\] 

The proof of the identity \eqref{var-R-id} relies on a reduction formula (see Lemma~\ref{lem-var-RR}) for the variance  of the Bergman kernel-valued linear statistics of $\PP_{K_\D}$: for any bounded  compactly supported radial function $\RR: \D\rightarrow \R_{+}$, we have
\[
\Var_{\PP_{K_\D}}(g_X^\RR(z; F_\D))  = \frac{1}{2} \int_{\D}\int_\D |\RR(\varphi_z(x)) - \RR(\varphi_z(y))|^2   K_{\D}(x,y)  |K_\D(x,y)|^2 dA(x)dA(y),  
\] 
where $F_\D: \D\rightarrow A^2(\D)$ is the Bergman kernel-valued function defined by 
\[
F_\D(w): = K_\D(\cdot, w) \in A^2(\D). 
\]
In particular, for obtaining the above reduction formula,  we will use the following geneneral identity  (see the proof of Lemma~\ref{lem-unif-var}) for any weight $W$ on $\D$ inducing a weighted Bergman space $A^2(\D, W)$: 
\[
\int_\D K_W(w, z) |K_\D(z,w)|^2 dA(w) = K_W(z, z) K_\D(z,z),
\]
where $K_W(\cdot, \cdot)$ is the reproducing kernel of $A^2(\D, W)$.

 The right hand side double integral in \eqref{var-R-id} seems to related to certain Sobolev-type norm of $\RR$.   From \eqref{var-R-id}, we derive the following universal lower bound showing the impossibility of the uniform linear interpolation of $A^2(\D)$: for any $z\in \D$, we have
\begin{align}\label{inf-all-R}
\inf_{\RR}  \E_{\PP_{K_\D}} \left[    \sup_{f\in A^2(\D): \|f \|\le 1}\Big|   \frac{  g_X^\RR(z,f)}{\E_{\PP_{K_\D}} (g^\RR_X(z))}  -  f(z)  \Big|^2\right]  \ge \frac{1}{128},
\end{align}
where $\RR$ ranges over all compactly  supported radial weights on $\D$.   
\end{itemize}

\begin{remark}
The reader may notice that sometimes it is more convenient  for us to work with the normalization $g_X(s, z; f)/ \E_{\PP_{K_\D}}(g_X(s,z))$ than to work with the normalization $g_X(s, z; f)/g_X(s,z)$. 
We will show (cf. Theorem~\ref{thm-intro-poi} applied to the constant function $f \equiv 1$) that, for $\PP_{K_\D}$-almost every configuration $X\in \Conf(\D)$, the following limit equality
\[
\lim_{s\to 1^{+}} \sup_{z\in D}  \Big|\frac{g_X(s, z)}{ \E_{\PP_{K_\D}}(g_X(s,z))} -1\Big| = 0
\]
holds  for all relatively compact subsets $D\subset \D$. Therefore,  for dealing with the $\PP_{K_\D}$-almost sure convergence, we can use equally both normalizations $g_X(s, z; f)/ \E_{\PP_{K_\D}}(g_X(s,z))$ and  $g_X(s, z; f)/g_X(s,z)$.   For instance,  the following almost sure limit equalities are equivalent: 
\[
f(z) =      \lim_{s\to 1^{+}}  \frac{g_X(s,z;f)}{g_X(s,z)} \an  f(z) = \lim_{s\to 1^{+}}  \frac{g_X(s,z;f)}{\E_{\PP_{K_\D}}(g_X(s,z))} . 
\] 
However, for dealing with the $L^2$-convergence, the normalization $g_X(s, z; f)/ \E_{\PP_{K_\D}}(g_X(s,z))$ has its advantage since the following equality 
\[
\E_{\PP_{K_\D}} \Big(\Big| \frac{g_X(s,z;f)}{\E_{\PP_{K_\D}} (g_X(s,z) ) }  - f(z)\Big|^2\Big) = \frac{\Var_{\PP_{K_\D}} (g_X(s,z; f))}{[\E_{\PP_{K_\D}} (g_X(s,z) ) ]^2}
\] 
does not have a counterpart for the normalization $g_X(s, z; f)/g_X(s,z)$.  In fact, it is not clear to us whether the limit equality \eqref{L2-iff} has the following analogue: 
\[
\lim_{s\to 1^{+}} \E_{\PP_{K_\D}} \Big(\Big| \frac{g_X(s,z;f)}{g_X(s,z)}  - f(z)\Big|^2\Big) = 0.
\]
\end{remark}

\begin{remark}\label{rem-cv-issue}
For the classical Bergman space $A^2(\D)$, the simultaneous uniform interpolation of all  $f\in A^2(\D)$ is much more harder, in fact,  for $\PP_{K_\D}$-almost every $X$, we have 
\[
\qquad \textit{the series $g_X(s, z; f)$ does not converge simulataneous for all $f\in A^2(\D)$. }
\]
See the statements in \eqref{near-cr-zero} below and Proposition \ref{prop-sharp} for more details.  
\end{remark}

\begin{remark}
In the case of weighted Bergman space,  it is not clear whether there is a similar inequality as the inequality \eqref{sup-sup-prob} for a general fixed $f\in A^2(\D, W)$.
\end{remark}

\begin{remark}
The choice of the weight $e^{-sd_\D(x,z)}$ (as in the classical Patterson-Sullivan theory) has its own advantage, see the proofs of Theorem \ref{thm-poi} and Theorem \ref{thm-n-sub}. 
\end{remark}

Note that   the standard procedure in sampling theory  of the  reconstruction of Bergman functions from their restrictions on a  sampling set (see Seip \cite[p. 53]{Seip-IS} or Duren \cite[p. 165]{Duren-Bergman}) is not applicable in our setting since it is proved in \cite{BQS-LP} that the  determinantal point process  $\PP_{K_\D}$ almost surely gives rise to a uniqueness but {\it non-sampling set} for $A^2(\D)$.

All the previous reconstruction results for functions on $\D$ have their higher dimensional counterparts.   Note that in particular, in higher dimension, we shall distinguish the $\MM$-harmonicity and pluriharmonicity: pluriharmonicity implies $\MM$-harmonicity but not conversely.

\begin{center}
\begin{tabular}{|p{7.3cm}|p{7.3cm}|}
\hline
\begin{center}\textbf{Main results: dimension $d =1$}\end{center}&\begin{center}\textbf{Main results: dimension $d\ge 2$}\end{center}\\
\hline
Theorem~\ref{prop-intro-single-bergman}: for a fixed {\it harmonic} function on $\D$,  the necessary and sufficient condition for  the Patterson-Sullivan interpolation is  the tempered growth condition, cf. \eqref{T-temp}.  &  Theorem~\ref{thm-pl-2side}:  for a fixed {\it pluriharmonic} function on $\D_d$,  the necessary and sufficient condition for  the Patterson-Sullivan interpolation is  the tempered growth condition, cf. Definition \ref{def-MVP}. \\
\hline
Theorem~\ref{cor-intro-single-bergman}: a fixed {\it harmonic} function on $\D$ with tempered growth satisfies the Patterson-Sullivan interpolation.  & Theorem~\ref{thm-H-L}: a fixed {\it $\MM$-harmonic} function on $\D_d$ with tempered growth satisfies the Patterson-Sullivan interpolation. \\
\hline
Theorem~\ref{prop-w-disc}: optimal simultaneous uniform Patterson-Sullivan interpolation for weighted Bergman spaces on $\D$ with super-critical weights.  & Theorem~\ref{thm-wb}: optimal simultaneous uniform Patterson-Sullivan interpolation for weighted Bergman spaces on $\D_d$ with super-critical weights.\\
\hline
Theorem \ref{thm-intro-poi}: strong simultaneous uniform Patterson-Sullivan interpolation for  weighted harmonic Hardy spaces on $\D$.  & Theorem~\ref{thm-poi}:  strong simultaneous uniform Patterson-Sullivan interpolation for  weighted harmonic Hardy spaces on $\D_d$.\\
\hline
Theorem \ref{prop-failure-intro}:    impossibility of the uniform linear interpolation of $A^2(\D)$. & Theorem \ref{thm-fail-d}:  impossibility of the uniform linear interpolation of $A^2(\D_d)$. \\ 
\hline
\end{tabular}
\end{center}

\subsection{The reconstruction problem}\label{sec-formulation} 
Let $dA (z) =\frac{i}{2\pi} d\bar{z}\wedge dz$ be the  normalized Lebesgue measure on the open unit disk $\D \subset \C$ and  consider the Bergman space $A^2(\D)$ defined by
\[
A^2(\D): = \Big\{f : \D\rightarrow \C\Big| \text{$f$ is holomorphic and $\int_\D | f(z)|^2 dA(z) <\infty$}\Big\}.
\]
The Hilbert space $A^2(\D)$ admits a reproducing kernel given by
\[
K_\D(z, w) = \frac{1}{( 1 - z \bar{w})^2}, \quad z,w\in\D.
\]

Let $\N$ be the set of non-negative integers and let $(\mathfrak{a}_n)_{n\in \N}$ be the sequence of independent complex Gaussian random variables with expectation $0$ and variance $1$. The random series 
\[
\mathfrak{S}_\D(z) = \sum_{n = 0}^\infty \mathfrak{a}_n z^n 
\]
 almost surely has  radius of convergence $1$ and thus defines a holomorphic function on $\D$.  Peres and Vir\'ag \cite{PV-acta} proved  that  the zero set $Z(\mathfrak{S}_\D) \subset \D$ 
of $\mathfrak{S}_\D$  is the {\it determinantal point process} on $\D$ induced by the Bergman kernel $K_{\D}$. More precisely, let $\Conf(\D)$ denote the set of (locally finite) configurations on $\D$ and let $\PP_{K_\D}$ be the {\it  determinantal probability measure} on $\Conf(\D)$ induced by the  kernel $K_\D$. The Peres-Vir\'ag Theorem states that the distribution of $Z(\frak{S}_\D)$ is given by $\PP_{K_\D}$.  The precise definitions of  configurations, point processes and determinantal point processes  are recalled in  \S \ref{sec-pp}.

For $\PP_{K_\D}$-almost every   $X \in \Conf(\D)$, it is proved in \cite{BQS-LP} that any function $f\in A^2(\D)$,  equal to $0$ in restriction 
to $X$, must be the zero function; in other words, $\PP_{K_\D}$-almost every   $X$ is a  uniqueness set for $A^2(\D)$ and consequently all functions $f \in A^2(\D)$ are {\it simultaneously} uniquely determined by their restrictions $f|_X$.  It is thus natural to ask 

\begin{question-A}
How to recover a  Bergman function $f \in A^2(\D)$ from its restriction to a $\PP_{K_\D}$-typical configuration $X\in \Conf(\D)$ ?  And how to recover simultaneously all Bergman functions in $A^2(\D)$ from their restrictions to a $\PP_{K_\D}$-typical configuration $X\in \Conf(\D)$ ?
\end{question-A}

\begin{remark}
For a general uniquenss  set of any reproducing kernel Hilbert space, the  interpolation using the Gram-Schmidt procedure is possible. More precisely,  let $\HH(K)$ be a reproducing kernel Hilbert space with reproducing kernel $K: S\times S\rightarrow \C$ and   let $\{x_k\}_{k=1}^\infty \subset S$ be a  uniqueness set for $\HH(K)$.  Then  the Gram-Schmidt procedure applied to the sequence $\{K(\cdot, x_k)\}_{k=1}^\infty$ yields an orthonormal basis $\{\varphi_n\}_{n=1}^\infty$ of $\HH(K)$. Write 
\[
\varphi_{n} = \sum_{k=1}^n  c_{n,k} K(\cdot, x_k), \quad \text{where $c_{n,k}\in \C$},
\]
then for any $f\in \HH(K)$, we have 
\begin{align}\label{GS-rec}
f = \sum_{n=1}^\infty \varphi_n\cdot \langle f, \varphi_n\rangle_{\HH(K)} = \sum_{n=1}^\infty   \varphi_n\sum_{k=1}^n \overline{c}_{n,k}  f(x_k),
\end{align}
where the convergence is in norm.  See \cite[p. 135]{HKZ-Bergman} for more details.

Note however that  $\varphi_n$ and $c_{n,k}$ are not explicit and moreover depend on the ordering of  points in the uniqueness set, while the order of points is irrelevant to the uniqueness property. In  the case of Bergman space $A^2(\D)$,  the sequence $\{\varphi_{n}\}_{n=1}^\infty$ changes drastically if we remove or add  a finite number of points from a uniqueness set for $A^2(\D)$, while removing or adding a finite number of points from a uniqueness set for $A^2(\D)$ does not violate its uniqueness property. 
\end{remark}

\subsection{Patterson-Sullivan construction for point processes}\label{sec-PS-intro}
We  view $\D$ as the Poincar{\'e} model of the Lobachevsky plane. Recall that the Poincar\'e metric on the Lobachevsky plane $\D$  is given by the following explicit formula:
\begin{align}\label{def-d-D}
d_\D(x, z) := \log  \Big(\frac{ 1 +  |\varphi_x(z)|}{ 1 - |\varphi_x(z)|}\Big) \quad \text{ for $x, z \in \D$, where $\varphi_x(z) = \frac{z-x}{1 - \bar{x}z}$.}
\end{align}
 In the spirit of the Patterson-Sullivan construction (cf. Patterson \cite{Patterson-acta} and Sullivan \cite{Sullivan-IHES}), we consider, for $\PP_{K_\D}$-almost every  $X\in \Conf(\D)$, for any $z\in \D$ and $s>1$,  the probability measure (viewed as measure on the closed unit disk $\overline{\D}$):
\begin{align}\label{def-PS-meas}
\nu_X(s, z): =   \frac{1}{g_X(s, z)} \sum_{x \in X}   e^{-s d_\D(x, z)} \delta_x \quad \text{with} \quad g_X(s, z) := \sum_{x \in X}  e^{-s d_\D(x, z)}.
\end{align}
We can easily show (which is also an immediate consequence of Theorem \ref{thm-intro-poi} below) that the following weak convergence holds for $\PP_{K_\D}$-almost every  configuration $X\in \Conf(\D)$:
\begin{align}\label{harm-meas-cv}
\lim_{s\to 1^{+}} \nu_X(s, z)= P^z \quad \text{for all $z\in \D$,}
\end{align}
where $P^z$ is the harmonic measure on the unit circle $\T = \partial \D$ associated to $z$, that is,  $P^z(d\zeta) = P(z, \zeta)dm(\zeta) $ with $dm$ the normalized Lebesgue measure on $\T$ and  $P(z, \zeta)$ the Poisson kernel: 
\begin{align}\label{poi-ker}
P(z, \zeta)=   \frac{1- |z|^2}{ |1 - \bar{\zeta}z|^2} =  \frac{1- |z|^2}{ |z - \zeta|^2}, \quad z  \in \D \an \zeta \in \T.
\end{align}

\begin{remark}
The exponent $s = 1$ is critical since for  $\PP_{K_\D}$-almost every  $X\in \Conf(\D)$, we have $g_X(s, z)<\infty$ if and only if $s>1$. 
\end{remark}

\begin{remark}
The Lobachevsky metric  on $\D$ coincides with  the Bergman metric on $\D$  defined   using the Bergman kernel, cf. Krantz \cite[Chapter 1]{Krantz}.
\end{remark}

As an immediate consequence of \eqref{harm-meas-cv}, for $\PP_{K_\D}$-almost every  $X\in \Conf(\D)$, {\it simultaneously} for all $z\in \D$ and all  harmonic functions $u: \D\rightarrow \C$ that is continuous upto the closed disk $\overline{\D}$, we can recover $u(z)$ from the values of $u|_X$ as follows:
\begin{align}\label{rec-unif-harm}
u(z)= \lim_{s\to 1^{+}} \int_\D u d\nu_X(s, z)  = \lim_{s\to 1^{+}} \frac{1}{g_X(s, z)} \sum_{x\in X} e^{-s d_\D(x,z)} u(z).
\end{align}
However,  \eqref{rec-unif-harm} does not hold for a general $f\in A^2(\D)$. The main reasons are: 
\begin{itemize}
\item[(1):] A general $f\in A^2(\D)$  is not necessarily uniformly bounded on $\D$ and thus for $\PP_{K_\D}$-almost every $X\in \Conf(\D)$, it is unclear whether the series 
\begin{align}\label{cv-issue}
 \sum\limits_{x \in X} e^{-s d_\D(x, z)} | f(x)|
\end{align}
 is convergent or not when  $s>1$ is close to $1$.
\item[(2):] Since a general uniformly bounded harmonic function on $\D$ need not be continuous upto the closed disk $\overline{\D}$,  the formula \eqref{rec-unif-harm} for a general uniformly bounded harmonic function $u$ on $\D$ does not follow from the convergence \eqref{harm-meas-cv}.  
\end{itemize}

\subsection{Patterson-Sullivan interpolation of a fixed harmonic function}\label{sec-fixed}
For a fixed function,  we will obtain  interpolation formula for a harmonic function $f:\D\rightarrow \C$ belonging to the following class (which we will call the class of {\it tempered harmonic functions}): 
\begin{align}\label{T-temp}
\mathcal{T}(\D):= \Big\{f: \D\rightarrow \C\Big| \text{$f$ is harmonic, $\lim\limits_{\alpha\to 0^{+}} \alpha^2  \int_\D |f(z)|^2 (1 - |z|^2)^\alpha dA(z) = 0$}\Big\}. 
\end{align}
Note that while all harmonic functions with $\int_\D |f|^2 dA<\infty$ are tempered, a tempered harmonic function may satisfy $\int_\D |f|^2 dA  = \infty$.   For the purpose of simultaneous interpolation in \S \ref{sec-rec-family}, it is  important for us to deal with the class $\mathcal{T}(\D)$ of tempered harmonic functions (see Lemma~\ref{lem-sc-w-g} and Remark~\ref{rem-sc-w-g} below).

 Recall  that a non-negative function $\Lambda$ on $\R_{+}$ (or on $\N$)  is called {\it sub-exponential} if 
\[
\lim_{t\to + \infty} \Lambda(t) e^{-\alpha t} = 0 \quad \text{for any $\alpha >0$.} 
\]
 For any $k\in \N$ and $z\in \D$, set 
\[
\A_k(z)  = \{x\in\D: k \le d_\D(z, x) < k + 1\}.
\]
We say that a harmonic function $f:\D\rightarrow \C$ has  {\it sub-exponential mean-growth} if there exists a sub-exponential function $\Lambda: \N\rightarrow \R_{+}$ such that 
\begin{align}\label{def-sexp-g}
\int_{\A_k(0)} |f(w)|^2 \frac{ d A(w)}{(1 - |w|^2)^{2}} \le \Lambda(k) e^{2k} \quad \text{for all $k\in \N$.}
\end{align}
Set
\begin{align}\label{def-sub-exp-c}
\sexp(\D): = \Big\{f:\D\rightarrow \C\Big| \text{$f$ is harmonic with sub-exponential mean-growth}\Big\}.
\end{align}

\begin{remark}
For a  harmonic function $f:\D\rightarrow \C$, one can show that the condition \eqref{def-sexp-g} is equivalent to the following condition:
\[
\frac{1}{2\pi}\int_{0}^{2\pi} |f(r e^{i\theta})|^2 d\theta \le  \frac{1}{1-r} \cdot \widetilde{\Lambda}\Big(\log \frac{1}{1-r}\Big), \quad  \text{for all $r \in [0, 1)$,}
\]
where $\widetilde{\Lambda}:\R\rightarrow \R_{+}$ is a sub-exponential function depends on $f$. Nonetheless, the condition \eqref{def-sexp-g} is more convenient for us and has a simpler counterpart in higher dimension case.
\end{remark}

\begin{lemma}\label{lem-sub-expon}
We have the  inclusion: 
$
\mathcal{T}(\D) \subset  \sexp(\D).
$
\end{lemma}

 Since the series \eqref{cv-issue} need not be convergent for a general $f\in A^2(\D)$ or $f \in \mathcal{T}(\D)$, we  shall use summation over annuli to ensure convergence of our series: for any harmonic function $f$ on $\D$ and any $X\in \Conf(\D)$, define
\begin{align}\label{def-g-X-k}
g_X^{(k)}(s, z; f): = \sum_{x\in X\cap \A_k(z)} e^{-s d_\D(z, x)} f(x).
\end{align}

\begin{lemma}\label{lem-L2-sum-small-f}
Assume that $f\in \sexp(\D)$.  Then for any $s>1$ and  any relatively compact subset $D\subset \D$, we have 
 \[
\sum_{k = 0}^\infty  \sup_{z\in D}\Big\{ \E_{\PP_{K_\D}}\Big[ |g_X^{(k)}(s, z; f)|^2\Big]\Big\}^{1/2} <\infty.
\]
In particular,   for $\PP_{K_\D}$-almost every $X\in \Conf(\D)$,  we have 
\begin{align}\label{ae-z-abcv}
\sum_{k=0}^\infty       |  g_X^{(k)}(s, z; f)  | <\infty \quad \text{for Lebesgue almost every $z\in \D$.}
\end{align}
\end{lemma}

\begin{remark}
For a general $f \in \sexp(\D)$, it is not clear whether \eqref{ae-z-abcv} holds for all $z\in \D$. This should be compared to the stronger result obtained in Lemma \ref{lem-g-uncond}.
\end{remark}

From Lemma \ref{lem-L2-sum-small-f},  fixing any $f\in \sexp(\D)$, any $z\in \D$ and any $s>1$, for $\PP_{K_\D}$-almost every $X\in \Conf(\D)$, we can define
\begin{align}\label{def-db-s}
g_X(s, z; f): = \sum_{k=0}^\infty g_X^{(k)}(s, z; f) = \sum_{k = 0}^\infty \sum_{x\in X\cap \A_k(z)} e^{-s d_\D(z,x)} f(x).
\end{align}
The series $g_X(s, z; f)$ should be viewed as a certain  principal value of the linear statistics for the observable function $e^{-s d_\D(z,x)} f(x)$.

\begin{lemma}
Assume that $f\in \sexp(\D)$. Then for any $s>1$ and any $z\in \D$, we have 
\[
\E_{\PP_{K_\D}}(g_X(s, z; f)) = f(z) \cdot \E_{\PP_{K_\D}} (g_X(s, z)),
\]
where $g_X(s,z)$ is defined in \eqref{def-PS-meas}. 
\end{lemma}

\begin{theorem}\label{prop-intro-single-bergman}
Assume that $f\in \sexp(\D)$.  Then the limit equality 
\begin{align}\label{mean-cv-PS}
\lim_{s\to 1^{+}}  \E_{\PP_{K_\D}}\Big( \Big|  \frac{g_X(s, z; f)}{\E_{\PP_{K_\D}}(g_X(s, z))} - f(z)\Big|^2\Big)= 0
\end{align}
holds for a fixed point $z\in \D$ if and only if 
\[
f\in \TT(\D).
\] 
Moreover, for any $f\in \TT(\D)$, the convergence \eqref{mean-cv-PS} holds locally uniformly on $z\in \D$.  
\end{theorem}

The key ingredient in the proof of Theorem \ref{prop-intro-single-bergman} is given in the following 

\begin{proposition}\label{prop-2-side}
For any $z\in \D$, there exist two constants $c_1(z), c_2(z) >0$ such that for any $f\in \sexp(\D)$ and any $s\in (1, 2]$, we have 
\[
c_1 (z)  \le  \frac{\Var_{\PP_{K_\D}}(g_X(s, z; f)) }{\displaystyle \int_\D |f(w)|^2 (1-|w|^2)^{2s-2} dA(w)} \le c_2(z).
\]
\end{proposition}

\begin{remark}
The idea behind Lemma \ref{lem-L2-sum-small-f} is the following: although a general $f\in A^2(\D)$  or more generally a general $f\in \sexp(\D)$  is not necessarily uniformly bounded on $\D$, its average  over all circles $C(z, r) = \{w\in \D: d_\D(w, z) = r\}$ depends only on the centre $z\in \D$  and thus is bounded in $r\in (0, \infty)$: 
\begin{align}\label{mvp-circle}
f(z)=     \frac{1}{|C(z,r)|_B}\int_{C(z,r)} f(w) ds_B(w),
\end{align}
where $ds_B$ is the length element of the Bergman metric and $|C(z,r)|_B$ is the length of the circle $C(z,r)$ under the Bergman metric.   
Lemma \ref{lem-L2-sum-small-f} implies that that for any $f\in \sexp(\D)$ and $\PP_{K_\D}$-almost every $X\in \Conf(\D)$, there is enough cancellation inside the  partial sum \eqref{def-g-X-k} in such a way  that  for all $s>1$, we have 
\[
\sum_{k=0}^\infty |g_X^{(k)}(s, z; f)|<\infty. 
\]
\end{remark}

\begin{theorem}\label{cor-intro-single-bergman}
Fix $f \in \TT(\D)$.  Let  $(s_n)_{n\ge 1}$ be a fixed sequence in $(1, \infty)$ converging to $1$ and satisfying 
\[
 \sum_{n=1}^\infty    (s_n-1)^2 \int_{\D}   | f(w)|^2  (1 - |w|^2)^{2s_n-2} dA(w)  <\infty.
\]
 Then for $\PP_{K_\D}$-almost every $X\in \Conf(\D)$, the  equality  
\[
   \int_D \limsup_{n\to \infty}  \Big| \frac{g_X(s_n, z; f)}{\E_{\PP_{K_\D}}(g_X(s_n, z))} - f(z)\Big|^2  dA(z) = 0
\]
holds for any relatively compact subset $D\subset \D$  and moreover
\begin{align}\label{rec-single-f-intro}
 f(z) =   \lim_{n\to\infty} \frac{g_X(s_n, z; f)}{g_X(s_n, z)} \quad \text{for all Lebesgue almost every $z\in \D$}.
\end{align}
\end{theorem}

\begin{remark}
 For a general $f\in \TT(\D)$, it is not clear whether the convergence \eqref{rec-single-f-intro} holds for all $z\in \D$.  This should be compared to Theorem \ref{thm-intro-poi} below where the convergence \eqref{rec-single-f-intro} for all $z\in \D$ is established for weighted harmonic Hary functions.
\end{remark}

\begin{remark}
In Lemma \ref{lem-L2-sum-small-f} and in Theorem \ref{prop-intro-single-bergman}, the class $\sexp(\D)$ can  be replaced by the  larger class $\widehat{\mathcal{S}}(\D)$ consisting of all  harmonic functions such that 
\[
\sum_{k=1}^\infty e^{-\frac{(\alpha+1) k}{2}} \Big(\int_{0}^{2\pi} \big|f((1-e^{-k}) e^{i\theta})\big|^2 \frac{d\theta}{2\pi}\Big)^{1/2}<\infty  \quad \text{for all $\alpha>0$}, 
\]
and one can show (the proof involves similar estimates as those in Proposition \ref{prop-2-side}) that for a harmonic function $f$,  the statement in Lemma \ref{lem-L2-sum-small-f} holds if and only if $f$ belongs to this larger class $\widehat{\mathcal{S}}(\D)$. Nonetheless, it is sufficient for us to work with this smaller and simpler class $\sexp(\D)$.
\end{remark}

\subsection{Simultaneous  Patterson-Sullivan interpolation for families of functions}\label{sec-rec-family}

Now we consider the simultaneous Patterson-Sullivan interpolation for families of  holomorphic or harmonic functions on $\D$. 

\subsubsection{Informal description of the simultaneous interpolation}\label{sec-info-des}

Clearly the almost every statement in Theorem \ref{cor-intro-single-bergman} can be  extended to any fixed countable dense family $\mathscr{F} \subset A^2(\D)$. At the same time,  for any $1<s \le \frac{3}{2}$ and any $z\in \D$, we have (cf. Proposition~\ref{prop-sharp})
\begin{align}\label{div-norm}
\sup_{N\in \N}\E_{\PP_{K_\D}} \Big(\sup_{f\in A^2(\D): \|f\|\le 1} \Big|\sum_{k=0}^N g_X^{(k)}(s, z; f)\Big|^2\Big) = \infty. 
\end{align}
 \begin{remark}\label{rem-as-div}
In the sequel to this paper, the  almost sure version of \eqref{div-norm} is proved: 
\begin{align}\label{near-cr-zero}
\PP_{K_\D}\Big(\Big\{X\in \Conf(\D)\Big| \text{$\sum_{k} g_X^{(k)}(s, z; f)$ converges for all $f\in A^2(\D)$}\Big\}\Big)= 0.
\end{align}
\end{remark}
Hence for a fixed $z\in \D$,  once  $s>1$ is close to the critical value $1$,  then for $\PP_{K_\D}$-almost every  $X\in \Conf(\D)$, the normalization $g_X(s, z; f)/g_X(s, z)$ can not be simultaneously defined for all $f\in A^2(\D)$ and it is impossible to extend the  our Patterson-Sullivan interpolation \eqref{rec-single-f-intro} to the whole space  $A^2(\D)$.

\medskip

The above discussions lead to the following considerations.

I): Instead of considering the whole space $A^2(\D)$, we consider smaller families of functions inside $A^2(\D)$. The main results along this line are: 1) an optimal simultaneous and uniform interpolation for weighted Bergman spaces is obtained in Theorem~\ref{prop-w-disc}; 2) a strong simultaneous and uniform interpolation for weighted harmonic Hardy spaces, is obtained  in Theorem~\ref{thm-intro-poi}. 

II): By \eqref{near-cr-zero}, there is an issue of defining  $g_X(s,z; f)$ simultaneously for all $f\in A^2(\D)$. For bypassing this issue, we replace $e^{-s d_\D(z,x)}  = e^{-s d_\D(0, \varphi_z(x))}$ in the definition \eqref{def-db-s} of $g_X(s, z; f) $ by  $\RR(\varphi_z(x))$ with $\RR\in L^1(\D, dA)$ ranges over the set of all non-negative bounded {\it compactly supported} and {\it radial} functions on $\D$ and define, for all $X\in \Conf(\D)$,  
\begin{align}\label{g-W}
g^\RR_X(z; f): =  \sum_{x\in X} \RR(\varphi_z(x)) f(x) \an g^\RR_X(z): =  \sum_{x\in X} \RR(\varphi_z(x)).
\end{align}
Here the radial assumption is related to \eqref{mvp-circle} and  the compact support assumption is imposed on $\RR$ to ensure the convergence of $g_X^\RR(z; f)$ for all configurations $X\in \Conf(\D)$.

However, we obtain in Theorem \ref{prop-failure-intro} below  that for any $z\in \D$,  
\[
\inf_{\RR}  \E_{\PP_{K_\D}} \Big[    \sup_{f\in A^2(\D): \|f\|\le 1} \Big| \frac{ g_X^\RR(z, f)}{  g^\RR_{\PP_{K_\D}}} - f(z) \Big|^2\Big] \ge \frac{1}{128} \quad \text{with $g^\RR_{\PP_{K_\D}}: =  \E_{\PP_{K_\D}} [g^\RR_X (z) ]$},
\]
where the infimum runs over all  non-negative bounded compactly supported radial functions $\RR$ on $\D$ (the fact that $\E_{\PP_{K_\D}} [g^\RR_X (z) ]$ is independent of $z$ follows from the conformal invariance of $\PP_{K_\D}$).  This  demonstrates the impossiblity of the simultaneous Patterson-Sullivan interpolation for all functions in $A^2(\D)$. 

\begin{remark}
One can show the following analogue of \eqref{rec-single-f-intro}:   for any fixed $f \in A^2(\D)$, there exists a sequence $(\RR_n)_{n\ge 1}$ of  non-negative bounded  compactly supported radial functions on $\D$ such that for $\PP_{K_\D}$-almost every $X$, 
\[
f(z) =   \lim_{n\to\infty} \frac{g_X^{\RR_n}(z; f)}{g_X^{\RR_n}(z)} \quad \text{for Lebesgue almost every   $z\in \D$}.
\]
Indeed, we may take $\RR_n(x)= \mathds{1}(d_\D(x, 0)<r_n)$ for a certain sequence $(r_n)_{n\ge 1}$ of positive numbers converging to infinity fast enough. 
\end{remark}

{\flushleft \it Question:} Is it possible to obtain an optimal result, analogous to Theorem~\ref{prop-w-disc},  for  the simultaneous uniform interpolation  of weighted Bergman spaces by using  $g_X^\RR(z; f)$ ?

\subsubsection{Weighted Bergman spaces}\label{sec-wb-intro}
 The following weight is essential for us:
\begin{align}\label{def-crit-w-intro}
W_{\mathrm{cr}}(z)  =  \frac{1}{(1 - |z|^2) \log^{2} \big(\frac{4}{1 - |z|^2}\big)} \quad \text{for all $z \in \D$}.
\end{align} 
Here the subcript ``cr"  comes from the word ``critical".  A function $W\in L^1(\D, dA)$ with $W(z)\ge 0$ is called a {\it super-critical weight} if it satisfies
\[
\lim_{|z|\to 1^{-}}\frac{W(z)}{W_{\mathrm{cr}}(z)}  = \infty.
\]
Note that,  although $W_{\mathrm{cr}}$ is radial, super-critical weights need not be. 

The  weighted Bergman space associated to a weight $W$  is given by
\[
A^2(\D, W): = \Big\{f : \D\rightarrow \C\Big| \text{$f$ is holomorphic and $\|f\|_W^2: = \int_{\D} | f(z)|^2 W(z) dA(z) <\infty$}\Big\}.
\]
Let $\mathcal{B}(W)$ denote  the unit ball of $A^2(\D, W)$: 
\[
\mathcal{B}(W): = \Big\{f \in A^2(\D, W)\Big|  \|f\|_W <1 \Big\}.
\]

\begin{lemma}\label{lem-sc-g}
Let $W$ be a  weight on $\D$, either equal to $W_{\mathrm{cr}}$ or super-critical.  Then for any relatively compact subset $D\subset \D$ and any $s>1$, 
 \[
\sum_{k = 0}^\infty \sup_{z\in D} \Big\{ \E_{\PP_{K_\D}}\Big[  \sup_{f\in \mathcal{B}(W)}|g_X^{(k)}(s, z; f)|^2\Big]\Big\}^{1/2} <\infty.
\]
\end{lemma}

Fix $s>1, z\in \D$ and a weight $W$ which is equal to $W_{\mathrm{cr}}$ or super-critical. Then for $\PP_{K_\D}$-almost every  $X\in \Conf(\D)$,  Lemma \ref{lem-sc-g} resolves the problem (discussed in \S \ref{sec-info-des}) of the simultaneous definitions of  $g_X(s, z; f)$  for all $f\in A^2(\D, W)$:
\begin{align}\label{def-g-X-AD}
g_X(s, z; f) = \sum_{k= 0}^\infty g_X^{(k)} (s, z; f).
\end{align}

\begin{theorem}\label{prop-w-disc}
Let $W$ be a super-critical weight on $\D$. Then for any relatively compact subset $D\subset \D$, 
\begin{align}\label{def-g-PP}
\lim_{s\to 1^{+}} \sup_{z\in D}\E_{\PP_{K_\D}}  \Big ( \sup_{f\in \mathcal{B}(  W)} \Big|  \frac{g_X(s, z; f)}{g_{\PP_{K_\D}}(s)} - f(z)  \Big|^2  \Big) = 0,
\end{align}
where  
\[
g_{\PP_{K_\D}}(s): = \E_{\PP_{K_\D}}[g_X(s, z)].
\]
In particular, there exists a sequence $(s_n)_{n\ge 1}$ in $(1, \infty)$ converging to $1$ such that  for $\PP_{K_\D}$-almost every  $X\in \Conf(\D)$, we have that for Lebesgue almost every $z\in \D$,   
\[
 f(z) =   \lim_{n\to\infty}  \frac{g_X(s_n, z; f)}{g_X(s_n, z)}  \quad \text{simultaneously for all $f\in A^2(\D, W)$,}
\]
where the convergence is uniform for $f$ in the unit ball $\mathcal{B}(W)$ of $A^2(\D, W)$.
\end{theorem}

We shall see later in \S \ref{sec-proof-H-L} that  the sequence $(s_n)_{n\ge 1}$ in Theorem~\ref{prop-w-disc} can be taken to be any sequence $(s_n)_{n\ge 1}$  in $(1, \infty)$  satisfying 
\[
 \sum_{n=1}^\infty     (s_n-1)^2 \int_{\D}  (1 - |x|^2)^{2s_n -2}  K_W(x,x)  dA(x)  <\infty.
\]
 Note that our condition depends only on the given super-critical weight $W$.  

For the critical weight $W_{\mathrm{cr}}$, we have the following 
\begin{proposition}\label{prop-critical-weight}
 Take $z = 0$, then  we have 
\begin{align}\label{inf-non-zero}
\liminf_{s\to 1^{+}}\E_{\PP_{K_\D}}  \Big ( \sup_{f\in \mathcal{B}(  W_{\mathrm{cr}})} \Big|  \frac{g_X(s, z; f)}{g_{\PP_{K_\D}}(s)} - f(z)  \Big|^2  \Big)> 0.
\end{align}
\end{proposition}

As we mentioned before,  in the proof of Proposition \ref{prop-critical-weight}, we use the following estimate (cf. Lemma~\ref{lem-var-M}):  there exists a constant $c>0$ such that  the reproducing kernel  $K_{W_{\mathrm{cr}}}(x,y)$  of the weighted Bergman space $A^2(\D, W_{\mathrm{cr}})$ satisfies
\[
\int_0^{2\pi} K_{W_{\mathrm{cr}}}(x e^{i\theta}, y) |K_\D(xe^{i\theta},y)|^2 \frac{d\theta}{2\pi} \ge \frac{c}{(1 - |xy|^2)^4} \log \Big(\frac{2}{1 - |xy|^2}\Big) \quad\text{ for all $x,y \in\D$.} 
\]
This estimate does not follow from the asymptotics (cf. Lemma~\ref{lem-critical-w}) of the diagonal values $K_{W_{\mathrm{cr}}}(x,x)$ when $x$ is close to the boundary of $\D$. 

It would be interesting to prove a pointwise version of \eqref{inf-non-zero} in the same spirit of \eqref{near-cr-zero}: 
\[
\liminf_{s\to 1^{+}} \sup_{f\in \mathcal{B}(  W_{\mathrm{cr}})} \Big|  \frac{g_X(s, z; f)}{g_{\PP_{K_\D}}(s)} - f(z)  \Big|> 0 \quad \text{ for $\PP_{K_\D}$-almost every  $X\in \Conf(\D)$}.
\]

\subsubsection{Weighted harmonic Hardy spaces}

Recall that   the Poisson transformation a signed Radon measure $\nu$  on $\T$  is a harmonic function on $\D$  given by 
\[
P[\nu] (z): = \int_\T P(z,  \zeta)  d \nu ( \zeta) = \int_\T \frac{1- |z|^2}{ |1 - \bar{\zeta}z|^2} d\nu(\zeta).
\]
Given any  Borel probability measure $\mu$ on $\T$, we define the associated weighted harmonic Hardy space by 
\[
\mathcal{H}^2(\D; \mu) : = \Big\{f: \D\rightarrow \C\Big|  f (z)= P[h\mu](z) = \int_\T P(z, \zeta) h(\zeta) d\mu(\zeta), \, h \in L^2(\T,\mu)\Big\}.
\]

\begin{lemma}\label{lem-g-uncond}
Let $\mu$ be any Borel probability measure on $\T$. Then for $\PP_{K_\D}$-almost every $X\in \Conf(\D)$,  simultaneously for  all $ f  \in \mathcal{H}^2(\D; \mu)$, all $z\in \D$ and all $s >1$, we have   
\[
\sum_{x \in X}   e^{-s d_\D(x, z)} |f(x)|<\infty.
\]
\end{lemma}

For a fixed Borel probability measure $\mu$ on $\T$,  Lemma \ref{lem-g-uncond} implies that  for $\PP_{K_\D}$-almost every $X\in \Conf(\D)$, we can define   $g_X(s, z; f)$ simultaneous for all $f\in \mathcal{H}^2(\D, \mu)$, all $s>1$  and all $z\in \D$. And, in this situation,  $g_X(s, z; f)$ defined in \eqref{def-db-s} can be simplified as
\[
g_X(s, z; f)= \sum_{x\in X}e^{-s d_\D(x,z)} f(x).
\]

\begin{theorem}\label{thm-intro-poi}
Let $\mu$ be any Borel probability measure on $\T$. Then for $\PP_{K_\D}$-almost every $X\in \Conf(\D)$,  simultaneously for  all $ f  \in \mathcal{H}^2(\D; \mu)$ and all $z\in \D$, we have 
\begin{align}\label{PS-f-hardy}
f(z) =    \lim_{s\to 1^{+}}  \frac{g_X(s,z;f)}{g_X(s,z)},
\end{align}
where the convergence is uniform for $f\in \{ P[h\mu]: \| h\|_{L^2(\mu)} \le 1\}$ and locally uniformly on $z\in \D$, that is,  for any relatively compact $D\subset \D$,  
\[
\lim_{s\to 1^{+}}   \sup_{f = P[h\mu]\atop \| h\|_{L^2(\mu)}\le 1}\sup_{z\in D}\Big| \frac{g_X(s,z;f)}{g_X(s,z)} - f(z) \Big| = 0. 
\]  
\end{theorem}

\begin{remark}
Note that in Theorem \ref{thm-intro-poi},  the measure $\mu$ is  not involved explicitly in  \eqref{PS-f-hardy}. The implied subset of full $\PP_{K_\D}$-measure in $\Conf(\D)$ in our statement might however depend on this measure $\mu$ and it is not clear  whether the limit equality \eqref{PS-f-hardy} can be extended to the following family: 
\[
\Big\{f: \D\rightarrow \C\Big| \text{ $f = P[\nu]$ for a  signed Radon measure $\nu$ on $\T$}\Big\}.
\]
\end{remark}

\begin{remark}
We note that, although a function $f  \in \mathcal{H}^2(\D; \mu)$ is harmonic, neither the function $g_X(s, z; f)$ nor the normalization $g_X(s, z; f)/ g_X(s, z)$ is harmonic in $z$. On the other hand, for fixed $x\in \D$, one can show that the function $z\mapsto e^{-s d_\D(x,z)}$ is subharmonic in the region
\[
O_{s,x}: = \Big\{z\in \D:    |\varphi_x(z)| >  s - \sqrt{s^2 - 1} \Big\}.
\]
Note that for a fixed $s> 1$, any fixed compact subset $\D$ eventually is contained in the region $O_{s,x}$ when $x$ is close enough to the boundary. 
\end{remark}

\subsubsection{Comments}
Let us describe a simple reconstruction algorithm for all  harmonic Hardy functions and explain why it is not applicable to weighted Bergman spaces and more general spaces of harmonic functions.

Let $dm$ denote  the normalized Lebesgue measure on $\T$.   If $h \in L^1(\T) = L^1(\T, m)$, then we write $P[h] : = P[hdm]$. The space of harmonic Hardy functions on $\D$ is defined by 
\[
\mathcal{H}^1(\D):  = \left\{u: \D\rightarrow \C\Big| u = P[h],  \, h \in L^1(\T)\right\}. 
\]

We can show that $\PP_{K_\D}$-almost every  $X$ satisfies that for Lebesgue almost every $\zeta \in \T$, the Stolz angle $S_\zeta$, the closed convex  hull of $\{\zeta\} \cup \{z\in \D: |z| \le 1/\sqrt{2}\}$, contains infinitely many points.  Then for any $u = P[h]\in \mathcal{H}^1(\D)$ with $h \in L^1(\T)$, for Lebesgue almost every point $\zeta\in \T$, the {\it non-tangential limit} of $u$ exists and is equal to $h$:
\[
h(\zeta) =  u^*(\zeta): =   \lim_{S_{\zeta} \ni z \to \zeta} u (z) =  \lim_{X\cap S_{\zeta} \ni z \to \zeta} u (z) \quad \text{for Lebesgue almost every $\zeta\in \T$}.
\]  
Therefore, for all $z\in \D$,  we have 
$
u(z)=  P[h](z) = P[u^{*}](z). 
$ 

However, the above reconstruction  algorithm is not applicable to any weighted Bergman space, since  it always contains functions without non-tangential limit (cf. MacLane \cite{MacLane} or Duren \cite[Thm. 5.10]{Duren-Hp}).  Also,  it is not applicable to  harmonic functions of the form $u= P[\nu]$ with a singular measure $\nu$ on $\T$. For instance, if $\nu$ is singular to $dm$, then $u^* = (P[\nu])^*$ vanishes Lebesgue almost everywhere (cf. e.g. Duren \cite[Thm 1.2]{Duren-Hp} and \cite{BR, P1, P2}), therefore, we obtain $u \ne P[u^*] =0$.

\subsubsection{Outline of the proofs of Theorems \ref{prop-w-disc} and \ref{thm-intro-poi}}
Let $\HH$ be a Hilbert space.  The proofs of Theorems \ref{prop-w-disc} and \ref{thm-intro-poi} rely on an extension  of Theorem \ref{prop-intro-single-bergman} to a fixed vector-valued harmonic function $F:\D\rightarrow \HH$ belonging to the following class 
\[
\TT(\D; \HH): = \Big\{F: \D\rightarrow \HH\Big| \text{$f$ is harmonic and $\lim\limits_{\alpha\to 0^{+}} \alpha^2 \cdot  \int_\D \|F(w)\|^2 (1 - |w|^2)^\alpha dA(w) = 0$}\Big\}. 
\]

Let us  illustrate the main steps for the proof of Theorem \ref{prop-w-disc} for a fixed $z\in \D$  as follows. Similarly to the definition \eqref{def-g-X-k}, for a fixed vector-valued function    $F:\D\rightarrow\HH$, we define
\begin{align}\label{def-g-k}
g_X^{(k)}(s, z; F): = \sum_{x\in X\cap \A_k(z)} e^{-s d_\D(z, x)} F(x).
\end{align}

{\flushleft \bf Step 1. }
The definition \eqref{def-sexp-g} of harmonic function with sub-exponential mean-growth naturally extends to $\HH$-valued harmonic function $F: \D\rightarrow \HH$. We denote 
\[
\sexp(\D, \HH): = \Big\{F:\D\rightarrow \HH\Big| \text{$F$ is harmonic with sub-exponential mean-growth}\Big\}.
\]
  It turns out (cf. Lemma \ref{prop-L2-L1-sum}) that if  $F\in \sexp(\D, \HH)$,  then for any $s>1$ and any $z\in \D$, 
 \begin{align*}
 \sum_{k=0}^\infty   \Big(\E_{\PP_{K_\D}}\Big[\| g_X^{(k)}(s, z; F) \|^2\Big]\Big)^{1/2}<\infty
\end{align*}
and thus for $\PP_{K_\D}$-almost every $X\in \Conf(\D)$, the following series converges in $\HH$:
\[
g_X(s, z; F) : = \sum_{k = 0}^\infty g_X^{(k)}(s,z; F).
\]
Then we prove (cf. Lemma \ref{prop-var-up-bd-rep}) that for any $F\in \sexp(\D, \HH)$,  any $s > 1$ and $z\in \D$,
\[
\E_{\PP_{K_\D}} ( g_X(s, z; F)  ) =   F(z) \cdot \E_{\PP_{K_\D}}(g_X(s, z))
\]
and
\begin{align}\label{upper-var}
\Var_{\PP_{K_\D}} (g_X(s, z; F)) \le  2  \int_\D  e^{-2s d_\D(z, w)} \| F(w)\|^2  \frac{dA(w)}{(1- |w|^2)^2}.
\end{align}

{\flushleft \bf Step 2.} We then show  (cf. Lemma \ref{lem-sub-ex})  that $\TT(\D, \HH) \subset \sexp(\D, \HH)$. By the definition of $\TT(\D, \HH)$, we  derive from \eqref{upper-var}  that there exists a  sequence $(s_n)_{n\ge 1}$ in $(1, \infty)$ converging to $1$ such that   for $\PP_{K_\D}$-almost every  $X\in \Conf(\D)$ and our fixed point $z\in \D$, 
\begin{align}\label{H-value-F-cov}
  \lim_{n\to\infty}   \Big\| \frac{g_X(s_n, z; F)}{g_X(s_n, z)} - F(z)\Big\|  = 0.
\end{align}

{\flushleft \bf Step 3.}   Let $W$ be a super-critical weight on $\D$. Then $A^2(\D, W)$ is a reproducing kernel Hilbert space with a reproducing kernel denoted by $K_W(\cdot, \cdot)$. Set $F_W: \D\rightarrow A^2(\D, W)$ by 
\[
F_W(w) = K_W(\cdot, w).
\]  
We show (cf. Lemma \ref{lem-sc-w-g}) that once $W$ is super-critical, then  $F_W \in \TT(\D, A^2(\D, W))$ and  Theorem \ref{prop-w-disc} follows from the equality \eqref{H-value-F-cov} and the reproducing property of $A^2(\D, W)$: 
\[
f(z)=  \langle f, F_W(z)\rangle_{A^2(\D, W)} \quad \text{for all $f\in A^2(\D, W)$ and all $z\in\D$.}
\]

\medskip
\medskip

The proof of Theorem \ref{thm-intro-poi} requires some further steps involving the concept of sharply tempered growth condition (cf. Definition \ref{def-cmvp-sharp} and Theorem \ref{thm-n-sub}) and the proof of fact that the Poisson kernel satisfies this condition (cf. Lemma \ref{lem-cmvp-cb}).

\subsection{Impossibility of simultaneous uniform interpolation of  $A^2(\D)$}

Let $\RR: \D \rightarrow \R_{+}$ be a non-negative bounded compactly supported {\it radial} function. For any $z\in \D$, any $f\in A^2(\D)$ and any $X\in \Conf(\D)$, we define $ g_X^\RR(z, f)$ and $ g_X^\RR(z)$ as in \eqref{g-W}.  Recall that   $\E_{\PP_{K_\D}} [g^\RR_X (z)]$ is independent of $z\in \D$.

\begin{theorem}\label{prop-failure-intro}
For any $z\in \D$, we have 
\[
\inf_{\RR}  \E_{\PP_{K_\D}} \Big[    \sup_{f\in A^2(\D): \|f\|\le 1} \Big| \frac{ g_X^\RR(z, f)}{  g^\RR_{\PP_{K_\D}} } - f(z) \Big|^2\Big] \ge \frac{1}{128} \quad \text{with $g^\RR_{\PP_{K_\D}}: =  \E_{\PP_{K_\D}} [g^\RR_X (z)]$,}
\]
where the infimum is taken over all non-negative bounded compactly supported radial functions $\RR$ on $\D$.
\end{theorem}

\begin{remark}
Recall that the natural radial assumption  on $\RR$  is related to \eqref{mvp-circle}.
\end{remark}

The following proposition  is an average version of \eqref{near-cr-zero}, which shows the difficulty of defining  simultaneously $g_X(s, z; f)$ for all $f\in A^2(\D)$.

\begin{proposition}\label{prop-sharp}
For any $1< s \le \frac{3}{2}$, we have
\[
\sup_{N\in \N}\E_{\PP_{K_\D}} \Big(\sup_{f\in A^2(\D): \|f\|\le 1} \Big|\sum_{k=0}^N g_X^{(k)}(s, z; f)\Big|^2\Big) = \infty. 
\]
\end{proposition}


\subsection{Questions and Conjectures}\label{sec-op}

\begin{question*}
 Can we  extend the equality \eqref{PS-f-hardy} to  all functions of the form $f = P[\nu]$ with $\nu$ signed Radon measures on $\T$ ? Do we have  
\[ 
\lim_{s\to 1^{+}} \E_{\PP_{K_\D}}\Big[    \sup_{\zeta\in \T} \Big| \frac{g_X(s,z;P(\cdot, \zeta))}{g_{\PP_{K_\D}}(s)} - P(z, \zeta) \Big| \Big] = 0 \, ? 
\]
Here $g_{\PP_{K_\D}}(s)$ is defined as in \eqref{def-g-PP} and $P(z, \zeta)$ is the Poisson kernel given in \eqref{poi-ker}. 
\end{question*}

\begin{question*}
Can we construct in deterministic ways  explicit configurations $X\in \Conf(\D)$ with {\it critical upper density} (namely, the number of points of $X$ inside any compact subset of $\D$ is controlled by the hyperbolic area of that compact subset) such that  it satisfies the simultaneous interpolation property as follows: 
\begin{align}\label{f-z-H-inf}
f(z) =    \lim_{s\to 1^{+}}  \frac{g_X(s,z; f)}{g_X(s,z)} \quad \text{for all $f\in H^\infty(\D)$ and all $z\in \D$}\, ? 
\end{align}
Here $H^\infty(\D)$ is the space of all bounded holomorphic functions on $\D$. 
\end{question*}

\begin{conjecture}
In the spirit of Theorem \ref{prop-failure-intro}, we conjecture that 
\[
 \inf_{\RR} \sup_{f\in A^2(\D): \|f\|\le 1} \Big| \frac{ g_X^\RR(z, f)}{  g^\RR_{\PP_{K_\D}} } - f(z) \Big|>0 \quad \text{ for $\PP_{K_\D}$-almost every  $X\in \Conf(\D)$},
\]
where the infimum is taken over all non-negative bounded compactly supported radial functions $\RR$ on $\D$.
\end{conjecture}

\medskip

\noindent {\bf{Acknowledgements.}} Mikhael Gromov taught the Patterson-Sullivan theory to the older of us  in 1999; we are greatly indebted to him. We are deeply grateful to Alexander Borichev, S\'ebastien Gou\"ezel,  Pascal Hubert, Alexey Klimenko and Andrea Sambusetti for useful discussions. Part of this work was done during a visit to the Centro  De Giorgi della Scuola Normale Superiore di Pisa. We are deeply grateful to the Centre for its warm hospitality.
AB's research has received funding from the European Research Council (ERC) under the European Union's Horizon 2020 research and innovation programme under grant agreement No 647133 (ICHAOS). YQ's research is supported by  the National Natural Science Foundation of China, grants NSFC Y7116335K1,  11801547 and 11688101.

\section{Preliminaries on point processes}\label{sec-pp}

\subsection{Configurations and point processes}\label{sec-conf-pp}
Let $E$ be a metric complete separable space, equipped with a $\sigma$-finite positive Radon measure $\mu$. 
A configuration $X$ on $E$ is  a  collection of points of $E$, possibly with finite multiplicities and considered without regard to order,  such that any relatively compact subset  $B\subset E$ contains only finitely many points.   Let $\Conf(E)$ denote the space of all configurations on $E$.  A configuration $X \in \Conf(E)$  may be identified  with a purely atomic Radon measure $\sum_{x \in X} \delta_x$,  where $\delta_x$ is the Dirac mass at the point $x$, and  the space $\Conf(E)$ is a complete separable metric space with respect to the vague topology on the space of Radon measures on $E$.   
A Borel probability measure $\PP$ on $\Conf(E)$ is called a {\it point process} on $E$.  For further background on  point processes,  see   Daley and Vere-Jones \cite{DV-1}, Kallenberg \cite{Kallenberg}. 

A configuration is called simple if all its points have multiplicity one. A point process $\PP$ is called simple, if $\PP$-almost every configuration is simple. 
For a simple point process  $\PP$  on $E$ and an integer $n\ge 1$, we say that a $\sigma$-finite measure $\xi_\PP^{(n)}$ on $E^n$ is the $n$-th correlation measure of $\PP$ if for any bounded compactly supported function $\phi: E^n \rightarrow \C$, we have 
\begin{align*}
\E_\PP\Big( \sum_{x_1, \dots, x_n \in X \atop x_i \ne x_j}  \phi(x_1, \dots, x_n) \Big)=   \int_{E^n} \phi(y_1, \dots, y_n)  d\xi_\PP^{(n)} (y_1, \cdots, y_n). 
\end{align*}
If $\xi_\PP^{(n)}$ is absolutely continuous to the measure $\mu^{\otimes n}$, then the Radon-Nikodym derivative 
\[
\rho_n^{(\PP)}(x_1, \cdots, x_n) := \frac{d\xi_\PP^{(n)}}{d\mu^{\otimes n}} (x_1, \cdots, x_n) \quad \text{where $(x_1, \cdots, x_n) \in E^n$,}
\]
is called the $n$-th correlation function  of $\PP$ with respect to the reference measure $\mu$.

\subsection{Determinantal point processes}
 Let $K$ be a {\it locally trace class positive contractive} integral operator on the complex Hilbert space $L^2(E,\mu)$. By slightly abusing the notation, we denote  by $K(x, y)$ the kernel of the integral operator $K$.  By  a theorem obtained by Macchi  \cite{Macchi-DP} and  Soshnikov \cite{DPP-S}, as well as by Shirai and Takahashi  \cite{ST-DPP}, there exists a unique simple point process $\PP_K$ on $E$ such that for any positive integer $n\in \N$, its $n$-th correlation function, with respect to the reference measure $\mu$,   exists and is given by 
\[
\rho_n^{(\PP_K)} (x_1, \cdots, x_n)  = \det(K(x_i, x_j))_{1 \le i, j \le n}.
\]
The point process $\PP_{K}$ is called  the determinantal point process induced by the kernel $K$.

Let us recall  the standard expression  for the variance of linear statistics under  determinantal point processes  induced by orthogonal projections.
\begin{lemma}\label{lem-var-sob}
Let $(E, \mu)$ be a locally compact metric complete separable space equipped with a Radon measure $\mu$. 
Let $\PP_{K}$ be the determinantal point process on  $E$ induced by a locally trace class orthogonal projection $K: L^2(E, \mu) \rightarrow L^2(E, \mu)$.  Let $\HH$  be a Hilbert space and  $F: E\rightarrow \HH$ be a function with 
$
\E_{\PP_{K}} [ \sum_{x\in X} \|F(x)\| + \| F(x)\|^2 ] <\infty$.  Then
\[
\Var_{\PP_{K}} \Big[    \sum_{x\in  X} F(x)  \Big]= \frac{1}{2} \int_{E} \int_{E} \|F(x) - F(y) \|^2 \cdot  |K(x, y)|^2 d\mu(x) d\mu(y)
\]
and thus
\[
\Var_{\PP_{K}} \Big[    \sum_{x\in  X} F(x)  \Big]\le 2 \int_{E}  \|F(x) \|^2 \cdot K(x, x) d\mu(x).
\]
\end{lemma}

\section{Preliminaries on the complex hyperbolic spaces}\label{sec-intro-chs}

Let $d \ge 1$ be an integer.  For any $z, w \in \C^d$, write 
$
z \cdot w  = \sum_{k=1}^d z_k w_k$ and $\bar{z} = (\bar{z}_1, \cdots, \bar{z}_d)$. The Euclidean norm is denoted by $|z| = \sqrt{ z \cdot \bar{z}}$.  Let $\D_d  = \{z  \in \C^d: |z|  < 1\}$ be the open unit ball in $\C^d$. The origin of $\C^d$ will be denoted by $o$, that is, $o = (0, 0, \cdots, 0)\in \C^d$.

\subsection{The Bergman metric on $\D_d$}
Recall that any bounded complex domain carries a natural Riemannian metric, the Bergman metric (cf. Krantz \cite[Chapter 1]{Krantz}),  defined in terms of the reproducing kernel of the space of square-integrable holomorphic functions on the bounded domain. For the open unit ball $\D_d$,  the Bergman metric takes the form
\[
d s_B^2 : =  4 \frac{| dz_1|^2 + \cdots + |dz_d|^2}{1 - |z|^2} + 4 \frac{| z_1 dz_1 + \cdots  + z_d dz_d|^2 }{(1 - |z|^2)^2}.
\]
Let $d_B(\cdot, \cdot)$ denote the distance under the Bergman metric on $\D_d$.  Let $\Aut(\D_d)$ denote the group of all biholomorphic functions $\psi: \D_d\rightarrow \D_d$.  
Note that  $d_B$ is invariant under the action of $\Aut(\D_d)$, that is, $d_B(\varphi(z), \varphi(w)) = d_B(z, w)$ for all $z, w \in \D_d$ and $\varphi \in \Aut(\D_d)$.  The ball $\D_d$ endowed with $d_B$ is a model for the complex hyperbolic space.

For $w = o$, set $\varphi_o(z) = - z$. For $w \in \D_d \setminus \{ o \}$, set
\begin{align}\label{inv-auto}
\varphi_w(z) :=   \frac{w - \frac{z \cdot \bar{w}}{|w|^2}w - \sqrt{1 - |w|^2}\big(z - \frac{z \cdot \bar{w}}{|w|^2} w\big)  }{1 - z \cdot \bar{w}},   \quad z \in \D_d.
\end{align}
  By Rudin \cite[Thm. 2.2.2]{Rudin-ball}, the map $\varphi_w$ defines a biholomorphic involution of $\D_d$ 
interchanging $w$ and $o$. For any $z, w \in \D_d$, we have
\begin{align}\label{def-B-metric}
d_B(z, w) = \log \left(\frac{1 +  |\varphi_w(z)|}{ 1- | \varphi_w(z)|}\right).
\end{align}

\subsection{The conformal invariant measure on $\D_d$}

Let $dv_d(z)$ denote the normalized Lebesgue measure on $\D_d$ such that $v_d(\D_d) = 1$. The volume measure $\mu_{\D_d}$  associated to the Bergman metric  and invariant under the group action of $\Aut(\D_d)$ is given by 
\begin{align}\label{Berg-vol-def}
d\mu_{\D_d} (z) = \frac{ d v_d(z)}{(1 - |z|^2)^{d+1}}.
\end{align}

Let
$
B(z,r): = \{w\in \D_d: d_B(w,z)<r\}
$ denote the ball in $\D_d$ with respect to $d_B$. 
The elementary asymptotics in Lemmas \ref{lem-cb-ent} and \ref{lem-s-d} will be useful for us.

\begin{lemma}\label{lem-cb-ent}
For any $z\in \D_d$, we have
\begin{align}\label{def-entropy}
\lim_{r\to\infty} \frac{\mu_{\D_d}(B(z,r))}{ e^{d r}}  = \frac{1}{4^d}.
\end{align}
\end{lemma}

\begin{proof}
Note that $\mu_{\D_d} (B(z, r))   = \mu_{\D_d} (B(o, r))$ for any $z\in \D_d$ and any $r>0$. Now by the formula of integration in polar coordinates and change of variables,  we have
\begin{align*}
\mu_{\D_d} (B(o, r))  = \int_{B(o, r)}  \frac{ d v_d(z)}{(1 - |z|^2)^{d+1}} =      \frac{d}{4^d}\int_{0}^r   e^{-d x} (e^x -1)^{2d -1} (e^x +1) dx,
\end{align*}
  By l'H\^opital's rule, we obtain the desired limit equality \eqref{def-entropy}.
\end{proof}

\begin{lemma}\label{lem-s-d}
 For any $z\in \D_d$, we have 
\begin{align}\label{lem-s-d-goal}
\lim_{s\to d^{+}}  (s - d) \int_{\D_d} e^{-sd_B(x, z)} d\mu_{\D_d}(x)   =   \frac{d}{4^d}. 
\end{align}
\end{lemma}

\begin{proof}
For any $s, t>0$, we can write 
\begin{align}\label{int-rep-exp}
e^{-s t} = \int_{\R_{+}}   s e^{-s r}  \cdot \mathds{1}(t < r) dr.
\end{align}
Thus for any $z\in \D_d$, we have
\begin{align*}
 & \int_{\D_d} e^{- s d_B(x, z)} d\mu_{\D_d}(x) =   \int_{\D_d} e^{- s d_B(x, o)} d\mu_{\D_d}(x)  
\\
& = \int_{\D_d}  d\mu_{\D_d}(x) \int_{\R_{+}} s e^{- s  r} \mathds{1}( d_B(x, o) < r) dr  = s \int_{\R_{+}} e^{-s r}  \mu_{\D_d}(B(o, r)) dr.
\end{align*}
Set $\kappa(r) : = \frac{\mu_{\D_d}(B(o, r))}{e^{rd}}$. Then for any $s>d$,
\[
\int_{\D_d} e^{- s d_B(x, z)} d\mu_{\D_d}(x) = s \int_{\R_{+}} \kappa(r) e^{- (s - d) r} dr  =\frac{s}{s-d} \int_{\R_{+}} \kappa \Big(\frac{t}{s - d}\Big)  e^{-t}dt. 
\]
The equality \eqref{lem-s-d-goal} follows from  \eqref{def-entropy} and the Dominated Convergence Theorem.  
\end{proof}

\subsection{$\MM$-harmonic and pluriharmonic functions}\label{sec-M-har}
The invariant Laplacian  $\Delta_B$ on $\D_d$ is given by 
\[
\Delta_B = (1 - |z|^2) \sum_{i, j=1}^d (\delta_{ij} - z_i \bar{z}_j) \frac{\partial^2}{\partial z_i \partial\bar{z}_j}.
\]
A function $f \in C^2(\D_d)$ is called  {\it $\MM$-harmonic} if $\Delta_B f\equiv 0$ on $\D_d$.  Note that while holomorphic functions on $\D_d$ are $\MM$-harmonic, an Euclidean harmonic function on $\D_d$ need not be.   Set $\Sph_d = \{z \in \C^d: |z|=1\}$ and let  $\sigma_{\Sph_d}$ be the normalized surface measure on $\Sph_d$ such that $\sigma_{\Sph_d}(\Sph_d)= 1$. By Rudin \cite[Cor. 2 of Thm. 4.2.4]{Rudin-ball}, a continuous function $u\in C(\D_d)$ is $\MM$-harmonic if and only if it satisfies the {\it invariant mean-value property}: 
\begin{align}\label{inv-mvp}
u(\psi(o)) = \int_{\Sph_d}  u(\psi(t \zeta)) d\sigma_{\Sph_d}(\zeta) \quad \text{for any $\psi \in \Aut(\D_d)$ and any $0 < t< 1$}.
\end{align}

The Poisson-Szeg\H{o} kernel $P^b: \D_d \times \Sph_d \rightarrow \R_{+}$  is defined by the formula
\begin{align}\label{def-poi-ker}
P^b(w, \zeta)  = \frac{(1-|w|^2)^d}{|1 - \zeta\cdot \bar{w}|^{2d}}, \quad w \in \D_d \an \zeta \in \Sph_d.
\end{align}
The Poisson transformation of a signed Borel measure $\nu$  on $\Sph_d$ of finite total variation is an  $\MM$-harmonic function on $\D_d$ and   is  defined by 
\begin{align}\label{def-cb-poi}
P^b[\nu] (z): = \int_{\Sph_d} P^b(z,  \zeta)  d \nu ( \zeta), \quad z \in \D_d.
\end{align}

The definition of $\MM$-harmonicity naturally  extends to vector-valued functions. 
Let $\HH$ be a Hilbert space and $F: \D_d\rightarrow \HH$ be an $\MM$-harmonic function, then \eqref{inv-mvp} implies 
\begin{align}\label{id-cb-mvp}
F(z) =\frac{1}{\mu_{\D_d}(B(z, r))}\int_{B(z, r)} F(x) d\mu_{\D_d}(x) \quad \text{for any  $z \in \D_d$ and $r > 0$.}
\end{align}

A function $F:\D_d\rightarrow \HH$ is called pluriharmonic if 
\[
\frac{\partial^2}{\partial z_j \partial\bar{z}_k} F = 0 \quad \text{for all $1\le j, k \le d$.} 
\]
Note that a function on $\D_d$ is pluriharmonic if and only  if it is both $\MM$-harmonic and Euclidean harmonic, see \cite[Section 4.4, p. 59]{Rudin-ball}.

\subsection{Weighted Bergman spaces}\label{sec-w-b}

A function $W\in L^1(\D_d, dv_d)$ with $W(z)\ge 0$ and $\int_{\D_d} W dv_d >0$ is called a {\it weight} on $\D_d$.  Given a weight $W$ on $\D_d$, set
\[
A^2(\D_d, W): = \Big\{f : \D_d\rightarrow \C\Big| \text{$f$ is holomorphic and $\| f\|_W^2: = \int_{\D_d} | f(z)|^2 W(z) dv_d(z) <\infty$}\Big\}.
\]
We call $W$  {\it Bergman-admissible} if for any compact subset $S\subset \D_d$,
\[
\sup_{z\in S} \sup_{f \in \mathcal{B}(W)}|f(z)| <\infty  \quad \text{where $\mathcal{B}(W)$ is the unit ball of $A^2(\D_d, W).$}
\]
Let   $W$  be a  Bergman-admissible weight on $\D_d$. Then $A^2(\D_d, W)$  is closed in $L^2(\D_d, W)$ and is called the {\it weighted Bergman space} associated with the weight $W$. It is a reproducing kernel Hilbert space whose reproducing kernel will be denoted  by  $K_W(\cdot, \cdot)$. We use the convention that the function $K_W(z,w)$ is holomorphic in $z$ and anti-holomorphic in $w$.  The following equality will be useful for us:  for any $z\in \D_d$, 
\begin{align}\label{K-W-diag-norm}
K_W(z,z) = \sup\Big\{|f(z)|^2 \Big| \text{$f\in A^2(\D_d, W)$ and $  \int_{\D_d} | f|^2 W dv_d \le 1$} \Big\}.
\end{align}

The Bergman kernel $K_{\D_d}$ (corresponding to $W\equiv 1$)   is given by  (cf. Rudin \cite[\S 3.1.2]{Rudin-ball}):
\begin{align}\label{berg-exp}
K_{\D_d} (z, w) = \frac{1}{ ( 1 - z \cdot \bar{w})^{ d+1}}, \quad z, w\in \D_d.
\end{align}
The reader is referred to Hedenmalm-Korenblum-Zhu \cite[Chapters 1 and 9]{HKZ-Bergman} for more details on  weighted Bergman spaces and  to  Hedenmalm-Jakobsson-Shimorin \cite{HJS} for  weighted Bergman spaces associated with  logarithmically subharmonic weights.

One can show (cf. Duren \cite[Thm. 1 in Chapter 1]{Duren-Bergman}) that if $\Phi: [0,1) \rightarrow [0,\infty)$ is an integrable function and $\int_r^1 \Phi(t)dt >0$ for any $r\in (0,1)$, then the radial weight $W_\Phi(z): = \Phi(|z|)$ on $\D_d$ is Bergman-admissible. Note also that if a weight $W$ is Bergman-admissible, then so is the weight $\rho(z)W(z)$ for any $\rho\in L^1(\D_d, W)$ with $\inf_z \rho(z) >0$.

\section{Patterson-Sullivan interpolations}\label{sec-PS-d}
In this section, the determinantal point  process $\PP_{K_{\D_d}}$ will be denoted simply by 
\[
\PP_d: = \PP_{K_{\D_d}}.
\]
\subsection{Outline of this section}
The main results of this section are described as follows. 

1). In Theorem \ref{thm-H-L} we obtain the Patterson-Sullivan interpolation of a {\it fixed} $\MM$-harmonic function on $\D_d$ satisfying the {\it tempered growth condition} (the precise meaning is given in Definition \ref{def-MVP}). 

2). We show  in  Theorem \ref{thm-pl-2side} that in the case of pluriharmonic functions, the tempered growth condition is necessary and sufficient    such that the Patterson-Sullivan interpolation  holds in the sense of $L^2$-mean convergence. 

3). In Theorem~\ref{thm-wb}, we obtain the simultaneous uniform interpolation of all functions in the weighted Bergman space $A^2(\D_d, W)$ for a super-critical weight $W$ on $\D_d$ (see \eqref{def-ge-W} for the precise meaning of super-critical weight). 

4). In Theorem \ref{thm-poi}, we obtain the simultaneous uniform interpolation of all functions in a weighted harmonic Hardy space $\mathcal{H}^2(\D_d, \mu)$ (see \eqref{def-cm-hardy} for the precise definition) associated to any Borel probability measure $\mu$ on the sphere $\Sph_d$. The proof of Theorem~\ref{thm-poi} relies  on  a {\it sharply tempered growth condition}  (see Definition \ref{def-cmvp-sharp} and Lemma~\ref{lem-cmvp-cb}).

\subsection{Preliminary properties of tempered $\MM$-harmonic functions}\label{sec-pre-temp}

Let  $\HH$  be a Hilbert space over $\R$ or $\C$.  For any integer $k\ge 0$ and $z\in \D_d$,  set 
\[
\A_k(z) =\{x\in \D_d | k\le d_B(x, z) < k + 1\}.
\]

\begin{definition}\label{def-MVP}
An $\MM$-harmonic function $F: \D_d \rightarrow \HH$ is called {\it tempered} if 
\begin{align}\label{alpha-to-zero}
\lim_{\alpha \to 0^{+}}  \alpha^2 \int_{\D_d} \| F(x)\|^2 (1- |x|^2)^{\alpha + d -1}  dv_d(x) = 0.
\end{align}
\end{definition}

\begin{remark}\label{rem-eq-temp}
 By \eqref{lem-s-d-goal}, it is easy to show that  the condition \eqref{alpha-to-zero} is equivalent to 
\begin{align}\label{global-growth-f}
 \lim_{s \to d^{+}} \frac{\displaystyle \int_{\D_d}  e^{-2s d_B(x, o)}  \| F(x)\|^2   d\mu_{\D_d}(x)}{ \displaystyle \Big[\int_{\D_d} e^{-sd_B(x, z)} d\mu_{\D_d}(x)  \Big]^2}   = 0. 
\end{align}
\end{remark}

\begin{definition}\label{def-cmvp-sharp}
 A tempered $\MM$-harmonic function $F: \D_d \rightarrow \HH$ is called  {\it sharply tempered} if  there exists a strictly decreasing  sequence $(\varepsilon_n)_{n\ge 1}$ in $\R_{+}$  satisfying
\begin{align}\label{e-n-gap}
\lim_{n\to\infty} \varepsilon_n = 0 \an \lim_{n\to\infty} \frac{\varepsilon_{n+1}}{\varepsilon_n} = 1,
\end{align}
such that 
\begin{align}\label{gap-sum}
\sum_{n=1}^\infty \varepsilon_n^{2} \int_{\D_d} e^{-2 \varepsilon_n d_B(x, o)}  \| F(x)\|^2 e^{-2d \cdot d_B(x, o)} d\mu_{\D_d}(x) <\infty. 
\end{align}
\end{definition}

\begin{definition}\label{def-sub-exp}
An $\MM$-harmonic function $F:\D_d\rightarrow \HH$ is said to have {\it sub-exponential mean-growth} if there is a sub-exponential function $\Lambda: \N\rightarrow \R_{+}$ such that 
\begin{align}\label{sub-exp-correction}
\int_{\A_k(o)} \|F(x)\|^2 d\mu_{\D_d}(x) \le \Lambda(k) e^{2kd} \quad \text{for all $k\in \N$.}
\end{align}
\end{definition}

In what follows, we denote 
\begin{align*}
& \TT(\D_d, \HH):= \Big\{F: \D_d\rightarrow \HH\Big| \text{$F$ is tempered $\MM$-harmonic function}\Big\};
\\
& \TT_{ph}(\D_d, \HH): = \Big\{F\in \TT(\D_d, \HH)\Big| \text{$F$ is pluriharmonic}\Big\};
\\
& \sexp(\D_d, \HH): = \Big\{F:\D\rightarrow \HH\Big| \text{$F$ is $\MM$-harmonic with sub-exponential mean-growth}\Big\}.
\end{align*}

\begin{lemma}\label{lem-sub-ex}
We have the inclusion:
$
\TT(\D_d, \HH) \subset \sexp (\D_d, \HH). 
$
\end{lemma}

\begin{proof}
 Note that there exist constants $c_1, c_2>0$ such that 
\[
c_1 e^{-k}\le  1 - |x|^2  \le c_2 e^{-k} \quad \text{for all $x\in \A_k(o)$.} 
\]  
Then by taking $\alpha_k = 1/\sqrt{k}$ and noting \eqref{Berg-vol-def}, we have  
\begin{multline*}
 e^{-2 kd} \int_{\A_k(o)}\| F(x) \|^2 d\mu_{\D_d}(x) = e^{-2 kd} \int_{\A_k(o)}\| F(x) \|^2  (1 - |x|^2)^{\alpha_k + d-1}   \frac{  dv_d(x)}{ (1 - |x|^2)^{\alpha_k +2d} }
\\
\le   \frac{ e^{-2 kd}}{  c_1^{\alpha_k + 2d} e^{-k(\alpha_k + 2d)}} \int_{\A_k(o)}\| F(x) \|^2  (1 - |x|^2)^{\alpha_k + d-1}   dv_d(x)
\\ \le  \frac{k e^{\sqrt{k}}}{c_1^{\alpha_k + 2d}} \cdot \alpha_k^2  \int_{\D_d}\| F(x) \|^2  (1 - |x|^2)^{\alpha_k + d-1}   dv_d(x). 
\end{multline*}
Since $\lim_k c_1^{\alpha_k} = 1$, 
the assumption  \eqref{alpha-to-zero} implies that there exists $C>0$ such that 
\[
e^{-2 kd} \int_{\A_k(o)}\| F(x) \|^2 d\mu_{\D_d}(x) \le  C \cdot  \frac{k e^{\sqrt{k}}}{c_1^{2d}}=: \Lambda(k). 
\]
This completes the proof of the lemma. 
\end{proof}

\begin{lemma}\label{lem-berg-mvp}
 An  $\MM$-harmonic function $F: \D_d \rightarrow \HH$ is tempered provided that
\begin{align}\label{integrable-cond}
\int_{\D_d} \| F(w)\|^2 (1- |w|^2)^{d -1}  dv_d(w) <\infty. 
\end{align}
\end{lemma}

\begin{proof}
The condition \eqref{integrable-cond} clearly implies the condition \eqref{alpha-to-zero}. 
\end{proof}

\begin{lemma}\label{prop-mgc}
 Let  $\Theta: \N \rightarrow \R_{+}$ be any function  such that $\lim_{k \to \infty} \Theta(k)  =0$. Then an  $\MM$-harmonic function    $F: \D_d \rightarrow \HH$ is tempered provided that 
\begin{align}\label{mean-cond-in-MVP}
\int_{\A_k(o)}\| F(x)\|^2 d\mu_{\D_d}(x)  \le   \Theta(k )  \cdot (k+1)   \cdot e^{2 kd} \quad \text{for all $k \in \N$.}
\end{align}
\end{lemma}

\begin{proof}
By Remark \ref{rem-eq-temp},  it suffices to show the limit equality \eqref{global-growth-f}.  By \eqref{mean-cond-in-MVP},  we have  
\begin{align*}
 \int_{\D_d} e^{-2s d_B(o, x)}  \| F(x)\|^2   d\mu_{\D_d}(x)  
&\le \sum_{k = 0}^\infty e^{-2s k}   \! \! \int_{ \A_k(o)} \| F(x)\|^2   d\mu_{\D_d}(x)    
\le  \sum_{k = 0}^\infty e^{-2(s - d) k}  \Theta(k) (k+1).
\end{align*}
It is easy to see that the assumption $\lim_{k\to\infty} \Theta(k) = 0$ implies 
\[
\lim_{s\to d^{+}} \frac{\sum_{k = 0}^\infty e^{-2(s - d) k} \cdot \Theta(k) \cdot (k+1)}{\sum_{k = 0}^\infty e^{-2(s - d) k} \cdot (k+1)}  = 0.
\]
By direct computation, we have 
\[
 \sum_{k = 0}^\infty e^{-2(s - d) k} \cdot (k+1) = \frac{1}{(1- e^{-2 (s-d)})^2} \an  \lim_{s\to d^{+}}  (s-d)^2 \sum_{k = 0}^\infty e^{-2(s - d) k} \cdot (k+1) = \frac{1}{4}.
\]
It follows that 
\[
\lim_{s\to d^{+}} (s-d)^2 \int_{\D_d} e^{-2s d_B(o, x)}  \| F(x)\|^2   d\mu_{\D_d}(x)= 0.
\]
Combining the above equality with \eqref{lem-s-d-goal}, we obtain the desired  limit equality \eqref{global-growth-f}.   
\end{proof}

\begin{corollary}\label{cor-in-MVP}
Let  $\Theta: \R_{+} \rightarrow \R_{+}$ be any  function with $\lim_{t\to\infty} \Theta(t)  = 0$. Then 
an  $\MM$-harmonic function $F: \D_d \rightarrow \HH$ is tempered provided that
\[
\| F(z)\|^2 \le  \Theta\Big( \frac{1}{1-|z|^2} \Big)  \cdot    \frac{1}{(1 - |z|^2)^d} \log\Big(\frac{2}{1- |z|^2}\Big)  \quad \text{for all $z \in \D_d$}.
\]
\end{corollary}

\begin{proof}
For any $k\in \N$ and  for any $z\in \A_k(o)$, we have 
$
e^k/4 \le (1-|z|^2)^{-1} \le  e^{k}.
$
 Note also that $\mu_{\D_d}(\A_k(o)) \le C_d \cdot e^{dk}$, where $C_d>0$  depends only on $d$. Thus Corollary   \ref{cor-in-MVP} follows from Lemma~\ref{prop-mgc}.
\end{proof}

\begin{lemma}\label{prop-sharp-temp}
 Let $F: \D_d \rightarrow \HH$ be a tempered  $\MM$-harmonic  function. Assume that there exist constants $C>0, \alpha >0$ such that  
\begin{align}\label{bdd-e-a}
\int_{\D_d} e^{-2 \varepsilon d_B(x, o)} \| F(x)\|^2 e^{-2d \cdot d_B(x, o)} d\mu_{\D_d}(x) \le C  \varepsilon^{-2}\cdot \varepsilon^\alpha \quad \text{for any $\varepsilon\in (0, 1)$}.
\end{align}
Then $F$ is sharply tempered. More generally, if in the upper estimate \eqref{bdd-e-a}, the term $\varepsilon^\alpha$ is replaced by $\log^{-1- \alpha} ( 1/\varepsilon)$, then  $F$ is also sharply tempered.
\end{lemma}

\begin{proof}
Assume  \eqref{bdd-e-a}, then take $\varepsilon_n = n^{-2/\alpha}$, both conditions \eqref{e-n-gap} and  \eqref{gap-sum} are satisfied. Now assume that in \eqref{bdd-e-a}, the term $\varepsilon^\alpha$ is replaced by $\log^{-1- \alpha} ( 1/\varepsilon)$,  then  take $\varepsilon_n = e^{- n^\gamma}$ with $(1 + \alpha)^{-1} < \gamma < 1$, both conditions \eqref{e-n-gap} and  \eqref{gap-sum} are satisfied.
\end{proof}

\subsection{Interpolation of a fixed tempered $\MM$-harmonic function}

 We adapt  the definition \eqref{def-g-k} to a function    $F:\D_d\rightarrow\HH$:
\[
g_X^{(k)}(s, z; F): = \sum_{x\in X\cap \A_k(z)} e^{-s d_B(z, x)} F(x).
\]

\begin{lemma}\label{prop-L2-L1-sum}
Assume that $F\in \sexp(\D_d, \HH)$.  Then for any $s> d$ and any relatively compact subset $D\subset \D_d$, we have
\begin{align}\label{lem-L2-L1-sum-goal-1}
\sum_{k=0}^\infty   \sup_{z\in D} \Big\{  \E_{\PP_d} \Big(\|  g_X^{(k)}(s, z; F)  \|^2\Big) \Big \}^{1/2} <\infty.
\end{align}
In particular,   for $\PP_d$-almost every $X\in \Conf(\D_d)$,  we have 
\begin{align}\label{lem-L2-L1-sum-goal-2}
\sum_{k=0}^\infty       \|  g_X^{(k)}(s, z; F)  \| <\infty \quad \text{for Lebesgue almost every $z\in \D_d$}.
\end{align}
\end{lemma}

Therefore, by Lemmas  \ref{lem-sub-ex} and \ref{prop-L2-L1-sum}, fixing any $F\in \TT(\D_d, \HH)$ and any $s>d, z\in \D_d$, for $\PP_d$-almost every  $X\in \Conf(\D_d)$, we may define
\begin{align}\label{def-g-F}
g_X(s, z; F): = \sum_{k=0}^\infty g_X^{(k)}(s, z; F),
\end{align}
where the series converges absolutely in $\HH$.
Introduce also the following notation:
\begin{align*}
g_X^{(k)}(s, z): & = \sum_{x\in X\cap \A_k(z)} e^{-s d_B(z, x)} \an g_X(s, z): =  \sum_{x\in X} e^{-s d_B(z, x)},
\\
g_{\PP_d} (s):=& \E_{\PP_d}[g_X(s,z)]  = \int_{\D_d} e^{-sd_B(x, z)} d\mu_{\D_d}(x) = \int_{\D_d} e^{-sd_B(x, o)} d\mu_{\D_d}(x).
\end{align*}

\begin{lemma}\label{prop-var-up-bd-rep}
 Assume that $F\in \sexp(\D_d, \HH)$.  Then for any $s > d, z\in \D_d$, we have
\begin{align}\label{av-pv-f}
\E_{\PP_d} ( g_X(s, z; F)  ) =   F(z) \cdot g_{\PP_d}(s)
\end{align} 
and
\begin{equation}\label{prop-var-up-bd-rep-goal}
\Var_{\PP_{d}}   ( g_X(s, z; F) ) \le  2   \int_{\D_d} e^{-2s d_B(x, z)} \| F(x) \|^2 d\mu_{\D_d}(x).
\end{equation}
\end{lemma}

\begin{theorem}\label{thm-H-L}
Assume that $F \in \TT(\D_d, \HH)$. Then  for any relatively compact subset $D\subset \D_d$, we have 
\begin{align}\label{unif-z-HL}
\lim_{s\to d^{+}} \sup_{z\in D} \E_{\PP_d} \Big( \Big\| \frac{g_X(s, z; F)}{g_{\PP_d}(s)} - F(z)\Big\|^2 \Big)  = 0.
\end{align}
\end{theorem}

\begin{proposition}\label{prop-limsup}
Assume that $F \in \TT(\D_d, \HH)$.  Let  $(s_n)_{n\ge 1}$ be a sequence in $(d, \infty)$ converging to $d$ and satisfying 
\begin{align}\label{fast-to-cri}
 \sum_{n=1}^\infty     \frac{1}{g_{\PP_d}(s_n)^2 } \int_{\D_d}  e^{-2s_n d_B(x, o)}   \| F(x)\|^2   d\mu_{\D_d}(x)  <\infty.
\end{align}
 Then for $\PP_d$-almost every $X\in \Conf(\D_d)$,  the equality 
\begin{align}\label{int-D-limsup}
   \int_D \limsup_{n\to \infty}  \Big\| \frac{g_X(s_n, z; F)}{g_{\PP_d}(s_n)} - F(z)\Big\|^2  dv_d(z) = 0
\end{align}
 holds for any relatively compact subset $D\subset \D_d$ and  moreover,  
\begin{align}\label{thm-H-L-goal}
 \lim_{n\to\infty}   \Big\| \frac{g_X(s_n, z; F)}{g_X(s_n, z)} - F(z)\Big\| = 0 \quad \text{for Lebesgue almost every  $z\in \D_d$.}
\end{align}
\end{proposition}

For pluriharmonic functions, we obtain a necessary and sufficient condition such that the Patterson-Sullivan interpolation formula holds in the sense of $L^2$-mean convergence.

\begin{proposition}\label{prop-pl-2side} 
There exists a constant $c>0$ depending only on $d$ such that for any pluriharmonic function $F:\D_d\rightarrow \HH$ with sub-exponential mean-growth, for  all $s\in (d, 2d]$ and all $z\in \D_d$, we have 
\[
\Var_{\PP_{d}}(g_X(s, z; F)) \ge  c \int_{\D_d}   e^{-2s d_B(x,z)}\| F(x)\|^2 d\mu_{\D_d}(x). 
\]
\end{proposition}

\begin{theorem}[Necessary and sufficient condition for the interpolation of pluriharmonic functions]\label{thm-pl-2side}
Assume that  $F: \D_d\rightarrow \HH$ is a pluriharmonic function with  sub-exponential mean-growth.  Then the following limit equality 
\begin{align}\label{mean-cv-PS-high}
\lim_{s\to d^{+}}  \E_{\PP_d} \Big(\Big\|  \frac{g_X(s, z; F)}{g_{\PP_d}(s)} - F(z)\Big\|^2\Big) = 0
\end{align}
holds for a fixed point $z\in \D_d$  if and only if 
\[
F\in \TT_{ph}(\D_d, \HH).
\] 
Moreover, for any $F\in \TT_{ph}(\D_d, \HH)$, the convergence \eqref{mean-cv-PS-high} holds locally uniformly on $z\in \D_d$.  
\end{theorem}

\subsection{Simultaneous uniform interpolation for weighted Bergman spaces}

The following radial weight on $\D_d$ is Bergman-admissible and is essential for us. 
\[
W_{\mathrm{cr}}(z)  =  \frac{1}{(1 - |z|^2) \log^{2} \big(\frac{4}{1 - |z|^2}\big)} \quad \text{for all $z \in \D_d$}. 
\]
By the discussion in \S\ref{sec-w-b},  any function $W\in L^1(\D, dA)$ with $W(z)\ge 0$ and
\begin{align}\label{def-ge-W}
\lim_{|z|\to 1^{-}}\frac{W(z)}{W_{\mathrm{cr}}(z)}  = \infty
\end{align}
is a Bergman-admissible weight on $\D_d$. Weights satisfying \eqref{def-ge-W}  will be called {\it super-critical}.
Note that  {\it super-critical weights need not be radial}.

For a Bergman-admissible weight $W$ on $\D_d$, we denote by $K_W$ the reproducing kernel of $A^2(\D_d, W)$. Define an $\MM$-harmonic function $F_W: \D_d \rightarrow A^2(\D_d, W)$ by
\begin{align}\label{def-F-W}
F_W(w) : = K_W(\cdot, w). 
\end{align}

\begin{lemma}\label{lem-sc-w-g}
Let $W$ be a super-critical weight on $\D_d$. Then $F_W$  is  tempered. 
\end{lemma}
\begin{remark}\label{rem-sc-w-g}
For a super-critical weight $W$ on $\D_d$, the function $F_W$ need  not satisfy the condition \eqref{integrable-cond}. 
\end{remark}

Therefore, if $W$ is a super-critical weight on $\D_d$,  then  by Lemmas~\ref{prop-L2-L1-sum} and  \ref{lem-sc-w-g},  for any fixed $s>d, z\in \D_d$,  for $\PP_d$-almost every $X\in \Conf(\D_d)$, the series 
\[
g_X(s, z; F_W) = \sum_{k=0}^\infty g_X^{(k)}(s, z; F_W)
\]
is absolutely convergent in $A^2(\D_d, W)$ and consequently, we may define  $g_X(s, z; f)$, simultaneously for all $f\in A^2(\D_d, W)$, by 
\[
g_X(s, z; f) := \langle f, g_X(s, z; F_W)\rangle_{A^2(\D_d, W)} = \sum_{k = 0}^\infty g_X^{(k)} (s, z; f). 
\]

\begin{theorem}\label{thm-wb}
Let $W$ be a super-critical weight on $\D_d$.   Then   there exists a sequence $(s_n)_{n\ge 1}$ in $(d, \infty)$ converging to $d$ such that  for $\PP_{d}$-almost every   $X\in \Conf(\D_d)$, we have  that for Lebesgue almost every $z\in \D_d$, 
\[
 f(z) =   \lim_{n\to\infty}  \frac{g_X(s_n, z; f)}{g_X(s_n, z)}  \quad \text{simultaneously for all $f\in A^2(\D_d, W)$,}
\]
where the convergence is uniform for $f$ in the unit ball  of $A^2(\D_d, W)$.
\end{theorem}

\subsection{Simultaneous uniform interpolation for weighted harmonic Hardy spaces}
Recall the  Poisson transformation \eqref{def-cb-poi}. 
For any  Borel probability measure $\mu$ on $\Sph_d$, set
\begin{align}\label{def-cm-hardy}
\mathcal{H}^2(\D_d; \mu) : &= \left\{f: \D_d\rightarrow \C\Big|  f = P^b[h\mu], \, h \in L^2(\Sph_d, \mu)\right\}.
\end{align}

\begin{lemma}\label{lem-g-uncond-cb}
Let $\mu$ be any Borel probability measure on $\Sph_d$. Then for $\PP_{d}$-almost every $X\in \Conf(\D_d)$,  simultaneously for  all $ f  \in \mathcal{H}^2(\D_d; \mu)$, all $z\in \D_d$ and all $s >d$, we have   
\[
\sum_{x \in X}   e^{-s d_B(x, z)} |f(x)|<\infty.
\]
\end{lemma}

Therefore, fixing a Borel probability measure $\mu$ on $\Sph_d$,  for $\PP_{d}$-almost every $X\in \Conf(\D_d)$, we can define   $g_X(s, z; f)$ simultaneously for all $f\in \mathcal{H}^2(\D_d, \mu)$, all $s>d$  and all $z\in \D_d$ by 
\[
g_X(s, z; f)= \sum_{x\in X}e^{-s d_B(x,z)} f(x).
\]

\begin{theorem}\label{thm-poi}
Let $\mu$ be any Borel probability measure on $\Sph_d$. Then for $\PP_{d}$-almost every  $X\in \Conf(\D_d)$,  simultaneously for  all $ f  \in \mathcal{H}^2(\D_d; \mu)$ and all $z\in \D_d$, we have 
\[
f(z) =    \lim_{s\to d^{+}}  \frac{g_X(s,z;f)}{g_X(s,z)},
\]
where the convergence is uniform for $f\in \{ P^b[h\mu]: \| h\|_{L^2(\mu)} \le 1\}$ and locally uniformly on $z\in \D_d$. 
\end{theorem}

\subsection{Proofs of Lemma \ref{prop-L2-L1-sum} and  Lemma \ref{prop-var-up-bd-rep}}

\begin{lemma}\label{lem-ball-av-M}
Let $F: \D_d \rightarrow \HH$ be $\MM$-harmonic.   Then for any $s> 0, R>0$ and $z\in \D_d$, 
\[
\E_{\PP_{d}} \Big( \sum_{x\in X \cap B(z, R)}  e^{-s d_B(z, x)}  F(x) \Big) =F(z) \cdot \E_{\PP_{d}} \Big( \sum_{x\in X\cap B(z, R)}  e^{-s d_B(z, x)} \Big).
\]
\end{lemma}

\begin{proof}
The identity \eqref{int-rep-exp} implies that for any $s>0, t>0, R>0$, we have 
\[
e^{-s t} \mathds{1} ( t < R) = \int_0^R s e^{-s r} \mathds{1}(t< r) dr + \int_R^\infty s e^{-s r} \mathds{1}(t < R) dr.
\]
This identity and the mean value property equality  \eqref{id-cb-mvp} together  imply 
\begin{multline*}
\E_{\PP_d} \Big( \sum_{x\in X \cap B(z,R)}  e^{-s d_B(z, x)}  F(x) \Big)  = \int_{\D_d} e^{-s d_B(z, x)} \mathds{1}( d_B(z, x) < R)  F(x) d\mu_{\D_d}(x)  
\\
 = \int_{\D_d}  F(x) d\mu_{\D_d}(x) \Big[ \int_0^R s e^{-s r} \mathds{1}( d_B(z, x) <r)  dr    + \int_R^\infty s e^{-s r} \mathds{1} (d_B(z, x) < R)  dr\Big]
\\
 = \int_0^R s e^{-s r} dr \int_{B(z, r)}   F(x)  d\mu_{\D_d}(x)    +  \int_R^\infty s e^{-s r}  dr \int_{B(z, R)} F(x) d\mu_{\D_d}(x)
\\
 = F(z) \cdot \Big(  \int_0^R s e^{-s r} \mu_{\D_d}(B(z, r))  dr +\int_R^\infty s e^{-sr}    \mu_{\D_d}( B(z, R)) dr\Big).
\end{multline*}
The same compution applies to $F\equiv 1$ yields a similar equality and  we complete the whole proof  by comparing these two equalities for $F$ and the constant function $1$. 
\end{proof}

\begin{proof}[Proof of Lemma \ref{prop-L2-L1-sum}]
Fix $s> d$. Let $D\subset \D_d$ be a relatively compact subset.   Let $N_D \in \N$ be the smallest integer with 
\[
N\ge \sup_{z\in D}d_B(z, o).
\]
 Then for any integer $k \ge N_D$ and $z\in D$, we have 
\[
\A_k(z) \subset \bigcup_{\ell =0}^{2N_D}  \A_{k- N_D + \ell} (o).
\] Replacing the sub-exponential function $\Lambda(k)$ by
\[
\widetilde{\Lambda}(k): = \sup_{0\le n\le k} \Lambda(n)
\] 
if necessary, 
we may assume that $\Lambda$ is non-decreasing. Therefore,  by Lemma \ref{lem-var-sob} and \eqref{sub-exp-correction}, for any integer $k \ge N_D$ and $z\in D$, we have
\begin{multline*}
 \Var_{\PP_d} ( g_X^{(k)}(s, z; F))  
\le  2 \int_{\A_k(z)}  e^{-2s d_B(x, z)} \| F(x)\|^2 d\mu_{\D_d}(x) 
\le 2 e^{-2s k} \int_{\A_k(z)} \| F(x)\|^2 d\mu_{\D_d}(x) 
\\
\le 2 e^{-2s k} \sum_{\ell =0}^{2N_D} \int_{\A_{k-N_D+\ell}(o)} \|F (x)\|^2 d\mu_{\D_d}(x)
\le 2 e^{-2sk} \sum_{\ell =0}^{2N_D}    \Lambda ( k - N_D  + \ell ) e^{2 (k- N_D + \ell) d} \le
\\
\le 2(2N_D +1)  e^{2 d N_D }  \Lambda (k + N_D) e^{- 2 ( s- d) k } 
\le C e^{- (2s - 2d) k}   \Lambda(k+N_D), 
\end{multline*}
where $C>0$ depends only on $d$ and $D$. 
Then applying Lemma \ref{lem-ball-av-M}, we obtain
\begin{multline*}
 \{ \E_{\PP_d}( \|  g^{(k)}_X(s, z; F)  \|^2)\}^{1/2} =  \Big \{\Var_{\PP_d} ( g_X^{(k)}(s, z; F))    +   \| \E_{\PP_d}( g^{(k)}_X(s, z; F))  \|^2\Big\}^{1/2}\le
\\
\le  \sqrt{C} e^{- (s - d)k } \sqrt{ \Lambda(k+N_D)} +   \| F(z)\|   \cdot  \E_{\PP_d} (g_X^{(k)}(s, z) ).
\end{multline*}
Now since $s>d$ and $\Lambda$ is sub-exponential, we have 
\[
\sum_{k=0}^\infty e^{- (s - d)k } \sqrt{ \Lambda(k+N_D )} <\infty
\] and  
\[
\sum_{k= 0}^\infty  \E_{\PP_d} (g_X^{(k)}(s, z)  )  = \E_{\PP_d} (g_X(s, z)  )= \int_{\D_d} e^{-s d_B(z, x)} d\mu_{\D_d}(x) =  \int_{\D_d} e^{-s d_B(o, x)} d\mu_{\D_d}(x)<\infty. 
\]
The desired convergence \eqref{lem-L2-L1-sum-goal-1} follows immediately.

Finally, the  convergence \eqref{lem-L2-L1-sum-goal-1} implies 
\[
\sum_{k=0}^\infty   \int_D   \E_{\PP_d} \Big(\|  g_X^{(k)}(s, z; F)  \|\Big)  dv_d(z) <\infty.
\]
 It follows that the convergence  \eqref{lem-L2-L1-sum-goal-2} holds for $\PP_d$-almost every $X\in \Conf(\D_d)$ and Lebesgue almost every $z\in D$. We then complete the proof of   the lemma by  taking an exhausting sequence $(D_k)_{k\ge 1}$ of relatively compact subsets of $\D_d$.
\end{proof}

\begin{proof}[Proof of Lemma \ref{prop-var-up-bd-rep}]
Clearly, Lemmas \ref{prop-L2-L1-sum} and  \ref{lem-ball-av-M}  imply the desired equality \eqref{av-pv-f}.  Now for any positive integer $N$, by Lemmas \ref{lem-var-sob} and  \ref{lem-ball-av-M},  we have 
\begin{multline*}
  \E_{\PP_d} \Big( \Big\| \sum_{k = 0}^{N-1} g_X^{(k)}(s, z; F)   \Big\|^2 \Big)  =   \E_{\PP_d}  \Big(\Big \| \sum_{x\in X \cap B(z, N)}  e^{-s d_B(x, z)} F(x) \Big\|^2\Big) =
\\
 =   \Var_{\PP_d} \Big(    \sum_{x\in X \cap B(z, N)}  e^{-s d_B(x, z)} F(x) \Big)  +   \Big\|     \E_{\PP_d} \Big( \sum_{x\in X \cap B(z, N)}  e^{-s d_B(x, z)} F(x)  \Big) \Big\|^2\le
\\
\le  2 \int_{B(z, N)} e^{-2s d_B(x, z)} \| F(x) \|^2 d\mu_{\D_d}(x) + \| F(z) \|^2  \Big( \int_{B(z, N)}  e^{-s d_B(x,z)} d\mu_{\D_d}(x) \Big)^2\le
\\
\le 2  \int_{\D_d} e^{-2s d_B(x, z)} \| F(x) \|^2 d\mu_{\D_d}(x) + \| F(z) \|^2  g_{\PP_d}(s)^2.
\end{multline*}
Therefore, by applying \eqref{lem-L2-L1-sum-goal-1} and  \eqref{av-pv-f},   we obtain 
\begin{multline*}
\Var_{\PP_d}  ( g_X(s, z; F))  =  \E_{\PP_d} ( \|  g_X(s, z; F) \|^2)  -  \|  \E_{\PP_d} ( g_X(s, z; F) ) \|^2  = 
\\
= \lim_{N\to\infty} \E_{\PP_d} \Big( \Big \| \sum_{k = 0}^{N-1}   g_X^{(k)}(s, z; F)  \Big\|^2\Big) - \| F(z) \|^2  g_{\PP_d}(s)^2  \le  2   \int_{\D_d} e^{-2s d_B(x, z)} \| F(x) \|^2 d\mu_{\D_d}(x).
\end{multline*}
This is the desired inequality \eqref{prop-var-up-bd-rep-goal}. 
\end{proof}

\subsection{Proof of Theorem \ref{thm-H-L} and Proposition \ref{prop-limsup}}\label{sec-proof-H-L}

In what follows,  set
\begin{align}\label{notation-T-sigma-3}
 R_X(s, z; F) : = \frac{g_X(s, z; F)}{g_X(s,z)} \an  \underline{R}_X(s,z; F) : = \frac{g_X(s,z; F) }{\E_{\PP_d}[g_X(s,z)]}  =  \frac{g_X(s,z; F)}{g_{\PP_d}(s)}.
\end{align}

\begin{proof}[Proof of Theorem \ref{thm-H-L}]
Let $D\subset \D_d$ be a relatively compact subset. Then 
\[
C_D = \sup_{z\in D} d_B(z, o)<\infty. 
\]  Let $F: \D_d \rightarrow \HH$ be a tempered $\MM$-harmonic function. For any $s> d$, by \eqref{av-pv-f}, we have
$\E_{\PP_d} (  \underline{R}_X(s, z; F) ) = F(z)$. 
Whence by Lemma~\ref{prop-var-up-bd-rep},
\begin{align}\label{sup-D-var}
\begin{split}
\sup_{z\in D}\E_{\PP_d} ( \|   \underline{R}_X(s, z; F)   - F(z)  \|^2 ) & \le   \frac{2}{g_{\PP_d}(s)^2 }   \sup_{z\in D}\int_{\D_d}  e^{-2s d_B(x, z)} \| F(x)\|^2  d\mu_{\D_d}(x)
\\
& \le \frac{2 e^{2s C_D}}{g_{\PP_d}(s)^2 }  \int_{\D_d}  e^{-2s d_B(x,o)} \| F(x)\|^2  d\mu_{\D_d}(x).
\end{split} 
\end{align}
The desired relation \eqref{unif-z-HL}  now follows from \eqref{global-growth-f}. 
\end{proof}

\begin{proof}[Proof of Proposition \ref{prop-limsup}]
For proving \eqref{thm-H-L-goal},  we may assume that $F$ is not identically zero.  Fix any  sequence $(s_n)_{n\ge 1}$ in $(d, \infty)$ converging to $d$ and  satisfying  the condition \eqref{fast-to-cri}. 
By \eqref{fast-to-cri} and \eqref{sup-D-var}, we obtain 
\[
\sum_{n=1}^\infty  \E_{\PP_d} \Big( \int_D \|  \underline{R}_X(s_n, z; F)  - F(z)  \|^2 dv_d(z)\Big) < \infty. 
\]
It follows that for $\PP_d$-almost every $X\in \Conf(\D_d)$, we have the limit equality 
\begin{align}\label{av-ratio-limit}
  \int_D  \limsup_{n\to\infty}\left\|  \underline{R}_X(s_n, z; F) - F(z)  \right\|^2 dv_d(z) = 0. 
\end{align}
This is the desired relation \eqref{int-D-limsup}.

 Now since $F$ is not identically zero, we have
\[
\sup_{d< s<2d}  \frac{  \int_{\D_d}  e^{-2s d_B(x, o)}    d\mu_{\D_d}(x)}{  \int_{\D_d}  e^{-2s d_B(x, o)}  \|F (x)\|^2  d\mu_{\D_d}(x)}  \le \frac{  \int_{\D_d}  e^{-2 d d_B(x, o)}    d\mu_{\D_d}(x)}{ \int_{\D_d}  e^{- 4 d d_B(x, o)}  \|F (x)\|^2  d\mu_{\D_d}(x)} <\infty.
\]
Hence the convergence \eqref{fast-to-cri} and $\lim_n s_n  = d$ together imply that the condition \eqref{fast-to-cri} holds for the scalar function $F\equiv 1$. Therefore, we may  apply \eqref{av-ratio-limit} to the scalar function $F\equiv 1$ to obtain 
\begin{align}\label{av-or-not-1-onez}
  \int_D  \limsup_{n\to\infty} \left|  \frac{g_X(s_n, z)}{g_{\PP_d}(s_n)} - 1  \right|^2 dv_d(z) = 0 \quad \text{for $\PP_d$-almost every $X\in \Conf(\D_d)$}. 
\end{align}
By \eqref{av-ratio-limit} and   \eqref{av-or-not-1-onez},  for $\PP_d$-almost every $X\in \Conf(\D_d)$ and Lebesgue almost every $z\in D$, we have  
\[
\lim_{n\to\infty}  \|   R_X (s_n, z; F)- F(z)  \| =  \lim_{n\to\infty}  \left\|   \frac{g_{\PP_d}(s_n)} { g_X (s_n, z)} \cdot  \underline{R}_X(s_n, z; F)- F(z)  \right\| = 0.
\]
Finally, we complete the proof of Theorem \ref{thm-H-L} by taking an exhausting sequence $(D_k)_{k\ge 1}$ of relative compact subsets $D_k\subset \D_d$. 
\end{proof}

\subsection{Proofs of Proposition \ref{prop-pl-2side} and Theorem \ref{thm-pl-2side}}

\begin{proof}[Proof of Proposition \ref{prop-pl-2side}]
Fix $z\in \D_d$.  Recall the definition \eqref{inv-auto} of the M\"obius transformation $\varphi_z$ on $\D_d$. 
Assume that $F: \D_d\rightarrow \HH$ is pluriharmonic such that $F\in \sexp(\D_d, \HH)$ then so is $F\circ \varphi_z$.  By \cite[Thm. 4.4.9]{Rudin-ball}, it is easy to see that there exist two sequences $(\xi_n)_{n\in \N^d}$ and $(\eta_n)_{n\in \N^d\setminus \{0\}}$ of vectors in $\HH$ such that 
\[
F(\varphi_z (x)) = \sum_{n\in \N^d} \xi_n x^n + \sum_{n\in \N^d\setminus\{0\}} \eta_n \overline{x^n}, \quad \text{where $x\in \D_d$ and $x^n = x_1^{n_1} \cdots x_d^{n_d}$.}
\]
Then  by the conformal invariance of the measure $\mu_{\D_d}$, we have 
\begin{multline}\label{w-norm-F}
\int_{\D_d}   e^{-2s d_B(x,z)}\| F(x)\|^2 d\mu_{\D_d}(x) = \int_{\D_d} e^{-2s d_B(x,o)} \| F(\varphi_z(x))\|^2 d\mu_{\D_d}(x)
\\
  = \sum_{n\in \N^d} \|\xi_n\|^2 \int_{\D_d} |x^n|^2 e^{-2s d_B(x,o)}  d\mu_{\D_d}(x)  + \sum_{n\in \N^d\setminus\{0\}}  \|\eta_n\|^2 \int_{\D_d} |x^n|^2 e^{-2s d_B(x,o)}  d\mu_{\D_d}(x).
\end{multline}
Define a measure on $\D_d\times \D_d$ by 
\[
dM_2(x,y) : = |K_{\D_d}(x,y)|^2 dv_d(x)dv_d(y).
\]
Then $dM_2$ is invariant under the transformation $(x,y) \mapsto (\varphi_z(x), \varphi_z(y))$. Thus
by Lemma~\ref{lem-var-sob}, we have
\begin{multline*}
\Var_{\PP_d} (g_X(s, z; F)) = \frac{1}{2}\int_{\D_d}\int_{\D_d} \|F(x) e^{-s d_B(x,z)} - F(y)e^{-s d_B(y,z)} \|^2 dM_2(x,y)
\\
= \frac{1}{2}\int_{\D_d}\int_{\D_d} \Big\| \underbrace{ F(\varphi_z(x)) e^{-s d_B(x,o)} - F(\varphi_z(y))e^{-s d_B(y,o)} }_{\text{denoted $\delta(x,y)$}} \Big\|^2 dM_2(x,y). 
\end{multline*}
Set 
\begin{align}\label{def-F-family}
\mathcal{F} : = \{\xi_n x^n e^{-s d_B(x, o)}: n\in \N^d\} \cup\{\eta_n \overline{x^n} e^{-s d_B(x,o)}: n \in \N^d\setminus \{0\}\}. 
\end{align}
Denote  a general element in $\mathcal{F}$ by $E_s$, we have 
\[
\delta(x,y) = \sum_{E_s\in \mathcal{F}} (E_s(x) -E_s(y)).
\]
Now take any $\theta  =(\theta_1, \cdots, \theta_d)\in [0,2\pi)^d$, write 
$
D_\theta(x): = (x_1 e^{i\theta_1}, \cdots, x_d e^{i \theta_d}). 
$ 
Since $K_{\D_d}(x,y)  = K_{\D_d}(D_\theta(x), D_\theta(y))$, we have 
\[
\Var_{\PP_d} (g_X(s, z; F)) 
= \frac{1}{2}\int_{\D_d}\int_{\D_d}  \Big[\frac{1}{(2 \pi)^d}\int_{[0,2\pi)^d}\|\delta(D_\theta(x), D_\theta(y))\|^2 d\theta\Big] dM_2(x,y).
\]
It is easy to see that for any two distinct elements 
$
E_s, \widetilde{E_s}\in \mathcal{F}$,  we have 
\[
\frac{1}{(2 \pi)^d}\int_{[0,2\pi)} \Big\langle  E_s(D_\theta(x))-E_s(D_\theta(y)), \, \widetilde{E_s}(D_\theta(x)) - \widetilde{E_s}(D_\theta(y))\Big\rangle_{\HH} d\theta  = 0.
\]
Therefore, by using the equality $\| E_s(D_\theta(x)) - E_s(D_\theta(y))\| = \| E_s(x) - E_s(y)\|$, we have 
\begin{align*}
\frac{1}{(2 \pi)^d}\int_{[0,2\pi)^d}\|\delta(D_\theta(x), D_\theta(y))\|^2 d\theta 
 = \sum_{E_s\in \mathcal{F}} \|E_s(x)-E_s(y) \|^2.
\end{align*}
Hence 
\begin{align}\label{var-g-F}
\Var_{\PP_d} (g_X(s, z; F)) =  \frac{1}{2} \sum_{E_s\in \mathcal{F}} \underbrace{\int_{\D_d}\int_{\D_d} \| E_s(x)-E_s(y) \|^2 dM_2(x,y)}_{\text{denoted $I(E_s)$}}.
\end{align}

Let us now estimate the double integral $I(E_s)$ for $E_s\in \mathcal{F}$. 

{\flushleft \bf Claim I:} there is a  constant $c>0$ such that for any $s\in [d, 2d]$ and any $E_s\in \mathcal{F}$, we have 
\begin{align}\label{I-E-s}
I(E_s)\ge c \int_{\D_d} \| E_s(x)\|^2  d\mu_{\D_d}(x) =  c \int_{\D_d} \| E_s(x)\|^2 K_{\D_d}(x,x) dv_d(x). 
\end{align}

Assuming  Claim I, then by  \eqref{w-norm-F}, \eqref{var-g-F} and \eqref{I-E-s}, we obtain the desired inequality
\begin{multline*}
\Var_{\PP_d} (g_X(s, z; F)) \ge  \frac{c}{2} \sum_{E_s\in \mathcal{F}}\int_{\D_d} \| E_s(x)\|^2 d\mu_{\D_d}(x)  = \frac{c}{2} \int_{\D_d}   e^{-2s d_B(x,z)}\| F(x)\|^2 d\mu_{\D_d}(x).
\end{multline*}

It remains to prove Claim I.  Note that 
\[
I(E_s) = 2 \int_{\D_d} \| E_s(x)\|^2 K_{\D_d}(x,x)dv_d(x) - 2 \Re \int_{\D_d}\int_{\D_d}  \langle E_s(x), E_s(y) \rangle dM_2(x,y).
\]
Thus it suffices to show that there exists a constant $\gamma\in (0, 1)$ such that for any $s\in [d, 2d]$, we have
\begin{align}\label{re-less-2}
\Re \int_{\D_d}\int_{\D_d}  \langle E_s(x), E_s(y) \rangle dM_2(x,y) \le \gamma \cdot \int_{\D_d} \| E_s(x)\|^2 K_{\D_d}(x,x)dv_d(x).
\end{align}
Now for any $n\in \N^d$, set 
\begin{align*}
N(n,s): &=  \int_{\D_d}\int_{\D_d}  x^n \overline{y^n}  e^{-s d_B(x,o)} e^{-s d_B(y,o)} dM_2(x,y),
\\
D(n,s): &=\int_{\D_d} |x^n|^2 e^{-2s d_B(x,o)} K_{\D_d}(x,x)dv_d(x).
\end{align*}
By the definition \eqref{def-F-family} of the family $\mathcal{F}$, for proving \eqref{re-less-2}, it suffices to show 
\begin{align}\label{ratio-to-zero}
\sup_{n\in \N^d}  \sup_{s\in [d, 2d]}  \frac{\Re(N(n,s))}{D(n,s)}<1. 
\end{align}
Note that by Cauchy-Bunyakovsky-Schwarz inequality, it is easy to see that for any fixed $n\in\N^d$ and any $s\in [d, 2d]$, we have the strict inequality $\Re(N(n,s))< D(n,s)$. Therefore, by using the continuity on $s$, to prove  the inequality \eqref{ratio-to-zero}, it suffices to show 
\begin{align}\label{ratio-to-zero-limit}
\lim_{|n|\to\infty}  \sup_{s\in [d, 2d]}  \frac{\Re(N(n,s))}{D(n,s)} = 0 \quad \text{where $|n| = n_1 + \cdots + n_d.$}
\end{align}
By expanding $(1 - x\cdot \bar{y})^{-d -1}$, we may write $N(n,s)$ as follows:
\begin{multline*}
N(n,s) 
=   \int_{\D_d}\int_{\D_d}  x^n \overline{y^n}  \Big|\sum_{k=0}^\infty  \frac{\Gamma(k+d+1)}{k!\Gamma(d+1)}   (x\cdot \bar{y})^k\Big|^2   \frac{dv_d(x)dv_d(y)}{e^{s d_B(x,o)} e^{s d_B(y,o)}}
\\
=  \int_{\D_d}\int_{\D_d}   \Big[ \underbrace{\int_0^{2\pi} e^{i |n| \theta}  x^n \overline{y^n}  \Big|\sum_{k=0}^\infty  \frac{\Gamma(k+d+1)}{k!\Gamma(d+1)}   ( e^{in\theta} x\cdot \bar{y})^k\Big|^2 \frac{d\theta}{2\pi}}_{\text{denoted by $A(x,y)$}}\Big] \frac{dv_d(x)dv_d(y)}{e^{s d_B(x,o)} e^{s d_B(y,o)}}.
\end{multline*}
Clearly, we have 
\[
A(x,y) =  x^n \overline{y^n}\sum_{k = 0}^\infty  \frac{\Gamma(k+d+1)}{k!\Gamma(d+1)}  \frac{\Gamma(k+|n|+d+1)}{(k+|n|)!\Gamma(d+1)}  (x\cdot \bar{y})^k  \cdot (\bar{x}\cdot y)^{k+|n|}.
\]
By Cauchy-Bunyakovsky-Schwarz inequality, for any $n\in \N^d$ and $k\in \N$,  we have
\begin{multline*}
\Big|\int_{\D_d } \int_{\D_d } x^n \overline{y^n}   (x\cdot \bar{y})^k  \cdot (\bar{x}\cdot y)^{k+|n|} \frac{dv_d(x)dv_d(y)}{e^{s d_B(x,o)} e^{s d_B(y,o)}}\Big| 
\\
\le\Big( \int_{\D_d}\int_{\D_d} |x^n|^2     |x\cdot \bar{y}|^{2k+|n|}  \frac{dv_d(x)dv_d(y)}{e^{s d_B(x,o)} e^{s d_B(y,o)}}\Big)^{1/2} \Big( \int_{\D_d}\int_{\D_d} |y^n|^2     |x\cdot \bar{y}|^{2k+|n|}  \frac{dv_d(x)dv_d(y)}{e^{s d_B(x,o)} e^{s d_B(y,o)}}\Big)^{1/2}
\\
= \int_{\D_d}\int_{\D_d} |x^n|^2     |x\cdot \bar{y}|^{2k+|n|}  \frac{dv_d(x)dv_d(y)}{e^{s d_B(x,o)} e^{s d_B(y,o)}}
\\
 \le  \underbrace{ \int_{\D_d}\int_{\D_d} |x^n|^2     |x\cdot \bar{y}|^{2k+|n|}   (1 - |x|)^s (1 - |y|)^s  dv_d(x)dv_d(y)}_{\text{denoted by $B_s(k, n)$}}.
\end{multline*}
It is clear that 
\begin{multline*}
B_s(k,n) =  a_{k,n} \cdot\int_{\D_d}\int_{\D_d} |x|^{2|n|} \cdot  |x|^{2k+|n|} |y|^{2k+|n|}     (1 - |x|)^s (1 - |y|)^s  dv_d(x)dv_d(y)
\\
= a_{k,n} (2d)^2 \int_0^1 \int_0^1  r_1^{2d-1} r_2^{2d-1}  r_1^{2k+3|n|} r_2^{2k+|n|}     (1 - r_1)^s (1 - r_2)^s  dr_1dr_2
\\
\le 4d^2\cdot a_{k,n}  \frac{\Gamma(2k+3|n|+1)\Gamma(s+1)}{\Gamma(2k+3|n|+s+2)}\frac{\Gamma(2k+|n|+1)\Gamma(s+1)}{\Gamma(2k+|n|+s+2)},
\end{multline*}
where 
\[
a_{k,n}= \int_{\Sph_d} \int_{\Sph_d} |\zeta^n|^2       |\zeta\cdot \bar{\xi}|^{2k+|n|} d\sigma_{\Sph_d}(\zeta) d\sigma_{\Sph_d}(\xi)=  \frac{\Gamma(d)\Gamma(k+\frac{|n|}{2} +1)}{\Gamma(d+ k + \frac{|n|}{2})} \cdot \frac{\Gamma(d) n_1!\cdots n_d!}{\Gamma(d + |n|)}.
\]
The computation of $a_{k,n}$ relies on  \cite[section 1.4.5]{Rudin-ball} and \cite[Prop. 1.4.9]{Rudin-ball}.
Therefore,  
\[
\Re (N(n,s))  \le |N(n,s)| \le \sum_{k = 0}^\infty  \frac{\Gamma(k+d+1)}{k!\Gamma(d+1)}  \frac{\Gamma(k+|n|+d+1)}{(k+|n|)!\Gamma(d+1)} B_s(k,n).
\]
By the classical asymptotics for Gamma function, there exists a constant $C_d>0$ depending only on $d$ such that for all $s\in [d, 2d]$ and all $n\in\N^d$ with $|n|\ge 1$, we have 
\begin{align}\label{re-N-up}
\Re(N(n,s)) \le   C_d \cdot \frac{\Gamma(d) n_1!\cdots n_d!}{\Gamma(d + |n|)} \cdot \sum_{k = 0}^\infty  \frac{k^d}{ (k + |n|)^{2s+1}}.
\end{align}
On the other hand,  for $D(n,s)$, it is easy to see that 
\begin{multline*}
D(n,s) = \int_{\Sph_d}|\zeta^n|^2 d\sigma_{\Sph_d}(\zeta) \cdot \int_{\D_d} |x|^{2|n|} \Big(\frac{1-|x|}{1+|x|}\Big)^{2s} \frac{dv_d(x)}{(1 -|x|^2)^{d+1}} 
\\
\ge \frac{1}{2^{2s+d+1}}\int_{\Sph_d}|\zeta^n|^2 d\sigma_{\Sph_d}(\zeta) \cdot \int_{\D_d} |x|^{2|n|} (1-|x|)^{2s-d-1} dv_d(x)  
\\
= \frac{2d}{2^{2s+d+1}}\int_{\Sph_d}|\zeta^n|^2 d\sigma_{\Sph_d}(\zeta) \cdot \int_{0}^1 r^{2d-1+2|n|} (1-r)^{2s-d-1} dr 
\\
= \frac{2d}{2^{2s+d+1}}  \frac{\Gamma(d) n_1!\cdots n_d!}{\Gamma(d + |n|)} \frac{\Gamma(2d+2|n|) \Gamma(2s-d)}{\Gamma(2|n|+2s+d)}.
\end{multline*}
Hence there exists a constant $C_d'>0$ such that for all $s\in [d, 2d]$ and any $n\in\N^d$ with $|n|\ge 1$, we have 
\begin{align}\label{D-low}
D(n,s) \ge C_d'\cdot \frac{\Gamma(d) n_1!\cdots n_d!}{\Gamma(d + |n|)} \cdot  \frac{1}{|n|^{2s-d}}. 
\end{align}
Combining \eqref{re-N-up} and \eqref{D-low}, for all $s\in [d, 2d]$ and all $n\in \N^d$ with $|n|\ge 1$, we obtain  
\[
\frac{\Re(N(n,s))}{D(n,s)} \le  \frac{C_d}{C_d'} \sum_{k = 0}^\infty  \frac{k^d \cdot |n|^{2s-d}}{ (k + |n|)^{2s+1}}.
\]
Finally, it suffices to show 
\begin{align}\label{sup-to-zero}
\lim_{m\to\infty}\sup_{s\in [d,2d]}\sum_{k=1}^\infty \frac{k^d m^{2s-d}}{(k+m)^{2s+1}} = 0.
\end{align}
Indeed, write
\begin{align}\label{sum-k-m}
\sum_{k=1}^\infty \frac{k^d m^{2s-d}}{(k+m)^{2s+1}}  = \frac{1}{m} \sum_{k=1}^\infty \frac{(k/m)^d}{(1+k/m)^{2s+1}} \le  \frac{1}{m} \sum_{k=1}^\infty \frac{(k/m)^d}{(1+k/m)^{2d+1}}.
\end{align}
Set $H(t)= t^d/(1 + t)^{2d+1}$. Note that $H$  is increasing on a finite interval $[0, t_0]$ and then decreasing on the interval $[t_0, \infty)$.   It follows that 
\begin{align}\label{sum-to-int}
\sum_{k=1}^\infty \frac{(k/m)^d}{(1+k/m)^{2d+1}}\le  \max_{t\in \R_{+}} H(t) + 2\int_{0}^\infty H(t)dt. 
\end{align}
The desired limit equation  \eqref{sup-to-zero} follows immediately from \eqref{sum-k-m} and \eqref{sum-to-int}. 
\end{proof}

\begin{proof}[Proof of Theorem \ref{thm-pl-2side}]
Theorem \ref{thm-pl-2side} is an immediate consequence of Lemma \ref{prop-var-up-bd-rep} and Proposition \ref{prop-pl-2side}. 
\end{proof}

\subsection{Proofs of Lemma \ref{lem-sc-w-g} and Theorem \ref{thm-wb}}

\begin{proposition}\label{prop-w-cb}
Let  $W$ be a super-critical weight on $\D_d$. Then  there exists a function  $\Theta: \R_{+} \rightarrow \R_{+}$ with $\lim_{t\to\infty} \Theta(t)=0$ such that
\[
K_{W}(z, z)  \le \Theta\big(\frac{1}{1 - |z|^2}\big) \cdot \frac{1}{(1  - |z|^2)^d} \log \big( \frac{2}{1 - |z|^2}\big) \quad \text{for all $z\in \D_d$}.
\]
\end{proposition}

The proof of  Proposition  \ref{prop-w-cb} is  postponed  to \S \ref{sec-kernel-es}. 

\begin{proof}[Proof of Lemma \ref{lem-sc-w-g}]
 By the reproducing property of the kernel $K_W$, we have
\[
\| F_W(z)\|_{A^2(\D_d, W)}^2 = \langle K_W(\cdot, z), K_W(\cdot, z)\rangle_{A^2(\D_d, W)} = K_W(z,z) \quad \text{for all $z\in \D_d$.}
\]
Thus Lemma \ref{lem-sc-w-g} follows from Corollary \ref{cor-in-MVP} and Proposition \ref{prop-w-cb}. 
\end{proof}

\begin{proof}[Proof of Theorem \ref{thm-wb}]
Theorem \ref{thm-H-L} and Lemma \ref{lem-sc-w-g} together imply that 
 there exists a sequence $(s_n)_{n\ge 1}$ in $(d, \infty)$ converging to $d$ such that if we fix a countable dense subset $\mathcal{D}\subset \D_d$,
 then for $\PP_d$-almost every  $X\in \Conf(\D_d)$, we have 
\[
 \lim_{n\to\infty}   \Big\| \frac{g_X(s_n, z; F_W)}{g_X(s_n, z)} - F_W(z)\Big\|_{A^2(\D_d, W)} = 0 \quad \text{for all $z\in \mathcal{D}$.}
\]
We complete the proof of Theorem \ref{thm-wb} by noting that for all $f\in A^2(\D_d, W)$, $z\in \D_d$,
\[
g_X(s, z; f) = \langle f,  g_X(s, z; F_W)\rangle_{A^2(\D_d, W)} \an f(z) = \langle f, F_W(z)\rangle_{A^2(\D_d, W)}
\]
and
\[
\Big\| \frac{g_X(s_n, z; F_W)}{g_X(s_n, z)} - F_W(z)\Big\|_{A^2(\D_d, W)} = \sup_{f\in \mathcal{B}(W)}  \Big|\Big\langle f, \,\, \frac{g_X(s_n, z; F_W)}{g_X(s_n, z)} - F_W(z) \Big\rangle_{A^2(\D_d, W)}\Big|,
\]
where $\mathcal{B}(W)$ is the unit ball of $A^2(\D_d, W)$.   
\end{proof}

\subsection{Proofs of Lemma \ref{lem-g-uncond-cb} and Theorem \ref{thm-poi}}
  In this subsection, assume that the Hilbert space $\HH$ is over the field $\R$ and is of the form 
\[
\HH   = L_\R^2(\nu) = L_\R^2(\Sigma, \nu),
\]
where $\nu$ is a Borel probability measure on a metric complete seperable space $\Sigma$ and $L_\R^2(\nu)$ is the space of real-valued square-integrable functions on $(\Sigma, \nu)$.  The subset of non-negative functions in $L_\R^2(\nu)$ will be denoted by  $L_\R^2(\nu)_{+}$.

Recall that a convergent series  $\sum_{n=1}^\infty v_n$ in $L_\R^2(\nu)$  is said to {\it converge unconditionally} if its sum does not change under any reordering of the terms.

\begin{theorem}\label{thm-n-sub}
 Let $F: \D_d\rightarrow L_\R^2(\nu)_{+}$ be a  sharply tempered $\MM$-harmonic function. Let $D\subset \D_d$ be a relatively compact subset.  Then  $\PP_d$-almost every $X\in \Conf(\D_d)$ satisfies:
\begin{itemize}
\item[(i)] The series $g_X(s,z) = \sum_{x\in X} e^{-s d_B(x,z)}$ converges for all $s>d$ and all $z\in \D_d$.
\item[(ii)] The series 
$
g_X(s, z; F) =   \sum_{x\in  X} e^{-s d_B(x, z)} F(x)  
$
converges unconditionally in $L_\R^2(\nu)$ for all $s>d$ and all $z\in \D_d$.  
\item[(iii)] For all $z\in \D_d$, we have 
\begin{align}\label{thm-n-sub-goal}
 \lim_{s \to d^{+}}\Big \|  \frac{g_X(s,z; F) }{g_X(s,z)} - F(z)\Big\|_{L_\R^2(\nu)} = 0.
\end{align}
\item[(iv)] The following local uniform convergence holds: 
\[
\lim_{s\to d^{+}}   \sup_{z\in D}\Big \|  \frac{g_X(s,z; F) }{g_X(s,z)} - F(z)\Big\|_{L_\R^2(\nu)} = 0.
\]
\end{itemize}
\end{theorem}

The proof of Theorem \ref{thm-n-sub} is postponed to the end of this subsection.  We now show how to derive Lemma \ref{lem-g-uncond-cb} and Theorem \ref{thm-poi} from Theorem \ref{thm-n-sub}. Recall the Poisson-Szeg\H{o} kernel $P^b(w, \zeta)$ defined in \eqref{def-poi-ker}.  For any Borel probability measure $\mu$ on $\Sph_d$, we define a function $F_\mu: \D_d \rightarrow L^2_\R(\Sph_d, \mu)$ by 
\begin{align}\label{def-F-mu}
F_\mu(w): = P^b(w, \cdot) \in L_\R^2(\Sph_d, \mu).
\end{align}
Note for any $\zeta\in \Sph_d$, the function $w\mapsto P^b(w, \zeta)$ is $\MM$-harmonic. Therefore, $F_\mu$ defined in \eqref{def-F-mu} is a vector-valued $\MM$-harmonic function. The following Lemma \ref{lem-cmvp-cb} shows that $F_\mu$ is sharply tempered in the sense of Definition \ref{def-cmvp-sharp}.

\begin{lemma}\label{lem-cmvp-cb}
Let $\mu$ be a Borel probability measure on $\Sph_d$. Then  there exists a constant $C>0$, such that for any $k\in \N$, we have 
\begin{align}\label{mg-poi-cb}
\int_{\A_k(o)} \| P^b(w, \cdot) \|_{L^2(\mu)}^2 d\mu_{\D_d}(w) \le  C e^{2kd}. 
\end{align}
Moreover, there exists a constant $C'>0$ such that  for any $\varepsilon\in (0, 1)$, we have   
\begin{align}\label{lap-es-cb}
\int_{\D_d} e^{-2 \varepsilon d_B(w, o)} \| P^b(w, \cdot) \|_{L^2(\mu)}^2 e^{-2 d\cdot d_B(w, o)} d\mu_{\D_d}(w)  \le \frac{C'}{\varepsilon}. 
\end{align}
In particular, the $\MM$-harmonic function $F_\mu$ defined by \eqref{def-F-mu} is sharply tempered. 
\end{lemma}

\begin{proof}
Let $\mathcal{U}_d$ be the compact group of $d\times d$ unitary matrices equipped with  the normalized Haar measure $dU$. For any  $k \in \N$ and any $U\in \mathcal{U}_d$,  since $\mu_{\D_d}$ is radial,  we have 
\[
\int_{\A_k(o)} \| P^b(w, \cdot) \|_{L^2(\mu)}^2  d\mu_{\D_d}(w) = \int_{\A_k(o)} \| P^b(Uw, \cdot) \|_{L^2(\mu)}^2  d\mu_{\D_d}(w).
\]
Therefore, 
\begin{multline*}
\int_{\A_k(o)} \| P^b(w, \cdot) \|_{L^2(\mu)}^2  d\mu_{\D_d}(w) = \int_{\A_k(o)} \int_{\mathcal{U}_d}\| P^b(Uw, \cdot) \|_{L^2(\mu)}^2  dU d\mu_{\D_d}(w) =
\\
= 
\int_{\A_k(o)} \int_{\Sph_d}  \Big[  \int_{\mathcal{U}_d}\left(  \frac{(1-|U w|^2)^d}{|1 - \zeta\cdot \overline{U w}|^{2d}} \right)^2  d U \Big]  d \mu(\zeta) d\mu_\D(w).
\end{multline*}
 Clearly, for any $\zeta \in \Sph_d$, since $|Uw| = |w|$ and $\zeta\cdot \overline{U w} = (U^{-1}\zeta)\cdot \overline{w}$, we have 
\[
\int_{\mathcal{U}_d}\left(  \frac{(1-|U w|^2)^d}{|1 - \zeta\cdot \overline{U w}|^{2d}} \right)^2  dU =  \int_{\mathcal{U}_d}\left(  \frac{(1-|w|^2)^d}{|1 - (U^{-1}\zeta)\cdot \bar{w}|^{2d}} \right)^2  dU  = \int_{\Sph_d} \left(  \frac{(1-|w|^2)^d}{|1 - \xi\cdot \bar{w}|^{2d}} \right)^2  d\sigma_{\Sph_d}(\xi) .
\]
 By \cite[Prop. 1.4.10]{Rudin-ball}, there exists  a constant $c>0$ such that  for any $w \in \D_d$,
\[
  \int_{\Sph_d} \left(  \frac{(1-|w|^2)^d}{|1 - \xi\cdot \bar{w}|^{2d}} \right)^2  d\sigma_{\Sph_d}(\xi) \le ce^{d \cdot d_B(w, o)}. 
\]
Thus we obtain 
\begin{multline*}
\int_{\A_k(o)} \| P^b(w, \cdot) \|_{L^2(\mu)}^2  d\mu_{\D_d}(w) \le \int_{\A_k(o)} \int_{\Sph_d}   ce^{d \cdot d_B(w, o)} d \mu(\zeta) d\mu_{\D_d}(w)  \le c e^{(k+1)d} \mu_{\D_d}(\A_k(o)).
\end{multline*}
By Lemma \ref{lem-cb-ent}, there exists a constant $c'>0$ such that $\mu_{{\D_d}} (\A_k(o)) \le c' e^{kd}$ and we obtain the desired inequality \eqref{mg-poi-cb}.  By Lemma~\ref{prop-mgc}, the mean-growth estimate \eqref{mg-poi-cb} implies that  $F_\mu: \D_d \rightarrow L_\R^2(\mu)$ defined in \eqref{def-F-mu} is tempered in the sense of Definition \ref{def-MVP}. 

Since $1 - e^{-2x}  \ge 2 e^{-2}x $ for any $x\in(0,1)$, by \eqref{mg-poi-cb}, for any $\varepsilon\in (0, 1)$, 
\begin{align*}
& \int_{\D_d} e^{-2 \varepsilon d_B(w, o)} \| P^b(w, \cdot) \|_{L^2(\mu)}^2  e^{-2 d\cdot d_B(w, o)} d\mu_{\D_d}(w) 
\\
\le & \sum_{k = 0}^\infty e^{-(2d + 2 \varepsilon) k} \int_{\A_k(o)} \| P^b(w, \cdot) \|_{L^2(\mu)}^2  d\mu_{\D_d}(w)
\le C \sum_{k=0}^\infty e^{-2 \varepsilon k }  = \frac{C}{1 - e^{-2 \varepsilon}} \le \frac{e^2 C}{ 2 \varepsilon}. 
\end{align*}
This implies the desired inequality \eqref{lap-es-cb}.  By Lemma \ref{prop-sharp-temp},  $F_\mu$ is sharply  tempered.  
\end{proof}

\begin{proof}[Proof of Lemma \ref{lem-g-uncond-cb}]
By Thereom~\ref{thm-n-sub}, Lemma~\ref{lem-cmvp-cb}, for $\PP_d$-almost every $X\in \Conf(\D_d)$, simultaneously for all $s>d$ and all $z\in \D_d$, the series 
\[
g_X(s, z; F_\mu) =   \sum_{x\in  X} e^{-s d_B(x, z)} F_\mu(x)   = \sum_{x\in  X} e^{-s d_B(x, z)} P^b(x, \cdot)  
\]
converges unconditionally in $L_\R^2(\mu)$  and thus   for all $f = P^b[h\mu] \in \mathcal{H}^2(\D_d; \mu)$, the series 
\[
\sum_{x\in X} e^{-s d_B(x,z)} f(x)= \sum_{x\in X} e^{-s d_B(x,z)}  \int_{\Sph_d} h(\zeta) P^b(x, \zeta) d\mu(\zeta)= \sum_{x\in X} e^{-s d_B(x,z)}  \langle  h, P^b(x, \cdot) \rangle_{L^2(\mu)}
\]
converges unconditionally. We complete the proof of the lemma by using the equivalence  between the unconditional convergence and the absolute convergence  for scalar series. 
\end{proof}

\begin{proof}[Proof of Theorem \ref{thm-poi}]
Let $\mu$  be a Borel probability measure on $\Sph_d$.  Thereom~\ref{thm-n-sub} and Lemma~\ref{lem-cmvp-cb}  together imply that if $D\subset \D_d$ is a relatively compact subset, then  for $\PP_d$-almost every $X\in \Conf(\D_d)$, we have 
\[
 \lim_{s \to d^{+}} \sup_{z\in D}\Big \|  \frac{g_X(s,z; P^b(x, \cdot)) }{g_X(s,z)} - P^b(z, \cdot)\Big\|_{L^2(\mu)} = 0.
\]
 By observing that for any $f = P^b[h \mu]\in \mathcal{H}^2(\D_d; \mu)$,  
\[
  \frac{g_X(s,z; f) }{g_X(s,z)} = \Big \langle h,  \frac{g_X(s,z; P^b(x, \cdot)) }{g_X(s,z)} \Big\rangle_{L^2(\mu)}  \an f(z) = \langle h, P^b(z, \cdot)\rangle_{L^2(\mu)},
\] 
 we complete the whole proof of Theorem \ref{thm-poi}. 
\end{proof}

Let us now proceed to the proof of Theorem \ref{thm-n-sub}. 
\begin{proof}[Proof of Theorem \ref{thm-n-sub}]
 Fix a  non-identically zero sharply tempered $\MM$-harmonic function $F: \D_d\rightarrow L_\R^2(\nu)_{+}$.  Then we can fix a strictly decreasing  sequence $(\varepsilon_n)_{n\ge 0}$  converging to $0$ and satisfying \eqref{e-n-gap} and \eqref{gap-sum}. Let $(s_n)_{n\ge 0}$ be  defined by 
\begin{align}\label{def-s-n-epsilon}
s_n = d + \varepsilon_n. 
\end{align}
 Then  the sequence $(s_n)_{n\ge 0}$ converges to $d$ and satisfies  the condition \eqref{fast-to-cri}. 

 Fix a countable dense subset $\mathcal{D} \subset \D_d$.  By Lemma \ref{prop-L2-L1-sum} and Theorem \ref{thm-H-L}, there exists a subset $\Omega\subset \Conf(\D_d)$ with $\PP_d(\Omega)  = 1$ such that for any $X\in \Omega$ and any $z\in \mathcal{D}$, we have 
\begin{itemize}
\item $0< g_X(s_n, z)<\infty$  for all $n\in \N$; 
\item the following limit holds (cf. \eqref{av-or-not-1-onez} in the proof of Theorem \ref{thm-H-L}):
\begin{align}\label{av-or-not-1}
\lim_{n\to\infty} \frac{g_X(s_n, z)}{g_{\PP_{d}}(s_n)} = 1;
\end{align}
\item the following  series converges in $L_\R^2(\nu)$ for all $n \in \N$: 
\begin{align}\label{cv-g-n}
g_X(s_n, z; F) = \sum_{k = 0}^\infty g_X^{(k)}(s_n, z; F) = \sum_{k = 0}^\infty \sum_{x\in X\cap \A_k(z)} e^{-s_n d_B(x,z)} F(x);
\end{align}
\item the following limit equality holds:

\begin{align}\label{ratio-uc-cv}
 \lim_{n\to\infty}  \Big\| \frac{g_X(s_n, z; F) }{g_X(s_n, z)} - F(z) \Big\|_{L^2(\nu)}=0.
\end{align}
\end{itemize}

We now complete the proof of Theorem \ref{thm-poi} by proving the following claims.

{\flushleft \it Claim I.} For any $X\in \Omega$,  the limit equality
\eqref{av-or-not-1}  holds for all $z\in \D_d$.   

For   an arbitrary $z' \in \D_d$, there exists a sequence  $(z_k)_{k\ge 0}$ in $\mathcal{D}$ converging to $z'$. For any $X\in \Omega$, any $k, n\in \N$, we have
\[
e^{-s_n d_B(z_k, z')} g_X(s_n, z_k) \le g_X(s_n, z') \le e^{s_n d_B(z_k, z')} g_X(s_n, z_k).
\]
Therefore,  using the limit equality \eqref{av-or-not-1} for $z_k \in \mathcal{D}$, we obtain
\[
    e^{-  d \cdot d_B (z_k, z')} \le \liminf_{n\to\infty}   \frac{g_X(s_n, z')}{g_{\PP_{d}}(s_n)}  \le \limsup_{n\to\infty}   \frac{g_X(s_n, z')}{g_{\PP_{d}}(s_n)} \le e^{ d \cdot d_B(z_k, z')}. 
\]
The desired limit equality \eqref{av-or-not-1} for the point $z'$ then follows since $\lim_k d_B(z_k, z')= 0$.

{\flushleft \it Claim II.} For any $X\in \Omega$, the following series converges unconditionally in $L_\R^2(\nu)$ for all $s>d$ and all $z\in \D_d$:
\begin{align}\label{uncon-cv-s}
 g_X(s, z; F) =  \sum_{x\in  X} e^{-s d_B(x, z)} F(x).
\end{align}

We will use an elementary fact:  Let $(u_n)_{n\ge 0}$ be a sequence in $L_\R^2(\nu)_{+}$ and  $(b_n)_{n\ge 0}$ be a sequence of positive numbers with $\sup_{n\in \N} b_n<\infty$. If the series $\sum_{n = 0}^\infty u_n$ converges in $L^2(\nu)$, then it converges  unconditionally and so is the series $\sum_{n= 0}^\infty b_n u_n$. 

Fix any $X\in \Omega$.  Let $s>d$ and $z'\in \D_d$.  Fix any point $z_0\in \mathcal{D}$ and any integer $n_0$, large enough such that $s_{n_0} \le s$. Since $F$ takes values in $L_\R^2(\nu)_{+}$, the convergence of the  series \eqref{cv-g-n} for $s_{n_0}$ and $z_0$ implies the unconditional convergence of the series \eqref{uncon-cv-s} for $s_{n_0}$ and $z_0$. But then, by using   \[e^{-s d_B(x, z')} F(x)=  e^{-s d_B(x, z') + s_{n_0} d_B(x,z_0)} \cdot  e^{-s_{n_0} d_B(x,z_0)} F(x)
\] and 
\[
\sup_{x\in X}  e^{-s d_B(x, z') + s_{n_0} d_B(x,z_0)} \le  \sup_{x\in X}  e^{-s d_B(x, z') + s d_B(x,z_0)}  \le e^{s d_B(z', z_0)}<\infty,
\]
we immediately obtain the unconditional convergence of the series \eqref{uncon-cv-s} for $s$ and $z'$. 

{\flushleft \it Claim III.}
 For any $X\in \Omega$, the series 
\[
 g_X(s, z) =  \sum_{x\in  X} e^{-s d_B(x, z)} 
\]
converges  and $g_X(s, z)>0$ for all $s>d$, all $z\in \D_d$. The proof of Claim III is similar to that of Claim II.

{\flushleft \it Claim IV.} For any $X\in \Omega$, the  limit equality  \eqref{ratio-uc-cv} holds for all $z\in \D_d$.

Fix any $X\in \Omega$.  For   an arbitrary $z' \in \D_d$, there exists a sequence  $(z_k)_{k\ge 0}$ in $\mathcal{D}$ converging to $z'$. Clearly, for any $n, k\in \N$, we have 
\[
e^{-2s_n d_B(z_k, z')} \frac{g_X(s_n, z_k; F)}{g_X(s_n, z_k)}  \le \frac{g_X(s_n, z'; F)}{g_X(s_n, z')} \le  e^{2s_n d_B(z_k, z')} \frac{g_X(s_n, z_k; F)}{g_X(s_n, z_k)} 
\]
and thus, by using the simplified notation $\| \cdot \| = \|\cdot \|_{L^2(\nu)}$, we have 
\begin{align*}
& \Big\|\frac{g_X(s_n, z'; F)}{g_X(s_n, z')} -F(z')\Big\|  \le  \max_{\pm} \Big\| e^{\pm 2s_n d_B(z_k, z')}  \frac{g_X(s_n, z_k; F)}{g_X(s_n, z_k)} - F(z')\Big\|
\\
 \le & \max_{\pm} \Big\{ \Big\|  e^{\pm 2s_n d_B(z_k, z')} \Big( \frac{g_X(s_n, z_k; F)}{g_X(s_n, z_k)} -  F(z_k)\Big) \Big\| + \| e^{\pm 2s_n d_B(z_k, z')}F(z_k) - F(z')\|\Big\} 
\\
 \le&   e^{2s_n d_B(z_k, z')} \Big\{ \Big\|  \frac{g_X(s_n, z_k; F)}{g_X(s_n, z_k)} -  F(z_k) \Big\| + \| F(z_k) - F(z')\| \Big\} +  |e^{ 2s_n d_B(z_k, z')} -1|\|F(z')\|.
\end{align*}
Therefore,  for any $k\in \N$, by using  \eqref{ratio-uc-cv} for $z_k$ and  $\lim_n s_n =d$,  we obtain
\[
\limsup_{n\to\infty}\Big\|\frac{g_X(s_n, z'; F)}{g_X(s_n, z')} -F(z')\Big\|  \le   e^{2 d d_B(z_k, z')} \| F(z_k) - F(z')\| +    |e^{ 2d d_B(z_k, z')} -1|\|F(z')\|.
\]
Finally, by using the assumption $\lim_k d_B(z_k, z') = 0$ and the continuity of $F$ ($\MM$-harmonic functions are continuous), we obtain the desired limit equality \eqref{ratio-uc-cv}  for $z'$.

{\flushleft \it Claim V.} For any $X\in \Omega$, the limit equality \eqref{thm-n-sub-goal} holds for all $z\in \D_d$. 

Fix any $X\in \Omega$ and any $z\in \D_d$.   Recall the notation   $R_X(s, z; F)$ and  $\underline{R}_X(s,z; F)$ introduced in \eqref{notation-T-sigma-3}. By Claim I and Claim IV,  writing $\| \cdot \| = \|\cdot \|_{L^2(\nu)}$, we have 
\begin{align}\label{R-under-allz}
\lim_{n\to\infty} \| \underline{R}_X(s_n,z; F) - F(z)\| = 0.
\end{align}
Recall the definition \eqref{def-s-n-epsilon} of $s_n$.  Let $s\in (d, s_0)$. Since $(s_n)_{n \in\N}$ is strictly decreasing and converges to $d$,  there is a unique  $n_s \in \N$ such that 
$
 s_{n_s + 1}\le s <  s_{n_s}
$ 
and   we may define
\begin{align}\label{R-minus-plus}
\underline{R}_X^{+}(s,z; F): =  \frac{g_X(s_{n_s+1}, z; F)}{g_{\PP_d} (s_{n_s+1})}, \quad \underline{R}_X^{-}(s,z; F): =  \frac{g_X(s_{n_s}, z; F)}{g_{\PP_d} (s_{n_s})}, \quad  \beta(s):= \frac{g_{\PP_d}(s_{n_s +1})}{g_{\PP_d}(s_{n_s})}.
\end{align}
The limit equality \eqref{R-under-allz} implies that 
\begin{align}\label{R-pm-allz}
\lim_{s\to d^{+}} \| \underline{R}_X^{\pm}(s,z; F) - F(z)\| = \lim_{n\to\infty} \| \underline{R}_X(s_n,z; F) - F(z)\|  = 0.
\end{align}
 By Lemma \ref{lem-s-d}, we have \[
\lim_{s\to d^{+}}  (s - d)  g_{\PP_d}(s)   = \frac{d}{4^d}.\] Thus
by \eqref{e-n-gap}  and the definition \eqref{def-s-n-epsilon} of $s_n$, we have
\begin{align}\label{beta-to-1}
\lim_{s\to d^{+}} \beta(s) =  \lim_{n\to\infty}  \frac{g_{\PP_d} (s_{n+1})}{g_{\PP_d} (s_{n})} = 1.
\end{align}
The monotonicity of $s\mapsto e^{-s d_B(x,z)}$ and the assumption $F(x)\in L_\R^2(\nu)_{+}$ together imply
\[
\beta(s)^{-1} \underline{R}_X^{-}(s,z; F) \le  \underline{R}_X(s,z; F) \le \underline{R}_X^{+}(s,z; F)\beta(s).
\]
Therefore,  by noting that $\beta(s)\ge 1$, we have 
\begin{multline}\label{R-un-allz-alls}
\|\underline{R}_X(s,z; F) - F(z)\| \le 2 \max_{\pm} \| \beta(s)^{\pm 1} \underline{R}_X^{\pm}(s,z; F) - F(z) \|
\\
 \le 2 \max_{\pm} \Big\{ \| \beta(s)^{\pm}  ( \underline{R}_X^{\pm}(s,z; F) - F(z))\| + \| \beta(s)^{\pm 1} F(z) -F(z)\|\Big\}
\\
\le 2 \beta(s)  \max_{\pm} \| \underline{R}_X^{\pm}(s,z; F) - F(z)\| + 2 | \beta(s) -1| \| F(z)\|. 
\end{multline}
By \eqref{R-pm-allz}, \eqref{beta-to-1} and \eqref{R-un-allz-alls}, we have
$
\lim_{s\to d^{+}}\|\underline{R}_X(s,z; F) - F(z)\| = 0,
$
which, combined with Claim I, implies the desired equality
\begin{align}\label{single-inter}
\lim_{s\to d^{+}}\|R_X(s,z; F) - F(z)\| = 0.
\end{align}

{\flushleft \it Claim VI.} For any relatively compact subset $D\subset \D_d$, there exists $C_D>0$ such that for any $\varepsilon\in (0,1)$ and any $s\in (d, 2d)$, we have 
\begin{align}\label{Prob-epsilon-less}
\PP_d\Big( \sup_{z\in D}   \|\underline{R}_X(s,z; F) - F(z) \| > \varepsilon \Big) \le \frac{C_D}{\varepsilon^{2d +2}} \frac{1}{[g_{\PP_d}(s)]^2}  \int_{\D_d} e^{-2s d_B(x, o)} \| F(x) \|^2 d\mu_{\D_d}(x).
\end{align}

Indeed, since $D$ is relatively compact and $F$ is $\MM$-harmonic, there exists $C_1>0$ such that 
\[
\| F(x) - F(y)\| \le C_1 | x - y| \quad \text{for all $x, y \in D$}
\]
and 
\begin{align}\label{def-C23}
C_2: = \sup_{x, y \in D} e^{2d d_B(x,y)}<\infty,   \quad C_3 : = \sup_{x\in D}\|F(x)\|<\infty.
\end{align}
Note that for any $x,y\in D$, we have 
\[
  e^{-s d_B(x,y)}\underline{R}_X(s,y; F)  \le \underline{R}_X(s,x; F) \le  e^{s d_B(x,y)}\underline{R}_X(s,y; F).
\]
Therefore, for any $x, y \in D$ and any $s \in (d, 2d)$, by using  the elementary inequality
\[
|e^{ \pm s d_B(x,y)} -1| \le e^{ s d_B(x,y)} -1 \le e^{ 2d d_B(x,y)} -1
\]
and the definition \eqref{def-C23} of $C_2$ and $C_3$, 
  we obtain 
\begin{align}\label{x-y-ineq}
\begin{split}
& \|\underline{R}_X(s,x; F) - F(x) \|  \le \max_{\pm}  \| e^{\pm s d_B(x,y)} \underline{R}_X(s,y; F) -F(x) \| 
\\
\le&  \max_{\pm}   \Big\{ e^{\pm s d_B(x,y)} \Big(\|  \underline{R}_X(s,y; F) -   F(y)\| + \| F(y)-  F(x)\| \Big)+  |e^{\pm s d_B(x,y)} -1| \|F(x)  \| \Big\}
\\
\le & e^{s d_B(x,y)} \Big(\|  \underline{R}_X(s,y; F) -   F(y)\| + \| F(y)-  F(x)\| \Big)+  |e^{ s d_B(x,y)} -1| \|F(x)  \|
\\
\le & C_2 \|  \underline{R}_X(s,y; F) -   F(y)\|  + C_1 C_2 |x-y| + C_3  (e^{ 2d d_B(x,y)} -1).
\end{split}
\end{align}
For any $\varepsilon \in (0, 1)$, it is clear that there exists a constant  $c>0$ depending on the relatively compact subset $D$ such that 
\begin{align}\label{demi-epsilon}
\sup_{x, y \in D:  |x-y|\le c \varepsilon} \Big\{  C_1 C_2 |x-y| + C_3  (e^{ 2d d_B(x,y)} -1) \Big\}< \frac{\varepsilon}{2}.
\end{align}
Let $D_\varepsilon \subset D$ be any fixed finite $c\varepsilon$-net of $D$ with respect to the Euclidean metric, that is,  for any $x \in D$, we have
\[
\inf_{z\in D_\varepsilon} | z- x| \le c \varepsilon.
\]
By a classical volume argument, there exists a constant $c'>0$ in such a way that the previous finite subset $D_\varepsilon\subset D$ can be chosen with cardinality 
\begin{align}\label{card-cont}
\# D_\varepsilon \le  \frac{c'}{\varepsilon^{2d}}. 
\end{align}
Combining \eqref{x-y-ineq} with  \eqref{demi-epsilon}, by  our choice of $D_\varepsilon$, we obtain 
\begin{multline}\label{unif-prob}
\PP_d\Big( \sup_{z\in D}   \|\underline{R}_X(s,z; F) - F(z) \| > \varepsilon \Big) \le    \PP_d\Big( \sup_{z\in D_\varepsilon}   \|\underline{R}_X(s,z; F) - F(z) \| > \frac{\varepsilon}{2C_2} \Big)
\\
\le   \sum_{z\in D_\varepsilon} \PP_d\Big(    \|\underline{R}_X(s,z; F) - F(z) \| > \frac{\varepsilon}{2C_2} \Big).
\end{multline}
The Chebychev inequality and  Lemma \ref{prop-var-up-bd-rep}  imply that, for any $s\in (d, 2d)$ and $z\in D_\varepsilon$,  
\begin{multline}\label{single-prob}
\PP_d\Big(    \|\underline{R}_X(s,z; F) - F(z) \| > \frac{\varepsilon}{2C_2} \Big)\le  \frac{4 C_2^2}{\varepsilon^2} \frac{2}{[g_{\PP_d}(s)]^2}  \int_{\D_d} e^{-2s d_B(x, z)} \| F(x) \|^2 d\mu_{\D_d}(x)
\\
\le \frac{8 C_2^2}{\varepsilon^2} \frac{C_4}{[g_{\PP_d}(s)]^2}  \int_{\D_d} e^{-2s d_B(x, o)} \| F(x) \|^2 d\mu_{\D_d}(x),
\end{multline}
where 
\[
C_4 : =  \sup_{z\in D} e^{4d d_B(z,o)}<\infty.
\]
Combining \eqref{card-cont}, \eqref{unif-prob} and \eqref{single-prob}, we obtain the desired inequality:
\[
\PP_d\Big( \sup_{z\in D}   \|\underline{R}_X(s,z; F) - F(z) \| > \varepsilon \Big)  \le \frac{8 C_2^2 C_4 c'}{\varepsilon^{2d +2}} \frac{1}{[g_{\PP_d}(s)]^2}  \int_{\D_d} e^{-2s d_B(x, o)} \| F(x) \|^2 d\mu_{\D_d}(x).
\]

{\flushleft \it Claim VII.} Let $D\subset \D_d$ be  any relatively compact subset. Then for $\PP_d$-almost every $X\in \Conf(\D_d)$, we have 
\begin{align}\label{sup-a-seq}
\lim_{n\to\infty} \sup_{z\in D}   \|\underline{R}_X(s_n,z; F) - F(z) \|  = 0.
\end{align}
where the sequence  $(s_n)_{n\ge 1}$ is defined in \eqref{def-s-n-epsilon}. 

Indeed, recall that the sequence $(s_n)_{n\ge 1}$ defined in \eqref{def-s-n-epsilon} satisfies  the condition \eqref{fast-to-cri}. Thus the inequality \eqref{Prob-epsilon-less} implies that for any $\varepsilon\in(0,1)$, 
\[
\sum_{n= 1}^\infty \PP_d\Big( \sup_{z\in D}   \|\underline{R}_X(s_n,z; F) - F(z) \| > \varepsilon \Big) <\infty.
\]
Since $\varepsilon \in (0,1)$ is arbitrary, we obtain that for $\PP_d$-almost every $X\in \Conf(\D_d)$, the desired limit equality 
\eqref{sup-a-seq} holds. 

{\flushleft \it Claim VIII.} Let $D\subset \D_d$ be  any relatively compact subset. Then for $\PP_d$-almost every $X\in \Conf(\D_d)$, we have 
\begin{align}\label{lc-unif}
\lim_{s\to d^{+}} \sup_{z\in D}   \|R_X(s,z; F) - F(z) \|  = 0.
\end{align}

Indeed, repeating exactly the same  derivation of the  $\PP_d$-almost sure limit equality \eqref{single-inter}  from   \eqref{R-under-allz}, we  obtain  the proof of the $\PP_d$-almost sure limit equality \eqref{lc-unif} from \eqref{sup-a-seq}. 
\end{proof}

\section{The reproducing kernels with super-critical weights}\label{sec-kernel-es}
This section is devoted to the proof of  Proposition \ref{prop-w-cb}.   

\begin{lemma}\label{lem-critical-w}
There exist two constants $c, C>0$ depending only on $d$,  such that
\[
\frac{c}{(1  - |z|^2)^d} \log \Big( \frac{2}{1 - |z|^2}\Big) \le K_{W_{\mathrm{cr}}}(z, z) \le \frac{C}{(1  - |z|^2)^d} \log \Big( \frac{2}{1 - |z|^2}\Big) \quad \text{for all $z\in \D_d$}.
\]
\end{lemma}

The proof of the following elementary lemma is routine and will be omitted.  
\begin{lemma}\label{lem-com-series}
Let $(a_k)_{k\in \N}, (b_k)_{k\in \N}$ be two sequences in $\R_{+}$ with $\lim_{k\to\infty} a_k/b_k  = 0$. Assume that $\sum_{k\in \N} a_k t^k$ and $\sum_{k\in \N} b_k t^k$ converge for all $t\in (0, 1)$ and $\sum_{k\in\N} b_k = \infty$. Then 
\[
\lim_{t \to 1^{-}} \frac{ \sum_{k\in \N} a_k t^k }{\sum_{k\in \N} b_k t^k} = 0. 
\]
\end{lemma}

\begin{lemma}\label{lem-log-series}
For any integer $d\ge 1$, there exist constants $c_d, C_d> 0$ such that
\[
 \frac{c_d}{(1  - t)^d} \log \Big( \frac{2}{1 - t}\Big)\le \sum_{k =0}^\infty (k+1)^{d-1} \log (k + 2) t^k \le  \frac{C_d}{(1  - t)^d} \log \Big( \frac{2}{1 - t}\Big) \quad \text{for all $t\in (0,1)$}.
\]
\end{lemma}
\begin{proof}
For any integer $m\ge 1$ and any $t\in (0, 1)$, we have 
\[
(1 - t) \sum_{k=0}^\infty (k+1)^m \log (k+2) t^k  =  \sum_{k=0}^\infty \Big[\underbrace{(k+1)^m \log (k+2) - k^m  \log (k +1)}_{\text{denoted $a_{k,m}$}}\Big] t^k.
\]
Clearly, there exist contants $c_m, C_m>0$ depending only on $m$ such that 
\[
 c_m \cdot (k+1)^{m-1} \log (k+2) \le a_{k,m} \le C_m  \cdot (k + 1)^{m-1} \log (k +2) \quad \text{for all $k\ge 0$}.
\]
 Therefore, we have 
\[
\frac{c_m}{1-t} \le \frac{\sum_{k=0}^\infty (k+1)^m \log (k+2) t^k}{\sum_{k=0}^\infty (k+1)^{m-1} \log (k+2) t^k} \le \frac{C_m}{1-t} \quad \text{for all $t\in (0,1)$}.
\]
It follows that, there exist constants $c_d', C_d' >0$ such that 
\[
\frac{c_d'}{(1-t)^{d-1}} \le \frac{\sum_{k=0}^\infty (k+1)^{d-1}\log (k+2) t^k}{\sum_{k=0}^\infty  \log (k+2) t^k} \le \frac{C_d'}{(1-t)^{d-1}} \quad \text{for all $t\in (0,1)$}.
\]
Finally, note that for any $t\in (0, 1)$, 
\[
(1 - t) \sum_{k=0}^\infty \log (k + 2) t^k = \sum_{k=0}^\infty \log \Big(1 + \frac{1}{k+1}\Big) t^k \le  \log 2 + \sum_{k=1}^\infty \frac{t^k}{k} =  \log\Big( \frac{2}{ 1 - t}\Big) 
\]
and  there exists $c''>0$ such that for any $t\in (0, 1)$,
\[
(1 - t) \sum_{k=0}^\infty \log (k + 2) t^k = \sum_{k=0}^\infty \log \Big(1 + \frac{1}{k+1}\Big) t^k\ge c''  \Big(\log 2 + \sum_{k=1}^\infty \frac{t^k}{k} \Big) = c'' \log\Big( \frac{2}{ 1 - t}\Big).
\]
Combining the above inequalities, we  complete the proof of the lemma. 
\end{proof}

\begin{proof}[Proof of Lemma \ref{lem-critical-w}]
Since $W_{\mathrm{cr}}$ is radial, the polynomials $(z^{n} :  = z_1^{n_1} \cdots z_d^{n_d})_{n\in \N^d}$ are orthogonal and complete in $A^2(\D_d, W_{\mathrm{cr}})$. Thus  
\begin{align}\label{expansion-K-W-critical}
K_{W_{\mathrm{cr}}}(z, w) = \sum_{n\in \N^d} a_n(W_{\mathrm{cr}}) z^n \bar{w}^n \quad \text{with $a_n(W_{\mathrm{cr}})= \| z^n\|_{A^2(\D_d, W_{\mathrm{cr}})}^{-2}$}.
\end{align}
  For any $n\in \N^d$,  we have the identity (see e.g. Zhu \cite[Lem. 1.11]{Zhu-UB}), 
\[
\int_{\Sph_d} | \zeta^n|^2 d\sigma_{\Sph_d}(\zeta) = \frac{(d-1)! n_1! \cdots n_d!}{ (d-1 + |n|)!}.
\]
Therefore,  by the formula of integration in polar coordinates,
\begin{align*}
a_n(W_{\mathrm{cr}})^{-1} & = \int_{\D_d} | z^n|^2 W_{\mathrm{cr}} (z) dv_d(z) =  2d \int_0^1 r^{2d + 2|n|-1}  W_{\mathrm{cr}}(r)dr   \int_{\Sph_d} |\zeta^n|^2    d\sigma_{\Sph_d}(\zeta)
\\
& = d \cdot \frac{n_1! \cdots n_d!}{ (d-1 + |n|)!} \int_0^1 t^{|n| + d -1} ( 1 - t)^{-1} \log ^{- 2} \Big(  \frac{4}{ 1 - t}\Big)  dt  
\\
&  = d \cdot \frac{n_1! \cdots n_d!}{ (d-1 + |n|)!} \cdot \int_{\log 4}^\infty   \frac{(1 - 4 e^{-x})^{|n| + d -1}}{x^2} dx. 
\end{align*}
{\flushleft \bf Claim:} There exist two constants $c_1, c_2 > 0$ such that  
\[
\frac{c_1}{\log (4k +4)} \le \int_{\log 4}^\infty   \frac{(1 - 4 e^{-x})^k}{x^2} dx \le \frac{c_2}{\log (4k +4)} \quad \text{for all  integer $k\ge 0$}.
\]
Indeed, for any integer  $k \ge 0$, we have the lower-estimate of the integral: 
\begin{multline*}
\int_{\log 4}^\infty   (1 - 4 e^{-x})^k \frac{dx}{x^2}  \ge \int_{\log (4k+4)}^\infty    (1 - 4 e^{-x})^k \frac{dx}{x^2}   
\ge  \int_{\log (4k+4)}^\infty   \Big(1 - \frac{1}{k+1}\Big)^k \frac{dx}{x^2} \ge \frac{c_1}{\log (4k +4)}.   
\end{multline*}
Now set 
\[
H_k(x):   =  \frac{(1 - 4 e^{-x})^k}{x^2}\quad \text{for $x \in [\log 4, \infty)$}.
\] For $k\ge 1$,  we can show, by studying the derivative $H_k'(x)$, that $H_k$ is increasing on the interval  $[\log 4, \log (4k +4)]$ and hence 
\[
H_k(x) \le H_k(\log (4k+4)) \le \frac{1}{\log^2 (4k +4)} \quad \text{for all $x\in [\log 4, \log (4k +4)]$}.
\]
Therefore, for any $k\ge 1$, we have 
\begin{multline*}
\int_{\log 4}^\infty  \frac{ (1 - 4 e^{-x})^k}{x^2}dx    \le \int_{\log 4}^{\log (4k +4)}  \frac{ (1 - 4 e^{-x})^k}{x^2}dx   + \int_{\log (4k + 4)}^\infty  \frac{ (1 - 4 e^{-x})^k}{x^2}dx  \le 
\\
\le   \sup_{\log 4 \le x \le \log (4k +4)} H_k(x) \cdot  \int_{\log 4}^{\log (4k +4)}  dx   + \int_{\log (4k + 4)}^\infty  \frac{ 1}{x^2}dx     \le \frac{2}{\log (4k +4)}. 
\end{multline*}
Consequently,  there exist  constants $c_3, c_4> 0$ such that  for  any $n\in \N^d$, we have 
\begin{align}\label{expansion-K-W-critical-2}
c_3 \le \frac{ a_n(W_{\mathrm{cr}})^{-1}}{ \displaystyle d \cdot \frac{n_1! \cdots n_d!}{ (d-1 + |n|)!} \frac{1}{\log (4(|n|+d))}} \le c_4.
\end{align}

By using  the  elementary identity (see e.g. Zhu \cite[formula (1.1)]{Zhu-UB}) 
\[
\sum_{n \in \N^d: |n| = k } \frac{r_1^{n_1} \cdots r_d^{n_d}}{n_1!\cdots n_d!} =  \frac{(r_1 + \cdots + r_d)^k}{k!},
\]
 we obtain,  for any $z\in \D_d$, that
\begin{align*}
\Phi(z)  =  :  \sum_{n \in \N^d} \frac{(d -1 + |n|)!}{d} \frac{\log (4(|n|+d)) }{n_1!\cdots n_d!} |z^n|^{2}= \sum_{k=0}^\infty \frac{(d -1 +k)!}{ d \cdot k!} \log (4(k+d)) |z|^{2k}.
\end{align*}
Therefore, by the limit equalities  
\[
\lim_{k\to\infty}\frac{(d-1+k)!}{(k+1)^{d-1} \cdot k!} = 1 \an \lim_{k\to\infty}\frac{\log(4 (k + d))}{\log (k + 2)} = 1
\]
and Lemma \ref{lem-log-series},  there exist constants $c_5, c_6> 0$, such that 
\[
\frac{c_5}{(1  - |z|^2)^d} \log \Big( \frac{2}{1 - |z|^2}\Big) \le \Phi(z)\le  \frac{c_6}{(1  - |z|^2)^d} \log \Big( \frac{2}{1 - |z|^2}\Big). 
\]
Finally, comparing  \eqref{expansion-K-W-critical}  with the definition of $\Phi(z)$ and using \eqref{expansion-K-W-critical-2}, we obtain $c_3 \Phi(z) \le K_{W_{\mathrm{cr}}}(z,z) \le c_4 \Phi(z)$ for all $z\in \D_d$. This completes the whole proof.
\end{proof}

Let $\Phi: [0, 1) \rightarrow \R_{+}$ be a function in $L^1([0,1])$ such that $\int_{\delta}^1 \Phi(r)dr > 0$ for any $\delta \in (0, 1)$.   Let   $B: [0, 1)\rightarrow \R_{+}$ be a function such that $B\Phi \in L^1([0,1])$ and 
\[
\lim_{r\to 1^{-}} B(r) = \infty.
\] 
Define  two radial Bergman-admissible weights on $\D_d$ by 
\[
W_\Phi(z) = \Phi(|z|) \an W_{B\Phi}(z) = B(|z|) \Phi(|z|).
\]
  We shall compare the reproducing kernels $K_{W_\Phi}$ and $K_{W_{B\Phi}}$. Note that 
\begin{align}\label{lem-comp-coef} 
\lim_{k\to\infty} \frac{ \int_0^1 r^k \Phi(r) dr}{ \int_0^1 r^k B(r) \Phi(r)dr} = 0.
\end{align}
Indeed, for any $\varepsilon > 0$, take $r_\varepsilon \in (0, 1)$ such that $\varepsilon \cdot B(r) \ge  1$ for all $r\in [r_\varepsilon, 1)$. Take any $\delta_\varepsilon$ with $r_\varepsilon < \delta_\varepsilon < 1$, we have 
\begin{align*}
\limsup_{k\to\infty} \frac{ \int_0^{r_\varepsilon} r^k \Phi(r) dr }{\int_0^1 r^k B(r) \Phi(r) dr} & \le \limsup_{k\to\infty}  \frac{ \int_0^{r_\varepsilon} r^k \Phi(r) dr }{ \int_{\delta_\varepsilon}^1 r^k B(r) \Phi(r) dr}  \le  \limsup_{k\to\infty}  \frac{ r_\varepsilon^k \cdot  \int_0^{r_\varepsilon} \Phi(r) dr }{ \delta_\varepsilon^k \cdot  \int_{ \delta_\varepsilon}^1  B(r) \Phi(r) dr}   =0. 
\end{align*}
Therefore, we have 
\begin{align*}
\limsup_{k\to\infty} \frac{ \int_0^1 r^k \Phi(r) dr}{ \int_0^1 r^k B(r) \Phi(r)dr} &  \le  \limsup_{k\to\infty} \frac{ \int_0^{r_\varepsilon} r^k \Phi(r)dr + \varepsilon \cdot  \int_{r_\varepsilon}^1 r^k  B(r) \Phi(r)dr}{ \int_0^1 r^k B(r) \Phi(r)dr}    \le \varepsilon. 
\end{align*}
Since $\varepsilon> 0$ is arbitrary, we complete the proof of \eqref{lem-comp-coef}.

\begin{lemma}\label{lem-comp-2-rad-w}
Under the above assumptions on $\Phi$ and $B$, we have 
\[
\lim_{|z|\to 1^{-}} \frac{K_{W_{B\Phi}}(z,z)}{K_{W_\Phi}(z,z)} = 0.
\]
\end{lemma}

\begin{proof}
Since $W_\Phi: \D_d\rightarrow \R_{+}$ is radial, we have 
\[
K_{W_\Phi}(z, w) = \sum_{n\in \N^d} a_n(\Phi) z^n \bar{w}^n \quad \text{with $a_n(\Phi)= \| z^n\|_{A^2(\D_d, W_\Phi)}^{-2}   = \| z_1^{n_1} \cdots z_d^{n_d}\|_{A^2(\D_d, W_\Phi)}^{-2}$}.
\]
For any $n\in \N^d$,  by the formula of integration in polar coordinates, 
\begin{align*}
a_n(\Phi)^{-1} = \int_{\D_d} | z^n|^2 \Phi(|z|) dv_d(z) =  2d \int_0^1 r^{2d + 2|n|-1}  \Phi(r)dr   \int_{\Sph_d} |\zeta^n|^2    d\sigma_{\Sph_d}(\zeta).
\end{align*}
Replacing $\Phi$ by $B\Phi$, we obtain the corresponding formulas for  $K_{W_{B\Phi}}$ and for $a_n(B\Phi)$. In particular, we see that the ratio  $\frac{a_n(B\Phi)}{a_n(\Phi)}$ depends only on $|n|$: 
\[
R_{|n|}: = \frac{a_n(B\Phi)}{a_n(\Phi)} =   \frac{ \int_0^1 r^{2d + 2|n|-1}  W(r)dr}{ \int_0^1 r^{2d + 2|n|-1}  B(r) W(r)dr}.
\]   
Then for any $z\in \D_d$,  by writing $z = r\zeta$ with $r = |z|$ and $\zeta = z/r\in \Sph_d$, we have 
\begin{align*}
K_{W_\Phi}(z,z)& = \sum_{k=0}^\infty \Big(\sum_{n\in \N^d: |n| = k}  a_n(\Phi)   |z^n|^2\Big) =  \sum_{k=0}^\infty   \Big(\sum_{n\in \N^d: |n| = k}  a_n(\Phi)   |\zeta^n|^2\Big)  r^{2k},
\\
K_{W_{B\Phi}}(z,z)& = \sum_{k=0}^\infty \Big(\sum_{n\in \N^d: |n| = k}  a_n(B\Phi)   |z^n|^2\Big) =  \sum_{k=0}^\infty  R_k  \cdot \Big(\sum_{n\in \N^d: |n| = k}  a_n(\Phi)   |\zeta^n|^2\Big)  r^{2k}.
\end{align*}
Since $W_\Phi$ is radial, the transformation $f(\cdot) \mapsto f (U\cdot)$ is unitary on $A^2(\D_d, W_\Phi)$ for any $d\times d$ unitary matrix $U$. Hence by \eqref{K-W-diag-norm},   the function $z \mapsto K_{W_\Phi}(z,z)$ is radial and thus
\[
A_k(\Phi): = \sum_{n\in \N^d: |n| = k}  a_n(\Phi)   |\zeta^n|^2  \quad \text{does not depend on $\zeta \in \Sph_d$}.
\]
Now by  \eqref{lem-comp-coef}, we have  $\lim_{k\to\infty} R_{k} = 0$. Finally, by Lemma \ref{lem-com-series}, we obtain 
\begin{align*}
\limsup_{|z|\to 1^{-}} \frac{K_{W_{B\Phi}}(z,z)}{K_{W_\Phi}(z,z)} = \limsup_{r\to 1^{-}}   \frac{ \sum_{k=0}^\infty  R_k  A_k(\Phi)  r^{2k} }{ \sum_{k=0}^\infty   A_k(\Phi)   r^{2k} } = 0.
\end{align*}
This completes the proof of Lemma \ref{lem-comp-2-rad-w}.
\end{proof}

\begin{proof}[Proof of Proposition \ref{prop-w-cb}]
We can write $W_{\mathrm{cr}} = W_{\Phi}$ for a function $\Phi: [0, 1) \rightarrow \R_{+}$.  Let $W$ be a super-critical weight on $\D_d$. Set 
\[
B(r) :  = \inf_{\zeta \in \Sph_d} \frac{W(r \zeta)}{\Phi(r)}
\]
 for any $r\in [0,1)$.
By \eqref{def-ge-W}, 
$
\lim_{r\to 1^{-}} B(r)= \infty. 
$
Hence by Lemma \ref{lem-comp-2-rad-w},  we have 
\[
\lim_{|z|\to 1^{-}} \frac{K_{W_{B\Phi}} (z, z) }{K_{W_{\mathrm{cr}}} (z, z)}  = \lim_{|z|\to 1^{-}} \frac{K_{W_{B\Phi}} (z, z) }{K_{W_{\Phi}} (z, z)} = 0. 
\]
Since $
B(|z|) \Phi(|z|) \le W(z)$, by \eqref{K-W-diag-norm},  we have
$
K_W(z,z) \le K_{W_{B\Phi}}(z,z).
$
The desired estimate of $K_W(z,z)$ now follows from Lemma \ref{lem-critical-w}. 
\end{proof}

\section{Impossibility of simultaneous uniform interpolations}\label{sec-more-disk}

In this section,   we are going to prove Proposition~\ref{prop-critical-weight}, Theorem~\ref{prop-failure-intro}  and Proposition~\ref{prop-sharp} and their higher dimensional counterparts.    As in \S \ref{sec-PS-d}, the determinantal point  process $\PP_{K_{\D_d}}$ will be denoted simply by 
\[
\PP_d: = \PP_{K_{\D_d}}.
\]
In particular, for $d = 1$, we use the simplified notation 
\[
\PP: = \PP_1 =  \PP_{K_\D}.
\]

\subsection{The variances of some linear statistics}

  Recall the definition \eqref{def-crit-w-intro} of   $W_{\mathrm{cr}}$. Define a harmonic function  $F_{\mathrm{cr}}: \D\rightarrow A^2(\D, W_{\mathrm{cr}})$  by 
\[
F_{\mathrm{cr}}(w) : = K_{W_{\mathrm{cr}}}(\cdot, w). 
\]
By  Lemma~\ref{prop-L2-L1-sum} and Lemma~\ref{lem-critical-w},   for any $s>1, z\in \D$,  for $\PP$-almost every $X\in \Conf(\D)$,  we may define $g_X(s, z; F_{\mathrm{cr}})$ as in \eqref{def-g-F}. In particular, we will denote 
\[
M_X(s): = g_X(s, 0; F_{\mathrm{cr}}) =  \sum_{k= 0}^\infty  \sum_{x\in X\cap \A_k(0)}  e^{-s d_\D(x,0)}F_{\mathrm{cr}}(x). 
\]
Define a positive symmetric measure $dM_2$ on $\D^2 = \D\times \D$ by
\begin{align}\label{def-M-2}
dM_2(x,y): = |K_\D(x,y)|^2 dA(x)dA(y).
\end{align}
\begin{lemma}\label{lem-critical-var}
For any $s>1$, the variance of $M_X(s)$ is given by 
\begin{align}\label{var-f-cr}
\Var_\PP(M_X(s)) = \frac{1}{2} \int_{\D^2}|e^{-s d_\D(x,0)} - e^{-s d_\D(y, 0)}|^2   K_{W_{\mathrm{cr}}}(x,y) dM_2(x,y).
\end{align}
\end{lemma}

\begin{remark}
Although $K_{W_{\mathrm{cr}}}(x,y)$ is complex-valued, the integral in \eqref{var-f-cr} is positive. 
\end{remark}

Let $F_\D: \D\rightarrow A^2(\D)$ be the function defined by 
\[
F_\D(w) : = K_\D(\cdot, w) \in A^2(\D).
\]  For any bounded  compactly supported radial function $\RR: \D\rightarrow \R_{+}$ and any $X\in \Conf(\D)$, we define 
\begin{align}\label{def-F-D}
g_X^\RR(z; F_\D): = \sum_{x\in X} \RR(\varphi_z(x)) F_\D(x) \quad \text{where $\varphi_z(x) = \frac{z-x}{1 - \bar{z}x}$}.
\end{align}

\begin{lemma}\label{lem-var-RR}
For any bounded  compactly supported radial function $\RR: \D\rightarrow \R_{+}$, we have
\[
\Var_\PP(g_X^\RR(z; F_\D))  = \frac{1}{2} \int_{\D^2} |\RR(\varphi_z(x)) - \RR(\varphi_z(y))|^2   K_{\D}(x,y) dM_2(x,y).
\] 
\end{lemma}

Clearly, both Lemma \ref{lem-critical-var} and   Lemma \ref{lem-var-RR} will follow from  Lemma \ref{lem-var-sob} and  the following 
\begin{lemma}\label{lem-unif-var}
Let $W$ be a Bergman-admissible weight on $\D$ such that $\inf_{w\in \D} W(w)>0$. Let $\mathcal{P}:\D\rightarrow \R$ be a real-valued function such that 
\[
\int_{\D^2} \mathcal{P}(x)^2 |K_W(x,y)| dM_2(x,y)<\infty.
\]
Then, by writing $\| \cdot\|= \|\cdot\|_{A^2(\D, W)}$ and recalling the definition \eqref{def-F-W} of $F_W$, we have 
\[
\int_{\D^2} \| \mathcal{P}(x)F_W(x) - \mathcal{P}(y) F_W(y) \|^2  dM_2(x,y) = 
\int_{\D^2} |\mathcal{P}(x) - \mathcal{P}(y)|^2   K_{W}(x,y) dM_2(x,y).
\]
\end{lemma}

\begin{proof}
We first show that for any $z\in \D$, 
\begin{align}
\int_\D K_W(w, z) |K_\D(z,w)|^2 dA(w) = K_W(z, z) K_\D(z,z); \label{rep-KW-1}
\\  
\int_\D K_W(z, w) |K_\D(z,w)|^2 dA(w) = K_W(z, z) K_\D(z,z). \label{rep-KW-2}
\end{align}
Fix any $z\in \D$, define  $f_z(w) : =K_W(w, z)K_\D(w, z)$ for $w\in \D$. Then  $f_z\in A^2(\D)$ since it is  holomorphic and 
\begin{align*}
\int_\D |f_z(w)|^2 dA(w)& \le \sup_{w\in \D}\frac{|K_\D(w,z)|^2}{W(w)} \cdot \int_\D |K_W(w,z)|^2 W(w)dA(w)
\\
& = \sup_{w\in \D}\frac{|K_\D(w,z)|^2}{W(w)} \cdot K_W(z,z)<\infty.
\end{align*}
Therefore, we have 
$
f_z(z) = \int_\D K_\D(z, w) f_z(w) dA(w).
$
Using $f_z(z) = K_W(z,z)K_\D(z,z)$,  $K_W(z,w) f_z(w) = K_W(w,z) |K_\D(z,w)|^2$, we obtain the equality \eqref{rep-KW-1}. Taking complex congugate on both sides of \eqref{rep-KW-1}, we obtain \eqref{rep-KW-2}. 

Now define $H(x,y): = \|\mathcal{P}(x) F_{W} (x) -\mathcal{P}(y) F_{W} (y) \|^2$.  Note that we have
\[
H(x,y) =  \mathcal{P}(x)^2 K_W(x,x) + \mathcal{P}(y)^2 K_W(y,y) -   \mathcal{P}(x) \mathcal{P}(y)  K_W(x,y) -\mathcal{P}(x) \mathcal{P}(y)  K_W(y,x).
\]
By exchanging the integration variables $x,y$, we have 
\begin{align*}
\int_{\D^2}  \mathcal{P}(x) \mathcal{P}(y)  K_W(x,y) dM_2(x,y) &= \int_{\D^2}  \mathcal{P}(x) \mathcal{P}(y)  K_W(y,x) dM_2(x,y);
\\
\int_{\D^2} \mathcal{P}(x)^2 K_W(x,x) dM_2(x,y) & = \int_{\D^2} \mathcal{P}(y)^2 K_W(y,y) dM_2(x,y).
\end{align*}
Using  first the equality  \eqref{rep-KW-2} and then $\int_\D |K_\D(x,y)|^2 dA(y) = K_\D(x,x)$, we obtain
\begin{multline*}
 I(\mathcal{P}):  = \int_{\D^2}  \mathcal{P}(x)^2   K_W( x,y) dM_2(x,y) = \int_\D \mathcal{P}(x)^2 K_W(x,x) K_\D(x,x) dA(x) 
\\=  \int_{\D^2} \mathcal{P}(x)^2 K_W(x,x)dM_2(x,y).
\end{multline*}
Hence $I (\mathcal{P})\in \R$.  Thus,  by taking complex conjugate (using the equality $K_W(x,y) = \overline{K_W(y,x)}$) and then exchanging the integration variables $x, y$, we obtain 
\[
\int_{\D^2}  \mathcal{P}(x)^2   K_W( x,y) dM_2(x,y)   = \int_{\D^2}\mathcal{P}(y)^2 K_W(x,y)dM_2(x,y). 
\]
Combining all the above equalities, we obtain 
\begin{multline*}
\int_{\D^2}  H dM_2 
= 2 \int_{\D^2} \mathcal{P}(x)^2 K_W(x,x)dM_2(x,y)  - 2 \int_{\D^2} \mathcal{P}(x) \mathcal{P}(y)  K_W(x,y) dM_2(x,y)
\\
= \int_{\D^2}  \mathcal{P}(x)^2   K_W( x,y) dM_2(x,y)  + \int_{\D^2}  \mathcal{P}(y)^2   K_W( x,y) dM_2(x,y)  
\\
-2 \int_{\D^2} \mathcal{P}(x) \mathcal{P}(y)  K_W(x,y) dM_2(x,y)
\\
= \int_{\D^2} |\mathcal{P}(x) - \mathcal{P}(y)|^2 K_W(x,y) dM_2(x,y). 
\end{multline*}
This completes the proof of Lemma \ref{lem-unif-var}.
\end{proof}

\subsection{Proof of Proposition \ref{prop-critical-weight}}

The proof of Proposition \ref{prop-critical-weight} relies on the following

\begin{lemma}\label{lem-var-low}
 There exists  $C > 0$ such that  $\Var_{\PP}(M_X(s)) \ge  C\cdot (s -1)^{-2}$ for  $s\in (1, 2]$.
\end{lemma}

\begin{proof}[Proof of Proposition \ref{prop-critical-weight}]
Writing $\|\cdot\| = \|\cdot \|_{A^2(\D, W_{\mathrm{cr}})}$. By \eqref{av-pv-f}, we have  
\begin{multline*}
\E_\PP \Big( \sup_{f\in \mathcal{B}(  W_{\mathrm{cr}})} \Big|  \frac{g_X(s, 0; f)}{g_{\PP}(s)} - f(0)  \Big|^2 \Big) =  \E_\PP\Big( \frac{\|M_X(s) -  g_{\PP}(s)F_{\mathrm{cr}}(0)\|^2}{g_{\PP}(s)^2}\Big)  = \frac{\Var_\PP ( M_X(s))}{g_{\PP}(s)^2}.
\end{multline*}
  Therefore,   Proposition \ref{prop-critical-weight} follows from  Lemma \ref{lem-s-d} and  Lemma \ref{lem-var-low}.
\end{proof}

It remains to prove Lemma \ref{lem-var-low}.  Set 
\begin{align}\label{def-E-func}
E(t): =  \frac{1}{(1 - t)^4} \log \Big(\frac{2}{1 - t}\Big), \quad t\in (0,1). 
\end{align}

\begin{lemma}\label{lem-var-M}
There exists a constant $c>0$ such that for any $s>1$, we have 
\[
\Var_{\PP}(M_X(s)) \ge c  \int_\D\int_\D |e^{-s d_\D(0,x)} - e^{-s d_\D(0, y)}|^2    E(|xy|^2) dA(x) dA(y).  
\]
\end{lemma}

\begin{proof}
Set  $D(x,y): = K_{W_{\mathrm{cr}}}(x,y) |K_\D(x,y)|^2$.  
Since the function $x\mapsto e^{-s d_\D(0, x)}$ is radial, by Lemma \ref{lem-critical-var}, we have 
\begin{align*}
\Var_{\PP}(M_X(s)) & =  \frac{1}{2} \int_\D\int_\D |e^{-s d_\D(0,x)} - e^{-s d_\D(0, y)}|^2   D(x,y) dA(x) dA(y)
\\
& =  \frac{1}{2} \int_\D\int_\D |e^{-s d_\D(0,x)} - e^{-s d_\D(0, y)}|^2   \Big[ \underbrace{ \int_{0}^{2\pi }D(x e^{i\theta},y)  \frac{d\theta}{2\pi}}_{\text{denoted by $D^\sharp (x,y)$}}\Big] dA(x) dA(y).
\end{align*}
By \eqref{expansion-K-W-critical} and \eqref{expansion-K-W-critical-2}, we have 
\[
K_{W_{\mathrm{cr}}}(x, y) = \sum_{n = 0}^\infty a_n(W_{\mathrm{cr}}) x^n \bar{y}^n \quad \text{with $a_n(W_{\mathrm{cr}}) \ge c'\cdot \log (n +2)$ for  any $n\in \N$,}
\]
where $c'>0$ is a numerical constant.  Therefore, we have 
\begin{align*}
D^\sharp(x, y) &= \sum_{n= 0}^\infty a_n(W_{\mathrm{cr}}) \int_0^{2\pi} (x\bar{y} e^{i \theta})^n   \cdot  \Big|\frac{1}{(1 - x\bar{y} e^{i\theta})^2}\Big|^2  \frac{d\theta}{2 \pi}   
\\
&=    \sum_{n= 0}^\infty  a_n(W_{\mathrm{cr}})  |xy|^n     \Big[ \frac{(n+1) |xy|^n}{(1 - |xy|^2)^2} +\frac{2 |xy|^{n+2}}{(1 - |xy|^2)^3} \Big]
\\
&\ge  c'\sum_{n= 0}^\infty   \frac{(n+1)  \log (n+2)  |xy|^{2n}}{(1 - |xy|^2)^2}     + 2 c'  \sum_{n= 0}^\infty   \frac{ \log (n+2)  |xy|^{2n+2}}{(1 - |xy|^2)^3}.
\end{align*}
Then by Lemma \ref{lem-log-series}, there exists $c>0$ such that 
\[
D^\sharp(x,y) \ge \frac{c}{(1 - |xy|^2)^4} \log \Big(\frac{2}{1 - |xy|^2}\Big) = c \cdot E(|xy|^2).
\]
This completes the proof of the lemma. 
\end{proof}

\begin{proof}[Proof of Lemma \ref{lem-var-low}]
 By Lemma \ref{lem-log-series}, there exists $c> 0$ such that  
\begin{align}\label{low-E-es}
E(t)   = \frac{1}{(1 - t)^4} \log \Big(\frac{2}{1 - t}\Big)  \ge c \sum_{n= 0}^\infty n^3 \log (n+2) t^n \quad \text{for all $t\in(0,1)$.} 
\end{align}
Set $G_s(x,y): = |e^{-s d_\D(0,x)} - e^{-s d_\D(0, y)}|^2$.  By Lemma \ref{lem-var-M},
there exists $c'>0$ such that for any $s>1$, we have 
\begin{multline}\label{var-low-series}
\Var_{\PP}(M_X(s)) \ge c \int_\D\int_\D  G_s(x,y)    E(|xy|^2) dA(x) dA(y)  \ge 
\\
 \ge c' \sum_{n=0}^\infty n^3 \log (n+2)  \underbrace{\int_\D\int_\D G_s(x,y)  |xy|^{2n}  dA(x) dA(y)}_{\text{denoted by $I_n(s)$}}.
\end{multline}
Note that for any $n\in \N$, the integral  $I_n(s)$ is finite for any $s \ge 1$ and 
\[
I_n(s) = 2 V(n, s) - 2 U(n,s), 
\]
where $ U(n, s)$ and $ V(n, s)$ are defined for all $n\in \N$ and all $s\ge 1$ by 
\begin{align}\label{def-U-V}
\begin{split}
U(n,s) :&=  \int_\D \int_\D e^{- s d_\D(0, x)-s d_\D(0, y)} |xy|^{2n} dA(x) dA(y) = \Big(\int_0^1  \Big(\frac{1 - r}{1 + r}\Big)^s r^{2n+1}dr\Big)^2;
\\
V(n,s):  &= \int_\D \int_\D e^{-2 s d_\D(0, x)} |xy|^{2n} dA(x) dA(y) =  \frac{1}{2n+2}\int_0^1 \Big(\frac{1 - r}{1 + r}\Big)^{2s} r^{2n+1}dr.
\end{split}
\end{align}
{\flushleft \bf Claim A:} we have 
\begin{align}\label{def-gamma}
\gamma: = \sup_{n\in \N}\sup_{s\in [1,2]}\frac{U(n,s)}{V(n,s)}<1.
\end{align}

Let us complete the proof of the lemma by using Claim A.   By \eqref{def-gamma}, we have 
\[
I_n(s) =  2 V(n, s) - 2 U(n,s)  \ge 2 ( 1 - \gamma) V(n,s) \quad  \text{for any $n\in \N$ and any $s\in (1,2]$.}
\]
Thus by \eqref{var-low-series},  we have 
\[
\Var_{\PP}(M_X(s))  \ge 2 c' ( 1- \gamma)  \sum_{n= 0}^\infty n^3 \log (n + 2) V(n,s) =: 2c'(1 - \gamma) \Sigma_V(s). 
\]
Now by Lemma \ref{lem-log-series}, there exists a constant $c''>0$ such that 
\begin{multline*}
\Sigma_V(s) =  \sum_{n= 0}^\infty  \frac{n^3 \log (n + 2)}{2n+2}\int_0^1 \Big(\frac{1 - r}{1 + r}\Big)^{2s} r^{2n+1}dr  
 \ge c'' \int_0^1   \Big(\frac{1 - r}{1 + r}\Big)^{2s}  \cdot \frac{r \log (\frac{2}{1 -r^2})  }{(1 - r^2)^3}  dr =
\\
 =  \frac{c'' }{2} \int_0^1 \Big(\frac{1 - \sqrt{ t}}{1 + \sqrt{t}}\Big)^{2s}   \cdot \frac{\log (\frac{2}{1 -t})}{(1 - t)^3}   dt 
\ge \frac{c'' }{2^{1 + 4s}}\int_0^1 (1 - t)^{2s-3} \log\Big(\frac{2}{1-t}\Big)dt.
\end{multline*} 
By change of variable $t  = 1  - 2 e^{-x}$, we obtain, for any $s\in (1, 2]$, that 
\begin{multline*}
\int_0^1 (1 - t)^{2s-3} \log\Big(\frac{2}{1-t}\Big)dt  =  2^{2s-2}\int_{\log 2}^\infty  e^{-2(s-1)x} x dx \ge 
\\
\ge  \frac{1}{(s-1)^2}\int_{(s-1) \log 2}^\infty  e^{-2x} xdx \ge \frac{1}{(s-1)^2}\int_{\log 2}^\infty  e^{-2x} xdx.
\end{multline*}
Thus, there exists  $C>0$ such that 
$\Var_{\PP}(M_X(s)) \ge  C\cdot (s -1)^{-2}$  for any $s\in (1, 2]$.

It remains to prove Claim A. By Cauchy-Bunyakovsky-Schwarz inequality,  
$U(n,s) <V(n,s)$ for any $n\in \N$ and $s\ge 1$. For any $n\in \N$, since $\frac{U(n,s)}{V(n,s)}
$ is continuous on $s$,
 we have 
\begin{align}\label{def-gamma-n}
\gamma_n:  = \sup_{s\in [1, 2]} \frac{U(n, s)}{V(n, s)}<1.
\end{align}
Write $r_n: = 1 - \frac{1}{\sqrt{n+1}}$.  There exists  $c_1 > 0$  such that for any $s\in [1, 2]$ and any $n\in \N$,  
\begin{multline*}
 \int_0^{r_n}  \Big(\frac{1 - r}{1 + r}\Big)^s r^{2n+1}dr\le  \sup_{0\le r \le r_n}  \Big(\frac{1 - r}{1 + r}\Big) \cdot \int_0^{r_n} r^{2n+1} dr\le 
\\
 \le\frac{1}{\sqrt{n+1}}  \cdot \frac{(1 - \frac{1}{\sqrt{n+1}})^{2n+2}}{2n+2} \le  \frac{ e^{-c_1 \sqrt{n}}} {(n+1)^{3/2}}.
\end{multline*}
and 
\[
\int_0^1  \Big(\frac{1 - r}{1 + r}\Big)^s r^{2n+1}dr \ge   \frac{1}{4}\int_0^1 (1 - r)^2 r^{2n+1} dr
=  \frac{ \Gamma(2n+2)}{ \Gamma(2n+5)}   \ge \frac{1}{64 (n+1)^3}. 
\]
It follows that for any $s\in [1,2]$, we have 
\[
\frac{ \int_0^{r_n}  (\frac{1 - r}{1 + r})^s \cdot r^{2n+1}dr}{\int_0^1  (\frac{1 - r}{1 + r})^s \cdot r^{2n+1}dr} \le  \underbrace{ 64 (n+1)^{3/2} e^{-c_1 \sqrt{n}}}_{\text{denoted by $\alpha_n$}}
\]
and hence 
\[
\int_0^1  \Big(\frac{1 - r}{1 + r}\Big)^s r^{2n+1}dr \le \frac{1}{1- \alpha_n} \int_{r_n}^1  \Big(\frac{1 - r}{1 + r}\Big)^s r^{2n+1}dr.
\]
Thus for any $s\in [1,2]$ and any $n\in \N$, we have 
\begin{multline*}
\frac{U(n,s)}{V(n,s)} 
\le    \frac{  [ \frac{1}{1 -  \alpha_n} \int_{r_n}^1  (\frac{1 - r}{1 + r})^s \cdot r^{2n+1}dr ]^2 }{ \frac{1}{2n+2}\int_{0}^1 (\frac{1 - r}{1 + r})^{2s} \cdot  r^{2n+1}dr} 
\le    \frac{1}{(1 -  \alpha_n)^2} \frac{ \frac{1}{(1  +r_n)^{2s}} [  \int_{0}^1  (1 - r)^s r^{2n+1}dr]^2 }{ \frac{1}{2n+2}  \frac{1}{2^{2s}}\int_{0}^1 (1 - r)^{2s} r^{2n+1}dr} =
\\
=  \frac{2n +2}{(1 -  \alpha_n)^2}\Big( \frac{2}{1 + r_n}\Big)^{2s} \cdot  \frac{\Gamma(s+1)^2\Gamma(2n+2)^2}{\Gamma(2n+3+s)^2} \frac{\Gamma(2n+3 +2s)}{\Gamma(2s+1) \Gamma(2n+2)} = 
\\
 = \Big(\frac{2}{1+r_n}\Big)^{2s} \frac{1}{(1 - \alpha_n)^2} \frac{\Gamma(s+1)^2}{\Gamma(2s +1)} \frac{\Gamma(2n+3) \Gamma(2n+3+2s)}{\Gamma(2n+3+s)^2}.
\end{multline*}
By  Cauchy-Bunyakovsky-Schwarz inequality, we have 
\[
\Gamma(s +1)^2 = \Big(\int_0^\infty t^{s} e^{-t} dt\Big)^2 < \int_0^\infty t^{2s} e^{-t} dt \cdot \int_0^\infty e^{-t} dt= \Gamma(2s +1).
\]
Since the function 
$
[1, 2]\ni  s \mapsto \frac{\Gamma(s+1)^2}{\Gamma(2s+1)}
$
 is continuous, we have 
$
\sup\limits_{s\in [1,2]} \frac{\Gamma(s+1)^2}{\Gamma(2s+1)}<1. 
$
Note that the limit equality 
$
\lim\limits_{n\to\infty} \frac{\Gamma(n+t)}{n^t \Gamma(n)}  = 1
$
 holds uniformly for $t$ in a compact subset of $\R$, we have 
\[
\lim_{n\to\infty}\sup_{s\in [1,2]} \Big| \frac{\Gamma(2n+3) \Gamma(2n+3+2s)}{\Gamma(2n+3+s)^2} - 1\Big| = 0
\] 
and thus 
\begin{align}\label{lim-gamma-n}
\limsup_{n\to\infty} \gamma_n = \limsup_{n\to\infty} \sup_{s\in [1,2]}\frac{U(n,s)}{V(n, s)}  \le  \sup_{s\in [1,2]} \frac{\Gamma(s+1)^2}{\Gamma(2s+1)}<1.
\end{align}
Combining  \eqref{lim-gamma-n} and  \eqref{def-gamma-n}, we obtain the desired inequality \eqref{def-gamma}. 
\end{proof}

\subsection{Proof of Theorem \ref{prop-failure-intro}}

Recall the definition \eqref{def-F-D} of $g_X^{\RR}(z; F_\D)$.

\begin{proposition}\label{prop-new-var-f}
Let $\RR: \D \rightarrow \R_{+}$ be  bounded compactly supported and radial. Then  
\[
\Var_{\PP} ( g_X^{\RR}(z; F_\D))  =  \frac{1}{2 ( 1 - |z|^2)^2} \int_\D \int_\D  \frac{| \RR(x) - \RR(y)|^2}{(1 - |xy|^2)^5}  \cdot I_z(x, y) dA(x) dA(y)
\]
for any  $z\in \D$, where $I_z(x,y)$ is given by the formula:
\[
I_z(x, y) = 1 + (3+8|z|^2)|xy|^2 + (3|z|^4 + 8 |z|^2) |xy|^4 + |z|^4|xy|^6.
\]
\end{proposition}

\begin{proof}
 Recall the definition \eqref{def-M-2} of the measure $dM_2(x,y)$ on $\D^2$. We have  
\[
dM_2(x,y) = \Phi(x,y) d\mu_\D(x)d\mu_\D(y)   \quad \text{with}\quad \Phi(x,y)= \frac{|K_\D(x,y)|^2}{K_\D(x,x)K_\D(y,y)}. 
\]
Note that $\Phi(x,y)= \Phi(\varphi_z(x), \varphi_z(y))$. Indeed, this can be derived from  the identity (cf. Rudin \cite[Thm. 2.2.2]{Rudin-ball}):
\begin{align}\label{mob-id}
1 -  \varphi_{z}(x) \cdot \overline{\varphi_{z}(y)} = \frac{(1 - |z|^2) (1 -  x\cdot \bar{y})}{(1 - x \cdot \bar{z}) ( 1 - z \cdot \bar{y})} \quad  x, y, z \in \D.
\end{align}
Therefore, by Lemma \ref{lem-var-RR} and the conformal invariance of the measure $\mu_\D$, we have
\[
\Var_{\PP} ( g_X^{\RR}(z; F_\D)) = \frac{1}{2} \int_{\D^2} |\RR(x) - \RR(y)|^2   K_{\D}(\varphi_z(x),\varphi_z(y)) |K_\D(x,y)|^2 dA(x)dA(y).
\]
Now since $\RR$ is radial, we have 
\[
\Var_{\PP} ( g_X^{\RR}(z; F_\D)) = \frac{1}{2} \int_{\D^2} |\RR(x) - \RR(y)|^2  J_z(x,y) dA(x)dA(y),
\]
where $J_z(x,y)$ is defined by 
\[
J_z(x,y): = \int_0^{2\pi}\int_0^{2\pi} K_\D(\varphi_z(x e^{i\theta_1}), \varphi_z(y e^{-i\theta_2})) |K_\D(x e^{i\theta_1},y e^{-i\theta_2})|^2 \frac{d\theta_1}{2\pi} \frac{d\theta_2}{2\pi}.
\]
Using \eqref{mob-id}, we have 
\begin{multline*}
J_z(x,y) = \frac{1}{(1 - |z|^2)^2} \int_{0}^{2\pi}\int_{0}^{2\pi}  \frac{(1 - x e^{i\theta_1} \cdot \bar{z})^2 ( 1 - z \cdot \bar{y} e^{i\theta_2})^2}{(1 -  x \bar{y} e^{i\theta_1} e^{i\theta_2})^4 (1 - \bar{x}y e^{-i\theta_1} e^{-i\theta_2})^2}   \frac{d\theta_1 d\theta_2}{4\pi^2}
\\
= \frac{1}{(1-|z|^2)^2} \cdot \frac{1}{4\pi^2 } \oint_{C_1}\frac{d\zeta_1}{i \zeta_1}  \oint_{C_1} \frac{d\zeta_2}{i \zeta_2} \frac{(1 - x  \bar{z}\zeta_1)^2 ( 1 - z  \bar{y} \zeta_2)^2}{(1 -  x \bar{y} \zeta_1 \zeta_2)^4 (1 - \bar{x}y \zeta_1^{-1} \zeta_2^{-1})^2},
\end{multline*}
where $C_1$ the unit circle oriented counterclockwise. Using the residue method, we obtain 
\begin{multline*}
J_z(x,y) =  \frac{1}{(1 -|z|^2)^2}\Big[\frac{1 + 8 |z xy|^2 + 3 |z xy|^4}{ ( 1 - | xy|^2)^4} + \frac{4 |xy|^2 ( 1 + 4 |z xy|^2 + | zxy|^4)}{(1 - |x y|^2)^5}\Big]
\\
=  \frac{ 1 + (3+8|z|^2)|xy|^2 + (3|z|^4 + 8 |z|^2) |xy|^4 + |z|^4|xy|^6}{(1 -|z|^2)^2 \cdot(1 - |x y|^2)^5}.
\end{multline*}
This completes the proof of the proposition.
\end{proof}

\begin{proof}[Proof of Theorem \ref{prop-failure-intro}]
Fix any $z \in \D$. Note that 
\[
\E_{\PP} \Big[    \sup_{f\in A^2(\D): \|f\|\le 1} \Big| \frac{ g_X^\RR(z, f)}{  g^\RR_{\PP} } - f(z) \Big|^2\Big]=  \frac{\Var_\PP(g_X^\RR(z; F_\D))}{(g_\PP^\RR)^2}. 
\] Since $\RR: \D \rightarrow \R_{+}$ is radial and compactly supported, there exists $\Phi: [0, 1) \rightarrow \R_{+}$ with  $\supp(\Phi) \subset [0, 1 - \varepsilon]$ for some $\varepsilon > 0$ such that $\RR(z) = \Phi(|z|)$.   Set  
\begin{align}\label{def-V-g}
V(t): = \Phi( \sqrt{1-t}), \quad G(t):   = \frac{V(t)}{t^2} =  \frac{\Phi( \sqrt{1-t})}{t^{2}},\quad  \text{where $t\in (0, 1)$}.
\end{align} Note that  $ \supp(V) = \supp (G) \subset [\varepsilon', 1]$ for some $\varepsilon' \in (0,1)$. Then
\begin{multline*}
g_\PP^\RR =  \E_{\PP} ( g_X^\RR(z))   =   \E_{\PP} ( g_X^\RR(0))   =  \int_{\D} \frac{\Phi(|x|)}{ ( 1 - |x|^2)^{2}} dA(x)    = \int_0^1 \frac{\Phi( \sqrt{1-t})}{t^{2}}dt    = \int_0^1 G(t) dt. 
\end{multline*}
By Proposition \ref{prop-new-var-f}, we have 
\begin{multline*}
\Var_{\PP}  (g_X^\RR(z; F_\D))  \ge   \frac{1}{2} \int_\D \int_\D  \frac{| \RR(x) - \RR(y)|^2}{(1 - |xy|^2)^5}   dA(x) dA(y)
\\
 = 2 \int_{0}^1 \int_{0}^1  \frac{ | \Phi(r_1) - \Phi(r_2)|^2}{ (1 - r_1^2 r_2^2)^{5}} r_1r_2dr_1 dr_2 =  \frac{1}{2} \int_{0}^1 \int_{0}^1  \frac{  | \Phi(\sqrt{t_1}) - \Phi(\sqrt{t_2})|^2 }{ (1 - t_1 t_2)^{5}}dt_1 dt_2. 
\end{multline*}
Then by using the change of variables $t_1 = 1- s_1, t_2 = 1 -s_2$, we obtain 
\begin{multline*}
\Var_{\PP}  (g_X^\RR(z; F_\D))  \ge \frac{1}{2} \int_{0}^1 \int_{0}^1  \frac{  | V(s_1) - V(s_2)|^2 }{ (s_1 + s_2- s_1 s_2)^{5}}ds_1 ds_2 = \!\!\! \int\limits_{0\le s_1\le s_2 \le 1} \frac{  | V(s_1) - V(s_2)|^2 }{ (s_1 + s_2- s_1 s_2)^{5}}ds_1 ds_2.
\end{multline*}
Using change of variables $s_1 = \lambda t, s_2 = t$, we obtain 
\begin{multline*}
\Var_{\PP}  (g_X^\RR(z; F_\D))\ge \int\limits_{0\le s_1\le s_2 \le 1} \frac{  | V(s_1) - V(s_2)|^2 }{ (s_1 + s_2- s_1 s_2)^{5}}ds_1 ds_2 = \int_0^1\int_0^1 \frac{  | V(\lambda t) - V(t)|^2 }{ (\lambda t + t- \lambda t^2)^{5}}  t d\lambda dt \ge
\\
\ge  \int_0^1\int_0^1 \frac{  | V(\lambda t) - V(t)|^2 }{ (\lambda t + t)^{5}}  t d\lambda dt    = \int_0^1\int_0^1 \frac{  | \lambda^2  G(\lambda t) -  G(t)|^2 }{ (\lambda  + 1)^{5} }   d\lambda dt \ge 
\\
 \ge \frac{1}{32}\int_0^1\int_0^1 | \lambda^2  G(\lambda t) -  G(t)|^2   d\lambda dt.
\end{multline*}
Now note that 
\begin{align*}
\int_0^1\int_0^1 \left|  \lambda^2  G(\lambda t)  \right|^2 d\lambda dt  = \int_0^1 \lambda^{3} d\lambda \int_0^\lambda G(t')^2 dt' 
\le \frac{1}{4} \int_0^1 G(t)^2 dt. 
\end{align*}
Hence, using the triangle inequality on the space $L^2(d\lambda dt) = L^2([0,1]^2)$, we obtain 
\begin{multline}\label{sq-var-gR}
\sqrt{32 \cdot \Var_{\PP}  (g_X^\RR(z; F_\D))} \ge  \| \lambda^{2}  G(\lambda t) - G(t) \|_{L^2(d\lambda dt) } \ge
\\
 \ge  \|  G (t) \|_{L^2(d\lambda dt) } -  \| \lambda^{2}  G(\lambda t) \|_{L^2(d\lambda dt) }  \ge 
\\ 
\ge   \|  G(t) \|_{L^2(d\lambda dt) } -   \frac{\|  G(t) \|_{L^2(d\lambda dt) }}{2}   =  \frac{\|  G(t) \|_{L^2(d\lambda dt) }}{2} = \frac{\|  G(t) \|_{L^2(dt) }}{2}.
\end{multline}
Therefore, 
\[
 \frac{ \Var_{\PP}  (g_X^\RR(z; F_\D)) }{  (g_\PP^\RR)^2 }  \ge     \frac{\frac{1}{128} \| G(t) \|_{L^2( dt) }^2}{[\int_0^1 G(t) dt]^2} \ge \frac{1}{128}.  
\]
This completes the proof of Theorem \ref{prop-failure-intro}. 
\end{proof}

\subsection{Proof of Proposition \ref{prop-sharp}}
Fix an exponent $s\in (1, \frac{3}{2}]$.  For any  $N\in \N$, set 
\[
\Phi(r) = \Big(  \frac{1-r}{1+r}\Big)^s,\, \Phi_N(r) = \Phi(r)  \mathds{1}\Big(  \log \Big(\frac{1-r}{1+r} \Big) \le N+1\Big), \,
\RR_{N} (x) =  \Phi_N(|x|).
\]
Writing $\|\cdot\|= \|\cdot\|_{A^2(\D)}$,  by \eqref{sq-var-gR}, for any $N\in \N$ and any $z\in \D$, we have 
\begin{multline*}
\E_\PP\Big(\sup_{f\in A^2(\D): \|f\|\le 1} \Big|\sum_{k=0}^N g_X^{(k)}(s, z; f)\Big|^2\Big) = \E_\PP(\| g_X^{\RR_{N}} (z; F_\D)\|^2)  \ge 
\\
\ge \Var_{\PP} (   g_X^{\RR_{N}} (z; F_\D)) 
\ge  \frac{1}{128} \Big(\int_0^1  \frac{\Phi_N(\sqrt{1-t})^2}{t^4}dt\Big)^2.
\end{multline*}
Therefore, by using the assumption $1<s \le \frac{3}{2}$, we obtain  the claimed result:
\[
\sup_{N\in\N} \E_\PP\Big(\sup_{f\in A^2(\D): \|f\|\le 1} \Big|\sum_{k=0}^N g_X^{(k)}(s, z; f)\Big|^2\Big) 
\ge   \frac{1}{128} \Big(\int_0^1  \frac{\Phi(\sqrt{1-t})^2}{t^4}dt\Big)^2  =  \infty.
\]

\subsection{The case of dimension $d \ge 2$} In this subsection, we always assume that the dimension $d \ge 2$.  Let $F_d: \D_d\rightarrow A^2(\D_d)$ be the function defined by 
\[
F_d(w) : = K_{\D_d}(\cdot, w) \in A^2(\D_d).
\]  For any bounded  compactly supported radial function $\RR: \D_d\rightarrow \R_{+}$ and any configuration $X\in \Conf(\D_d)$, we define 
\[
g_X^\RR(z; F_d): = \sum_{x\in X} \RR(\varphi_z(x)) F_d(x), 
\]
where $\varphi_z(x)$ is defined as in \eqref{inv-auto}. 

\begin{theorem}\label{thm-fail-d}
For any integer $d \ge 2$, there exists a constant $c_d>0$, such that for any compactly supported radial weight $\RR: \D_d\rightarrow \R_{+}$ and   any  $z_o\in \D_d$, we have 
\begin{align}\label{low-univers-d}
\E_{\PP_d} \Big[    \sup_{f\in A^2(\D_d): \|f\|\le 1} \Big| \frac{ g_X^\RR(z_o, f)}{  g^\RR_{\PP_d} } - f(z_o) \Big|^2\Big] \ge \frac{c_d }{2^{3d}( 1 - |z_o|^2)^{d+1}}  \frac{1}{d^{2d+2}}\Big(1 - \frac{1}{\sqrt{2d + 2}} \Big)^2,
\end{align}
where 
\[
 g_{\PP_d}^\RR  =  \E_{\PP_{d}} \Big( \sum_{x\in X} \RR (\varphi_{z_o}(x))  \Big) = \E_{\PP_{d}} \Big( \sum_{x\in X} \RR (x)  \Big).
\]
\end{theorem}

 Proposition \ref{prop-new-var-d}  below will play the same role as that that of  Proposition \ref{prop-new-var-f} for  $d = 1$. 
\begin{proposition}\label{prop-new-var-d}
For any integer $d \ge 2$, there exists a constant $c_d>0$, such that for any compactly supported radial weight $\RR: \D_d\rightarrow \R_{+}$ and   any  $z_o\in \D_d$,
\begin{align}\label{I-d2}
\Var_{\PP_{d}} ( g_X^\RR(z_o; F_d)) \ge \frac{c_d}{( 1 - |z_o|^2)^{d+1}} \int_{\D_d} \int_{\D_d}  \frac{| \RR(z) - \RR(w)|^2}{(1 - |z|^2 |w|^2)^{3d +1}}   dv_d(z) dv_d(w).
\end{align} 
\end{proposition}

\begin{remark}\label{rem-1-d}
By taking $d = 1$ on the right hand side of the inequality \eqref{I-d2}, we obtain a term $(1 - |z|^2 |w|^2)^{-4}$, which is different to the correct order $(1 - |z|^2|w|^2)^{-5}$ obtained in Proposition \ref{prop-new-var-f} for the case of dimension $d= 1$. 
\end{remark}

\begin{lemma}\label{prop-Sob}
If $\mathcal{P}:\D_d \rightarrow \R$ is a compact supported real-valued measurable function, then
\[
\Var_{\PP_d} \Big[ \sum_{x\in X} \mathcal{P}(x) F_d(x)  \Big] = \frac{1}{2} \int_{\D_d} \int_{\D_d} | \mathcal{P}(z) - \mathcal{P}(w)|^2 \cdot   K_{\D_d}(z, w) |K_{\D_d}(z, w)|^2 dv_d(z) dv_d(w). 
\]
\end{lemma}

\begin{proof}
Lemma \ref{prop-Sob} is a consequence of   Lemma \ref{lem-var-sob} and the higher dimensional counterpart of Lemma \ref{lem-unif-var} (whose proof is exactly the same as Lemma \ref{lem-unif-var}). 
\end{proof}

\begin{proof}[Proof of Proposition \ref{prop-new-var-d}]
 Fix $z_o\in \D_d$. Using Lemma \ref{prop-Sob}, the  invariance of the measure 
\[
|K_{\D_d}(z, w)|^2 dv_d(z)dv_d(w)
\]
 under the diagonal action of $\Aut(\D_d)$ and the identity (cf. Rudin \cite[Theorem 2.2.2]{Rudin-ball}):
\[
1 -  \varphi_{z_o}(z) \cdot \overline{\varphi_{z_o}(w)} = \frac{(1 - |z_o|^2) (1 -  z \cdot \bar{w})}{(1 - z \cdot \bar{z_o}) ( 1 - z_o \cdot \bar{w})} \quad  z, w \in \D_d,
\]
 we have 
 \begin{multline*}
\Var_{\PP_{d}}  (g_X^\RR(z_o; F_d))  
=   \frac{1}{2} \int_{\D_d}\int_{\D_d}   |\RR(z) - \RR(w) |^2  K_{\D_d}(\varphi_{z_o}(z) , \varphi_{z_o}(w)) |    K_{\D_d}( z, w)|^2 dv_d(z) dv_d(w)
\\
 =  \frac{1}{2} \int_{\D_d}\int_{\D_d}    |\RR(z) - \RR(w) |^2   \underbrace{\left[\frac{(1 - z \cdot \bar{z_o}) ( 1 - z_o \cdot \bar{w})}{(1 - |z_o|^2) (1 -  z \cdot \bar{w})}  \right]^{d+1}  \frac{1}{| 1 - z \cdot \bar{w}|^{2d + 2}}}_{ \text{denoted $T(z, w; z_o)$}}dv_d(z) dv_d(w) . 
\end{multline*}
Denote $e_1 = (1, 0, \cdots, 0)\in \C^d$. By the rotational invariance of $\RR$ and $dv_d$, we have 
\begin{align}\label{var-non-sym}
\begin{split}
\Var_{\PP_{d}} (g_X^\RR(z_o; F_d)) &= \frac{1}{2} \int_{\D_d}\int_{\D_d} |\RR(z) - \RR(w)|^2 T(z, w; |z_o|e_1) dv_d(z) dv_d(w)
\\
& = \frac{1}{2} \int_{\D_d}\int_{\D_d} |\RR(z) - \RR(w)|^2  \widehat{T}(z, w; |z_o|) dv_d(z) dv_d(w)
\\
& = \frac{1}{2} \int_{\D_d}\int_{\D_d} |\RR(z) - \RR(w)|^2  \widetilde{T}(z, w; |z_o|) dv_d(z) dv_d(w),
\end{split}
\end{align}
where $\widehat{T}(z, w; |z_o|)$  and   $\widetilde{T}(z, w;  |z_o|)$  are  given  by 
\begin{align*}
\widehat{T}(z, w; |z_o|) : &= \int_0^{2\pi}T(e^{i\theta}z, e^{i\theta} w; |z_o|e_1) \frac{d \theta}{2 \pi};
\\
\widetilde{T}(z, w;  |z_o|) :& = \int_{\Sph_d}\int_{\Sph_d} \widehat{T}(|z| \cdot \zeta, |w| \cdot \xi; |z_o|) d\sigma_{\Sph_d}(\zeta) d\sigma_{\Sph_d}(\xi).
\end{align*}

Direct computation yields
\[
\widehat{T}(z, w; |z_o|) = \frac{K_{\D_d}(z,w) | K_{\D_d}(z, w)|^2}{(1 - |z_o|^2)^{d+1} }  \sum_{k=0}^{d+1} \binom{d+1}{k}^2 |z_o|^{2k} \langle z, e_1\rangle^k \overline{\langle w, e_1 \rangle}^k.
\]  
Note that by expanding the equality \eqref{berg-exp}, we have
\[
K_{\D_d}(z, w) = \sum_{k = 0}^\infty a_k (z\cdot \bar{w})^k, \quad a_k \ge 0. 
\]
Since the function $(\zeta, \xi) \mapsto \zeta \cdot \bar{\xi}$ is non-negative definite, by a classical result due to Schur (the pointwise  products of non-negative definite functions are still non-negative), the functions $(\zeta, \xi) \mapsto (\zeta \cdot \bar{\xi})^k$ for all $k \in \N$ are all non-negative definite. Thus  the  function  
\[
 (\zeta, \xi) \mapsto K_{\D_d}(|z|\zeta, |w| \xi) =  \sum_{k = 0}^\infty a_k |z| |w| (\zeta \cdot \bar{\xi})^k
\]
is non-negative definite and so is the function $(\zeta, \xi) \mapsto K_{\D_d}(|w|\xi, |z| \zeta)$.  Again by Schur's result on pointwise product of non-negative definite functions, the function \[
(\zeta, \xi) \mapsto F_{z, w}(\zeta, \xi) : =  \frac{K_{\D_d}(|z|\zeta, |w| \xi) \cdot  | K_{\D_d}(|z|\zeta, |w| \xi)|^2}{(1 - |z_o|^2)^{d+1}}
\] is non-negative definite.  Therefore, we have 
\begin{align*}
\widetilde{T}(z, w; |z_o|)  &  =  \int_{\Sph_d}\int_{\Sph_d}  F_{z, w}(\zeta, \xi)  \sum_{k=0}^{d+1} \binom{d+1}{k}^2 (|z_o|^{2} |z| |w|)^k \langle \zeta, e_1\rangle^k \overline{\langle \xi , e_1 \rangle}^k d\sigma_{\Sph_d}(\zeta) d\sigma_{\Sph_d}(\xi)
\\
& \ge   (d+1)^2 \int_{\Sph_d}\int_{\Sph_d}  F_{z, w}(\zeta, \xi)   d\sigma_{\Sph_d}(\zeta) d\sigma_{\Sph_d}(\xi)  
\\ 
& =  \frac{(d+1)^2}{(1 - |z_o|^2)^{d+1}} \int_{\Sph_d}\int_{\Sph_d}  \frac{1}{ ( 1 - |z||w| \zeta \cdot \bar{\xi})^{d+1}}  \frac{ d\sigma_{\Sph_d}(\zeta) d\sigma_{\Sph_d}(\xi)}{ | 1 - |z||w| \zeta \cdot \bar{\xi}|^{2d+2}}  . 
\end{align*}
By using  \cite[Lemma 1.9  \& formula (1.13)]{Zhu-UB}, we have 
\begin{multline*}
\widetilde{T}(z, w; |z_o|) \ge \frac{(d+1)^2 (d-1)}{(1 - |z_o|^2)^{d+1}} \int_{\D} \frac{1}{ ( 1 - |z||w| z')^{d+1}}  \frac{1}{ |( 1 - |z||w| z')|^{2d+2}}   \frac{dV(z')}{2\pi}  = 
\\
= \frac{(d+1)^2 (d-1)}{(1 - |z_o|^2)^{d+1}} \int_{0}^{2 \pi} \int_0^1 \frac{1}{ ( 1 - |z||w| r e^{i \theta})^{2d+2}}  \frac{1}{ ( 1 - |z||w| r e^{-i \theta})^{d+1}}   \frac{r dr d\theta}{2\pi} =
\\
= \frac{(d+1)^2 (d-1)}{(1 - |z_o|^2)^{d+1}} \int_0^1 r dr   \frac{1}{2\pi i}\oint_C\frac{\eta^d }{ ( 1 - |z||w| r \eta)^{2d+2}}  \frac{1}{ ( \eta - |z||w| r )^{d+1}}    d \eta  = 
\\
 = \frac{(d+1)^2 (d-1)}{(1 - |z_o|^2)^{d+1}} \int_0^1  \left[ \frac{1}{d!} \frac{\partial^d}{ \partial \eta^d}\Big|_{\eta  = |z||w| r}  \left( \frac{\eta^d }{ ( 1 - |z||w| r \eta)^{2d+2}} \right)  \right]    r dr.
\end{multline*}
Since for any integer $0\le k\le d$, 
\[
\frac{\partial^k}{ \partial \eta^k}\Big|_{\eta  = |z||w| r}  ( \eta^d ) \ge 0, \quad \frac{\partial^k}{ \partial \eta^k}\Big|_{\eta  = |z||w| r}  \left( \frac{1 }{ ( 1 - |z||w| r \eta)^{2d+2}} \right) \ge 0,
\]
we have 
\begin{multline*}
\frac{\partial^d}{ \partial \eta^d}\Big|_{\eta  = |z||w| r}  \left( \frac{\eta^d }{ ( 1 - |z||w| r \eta)^{2d+2}} \right) = 
\\
 =  \sum_{k = 0}^d \binom{d}{k}   \frac{\partial^{d-k}}{\partial \eta^{d-k}} (\eta^d) \frac{\partial^k}{ \partial \eta^k} \left( \frac{1}{ ( 1 - |z||w| r \eta)^{2d+2}} \right)\Big|_{\eta = |z||w|r} \ge 
\\ 
\ge \left[     \frac{\partial^d }{\partial \eta^d} (\eta^d) \frac{1}{(1 - |z||w|r)^{2d +2}}   +  \eta^d \frac{\partial^d}{ \partial \eta^d} \left( \frac{1}{ ( 1 - |z||w| r \eta)^{2d+2}} \right)\right]\Big|_{\eta = |z||w|r}  = 
\\
=     \frac{ d ! (1 - |z|^2 |w|^2 r^2)^d +  \frac{(2d +1+d)!}{(2d+1)!} (|z|^2|w|^2r^2)^{d} }{ (1 - |z|^2 |w|^2 r^2)^{3d +2}} \ge \frac{c_d'}{(1 - |z|^2 |w|^2 r^2)^{3d +2}},
\end{multline*}
where 
\[
c_d': =  \min_{x \in [0, 1]}  \left[ d ! (1 - x)^d +  \frac{(2d +1+d)!}{(2d+1)!} x^{d}\right] > 0.
\]
Therefore, there exist constants $c_d''>0, c_d'''>0$ depending only on $d$, such that 
\begin{multline*}
\widetilde{T}(z, w; |z_o|) \ge  \frac{c_d''}{(1 - |z_o|^2)^{d+1}} \int_0^1 \frac{1}{(1 - |z|^2 |w|^2 r^2)^{3d +2}} r dr \ge 
\\ 
\ge \frac{c_d'''}{(1 - |z_o|^2)^{d+1}}  \frac{1}{(1 - |z|^2 |w|^2)^{3d  +1}}  \left[ \frac{1 - (1 -|z|^2|w|^2)^{3d + 1}}{|z|^2 |w|^2} \right].
\end{multline*}
By taking 
\[
c_d : = c_d''' \cdot \min_{x \in (0, 1]} \left[ \frac{1 - (1 -x )^{3d + 1}}{x} \right] >0,
\]
we obtain 
\[
\widetilde{T}(z, w; |z_o|) 
\ge \frac{c_d}{(1 - |z_o|^2)^{d+1}}  \frac{1}{(1 - |z|^2 |w|^2)^{3d  +1}}.
\]
Finally, by substituting the lower bound obtained above for $\widetilde{T}(z, w; |z_o|)$ into the equality \eqref{var-non-sym}, we complete the proof of Proposition \ref{prop-new-var-d}. 
\end{proof}

\begin{proof}[Proof of Theorem \ref{thm-fail-d}]
Fix $z_o \in \D_d$. Since $\RR: \D_d \rightarrow \R_{+}$ is radial and compactly supported, there exists a function $\Phi: [0, 1) \rightarrow \R_{+}$ whose support is contained in $[0, 1 - \varepsilon]$ for some $\varepsilon > 0$ such that $\RR(z) = \Phi(|z|)$.  Set  
\[
V(t): = \Phi((1-t)^{1/2d}), \quad G(t):   = \frac{V(t)}{t^{d+1}},\quad  \text{where $t\in (0, 1)$}.
\] Then there exists $\varepsilon'> 0$ such that  $\supp (G) \subset [\varepsilon', 1]$.  By Proposition \ref{prop-new-var-d}, we have 
\begin{multline*}
\Var_{\PP_{d}}(g_X^\RR(z_o; F_d)) \ge \frac{c_d}{( 1 - |z_o|^2)^{d+1}} \int_{\D_d} \int_{\D_d}  \frac{ | \RR(z) - \RR(w)|^2}{ (1 - |z|^2 |w|^2)^{3d + 1}} dv_d(z) dv_d(w) = 
\\
 = \frac{c_d \cdot (2d)^2}{( 1 - |z_o|^2)^{d+1}} \int_{0}^1 \int_{0}^1  \frac{ | \Phi(r_1) - \Phi(r_2)|^2}{ (1 - r_1^2 r_2^2)^{3d + 1}} (r_1r_2)^{2d-1}dr_1 dr_2 = 
\\
 =\frac{c_d }{( 1 - |z_o|^2)^{d+1}} \int_{0}^1 \int_{0}^1  \frac{ | \Phi(t_1^{1/2d}) - \Phi(t_2^{1/2d})|^2}{ (1 - (t_1t_2)^{1/d})^{3d + 1}} dt_1 dt_2  . 
\end{multline*}
Since $(t_1 t_2)^{1/d} \ge t_1t_2$ whenever $0 \le t_1t_2 \le 1$, we have 
\begin{multline*}
\Var_{\PP_{d}} (g_X^\RR(z_o; F_d))  \ge
\frac{c_d }{( 1 - |z_o|^2)^{d+1}} \int_{0}^1 \int_{0}^1  \frac{ | \Phi(t_1^{1/2d}) - \Phi(t_2^{1/2d})|^2}{ (1 - t_1t_2)^{3d + 1}} dt_1 dt_2 = 
\\
 = \frac{c_d }{( 1 - |z_o|^2)^{d+1}} \int_{0}^1 \int_{0}^1  \frac{ | V(t_1) - V(t_2)|^2}{ (t_1 + t_2 - t_1t_2)^{3d + 1}} dt_1 dt_2  \ge 
\\
\ge \frac{c_d }{( 1 - |z_o|^2)^{d+1}} \int_{0}^1 \int_{0}^1  \frac{ | V(t_1) - V(t_2)|^2}{ (t_1 + t_2 )^{3d + 1}} dt_1 dt_2   = 
\\
 =   \frac{2c_d }{( 1 - |z_o|^2)^{d+1}} \int_{0\le t_1 \le t_2 \le 1}  \frac{ | V(t_1) - V(t_2)|^2}{ (t_1 + t_2)^{3d + 1}} dt_1 dt_2. 
\end{multline*}
Using change of variables $t_1 = \lambda t, t_2 = t$ in the last integral, we obtain 
\begin{multline*}
\Var_{\PP_{d}} (g_X^\RR(z_o; F_d))\ge
   \frac{2c_d }{( 1 - |z_o|^2)^{d+1}} \int_{0}^1\int_0^1  \frac{ | V(\lambda t) - V(t)|^2}{  (\lambda + 1)^{3d + 1}t^{3d}} d\lambda dt\ge 
\\
\ge \frac{c_d }{2^{3d}( 1 - |z_o|^2)^{d+1}} \int_{0}^1\int_0^1  \frac{ | V(\lambda t) - V(t)|^2}{  t^{3d}} d\lambda dt  = 
\\
=  \frac{c_d }{2^{3d}( 1 - |z_o|^2)^{d+1}} \int_{0}^1\int_0^1  \frac{ | \lambda^{d+1}G(\lambda t) - G(t)|^2}{  t^{d -2}} d\lambda dt 
\\
 \ge
   \frac{c_d }{2^{3d}( 1 - |z_o|^2)^{d+1}} \int_{0}^1\int_0^1  | \lambda^{d+1}G(\lambda t) - G(t)|^2 d\lambda dt. 
\end{multline*}
Now note that 
\begin{align*}
\int_0^1\int_0^1 \left|  \lambda^{d +1}  G(\lambda t)  \right|^2 d\lambda dt = \int_0^1 \lambda^{2d +1} d\lambda \int_0^\lambda G(t')^2 dt' \le \frac{1}{2d+2} \int_0^1 G(t)^2 dt. 
\end{align*}
The above inequality combined with the triangle inequality and then Cauchy-Bunyakovsky-Schwarz inequality yields
\begin{align*}
& \left(\int_0^1\int_0^1 \left|  \lambda^{d+1}  G(\lambda t) - G(t) \right|^2 d\lambda dt\right)^{1/2} 
\\
 \ge&   \left(\int_0^1\int_0^1 \left|  G(t) \right|^2 d\lambda dt\right)^{1/2} - \left(\int_0^1\int_0^1 \left|  \lambda^{d+1}  G(\lambda t)\right|^2 d\lambda dt\right)^{1/2}
\\
 \ge &   \Big(1 - \frac{1}{\sqrt{2d + 2}} \Big) \left(\int_0^1 G(t)^2 dt\right)^{1/2} \ge \Big(1 - \frac{1}{\sqrt{2d + 2}} \Big) \int_0^1 G(t) dt.
\end{align*}
 Using the Bernoulli's inequality  $(1 - t)^{1/d} \le 1 - t/d$ for any $t \le    1$,   we obtain 
\begin{multline*}
 g_{\PP_d}^\RR  =  \E_{\PP_{d}} \Big( \sum_{x\in X} \RR (x)  \Big) = \int_{\D_d} \frac{\Phi(|z|)}{ ( 1 - |z|^2)^{d+1}} dv_d(z)  = 
\\
 = 2d \int_0^1 \frac{\Phi(r)}{(1-r^2)^{d+1}} r^{2d-1}dr =  \int_0^1 \frac{\Phi(t^{1/2d})}{(1 - t^{1/d})^{d+1}} dt   = 
\\
  =\int_0^1 \frac{\Phi((1-t)^{1/2d})}{(1- (1 - t)^{1/d})^{d+1}}dx \le  d ^{d+1}\int_0^1 \frac{\Phi((1-t)^{1/2d})}{t^{d+1}}dx    =  d^{d+1}\int_0^1 G(t) dt. 
\end{multline*}
Finally, by noting that 
\[
\E_{\PP_d} \Big[    \sup_{f\in A^2(\D_d): \|f\|\le 1} \Big| \frac{ g_X^\RR(z_o, f)}{  g^\RR_{\PP_d} } - f(z_o) \Big|^2\Big]=  \frac{\displaystyle \Var_{\PP_{d}} (g_X^\RR(z_o; F_d))   }{ [g_{\PP_d}^\RR ]^2 },
\]
we complete the whole proof of Theorem \ref{thm-fail-d}.
\end{proof}


\def\cprime{$'$} \def\cydot{\leavevmode\raise.4ex\hbox{.}}

\end{document}